\documentclass[a4paper, 11pt]{article}
\usepackage{amsmath,amssymb,amsthm,amscd}
\usepackage{cleveref}
\usepackage{cases}
\usepackage{mathrsfs}
\usepackage{amsfonts}
\allowdisplaybreaks[4]
\usepackage[all]{xy}
\usepackage{comment}
\usepackage[pdftex]{graphicx}
\usepackage[top=30truemm,bottom=30truemm,left=15truemm,right=15truemm]{geometry}
\usepackage{tikz}
\usepackage{mathtools}
\usepackage{here}
\usepackage{stmaryrd}

\usetikzlibrary{shapes,arrows,positioning}
\usetikzlibrary{shapes.misc}
\usetikzlibrary{decorations.markings}
\tikzset{cross/.style={cross out, draw=black, minimum size=2*(#1-\pgflinewidth), inner sep=0pt, outer sep=0pt},
cross/.default={1pt}}
\tikzset{->-/.style={decoration={
  markings,
  mark=at position #1 with {\arrow[scale=1.5]{>}}},postaction={decorate}}}

\tikzset{-<-/.style={decoration={
  markings,
  mark=at position #1 with {\arrow[scale=1.5]{<}}},postaction={decorate}}}

\tikzset{midstealth/.style={decoration={
  markings,
  mark=at position #1 with {\arrow{stealth}}},postaction={decorate}}}
\newtheoremstyle{break}
  {}%
  {}%
  {\itshape}
  {}%
  {\bfseries}
  {.}%
  {\newline}%
  {}%

 \makeatletter
    
    \@addtoreset{equation}{section}
  \makeatother

\theoremstyle{plain}
\newtheorem{thm}{Theorem}[section]
\newtheorem{lem}[thm]{Lemma}
\newtheorem{prop}[thm]{Proposition}
\newtheorem{cor}[thm]{Corollary}
\newtheorem{exa}[thm]{Example}
\newtheorem{defn}[thm]{Definition}
\newtheorem{rem}[thm]{Remark}

\theoremstyle{break}

\DeclareMathOperator{\rank}{rank}

\DeclareMathOperator{\vol}{vol}

\DeclareMathOperator{\im}{Im}

\DeclareMathOperator{\Ker}{Ker}
\DeclareMathOperator{\Hom}{Hom}
\DeclareMathOperator{\End}{End}

\DeclareMathOperator{\Homo}{H}

\DeclareMathOperator{\diag}{diag}
\DeclareMathOperator{\DR}{DR}
\DeclareMathOperator{\FL}{FL}
\DeclareMathOperator{\id}{id}
\DeclareMathOperator{\cone}{cone}
\newcommand{\ve}{\varepsilon}
\newcommand{\barsigma}{\overline{\sigma}}
\newcommand{\PP}{\mathbb{P}}
\newcommand{\R}{\mathbb{R}}
\newcommand{\D}{\mathbb{D}}
\newcommand{\C}{\mathbb{C}}
\newcommand{\DD}{\mathcal{D}}
\newcommand{\LL}{\mathbb{L}}
\newcommand{\A}{\mathbb{A}}
\newcommand{\Q}{\mathbb{Q}}
\newcommand{\Z}{\mathbb{Z}}

\newcommand{\lef }{\left\{ }
\newcommand{\righ }{\right\} }
\newcommand{\Gm}{\mathbb{G}_m}

\newcommand{\ii}{\sqrt{-1}}
\newcommand{\s}{\sigma}

\newcommand{\bs}{\barsigma}

\DeclareMathOperator{\HD}{\mathbb{D}}
\DeclareMathOperator{\spanning}{span}

\newcommand*{\transp}[2][-3mu]{\ensuremath{\mskip1mu\prescript{\smash{\mathrm t\mkern#1}}{}{\mathstrut#2}}}

\crefname{thm}{Theorem}{Theorems}
\crefname{bthm}{Theorem}{Theorems}
\crefname{bprop}{Proposition}{Propositions}
\crefname{blem}{Lemma}{Lemmata}
\crefname{prop}{Proposition}{Propositions}
\crefname{lem}{Lemma}{Lemmata}
\crefname{bcor}{Corollary}{Corollary}
\crefname{cor}{Corollary}{Corollary}
\crefname{defn}{Definition}{Definitions}
\crefname{rem}{Remark}{Remarks}
\crefname{exa}{Example}{Exmples}

\title{Euler and Laplace integral representations of GKZ hypergeometric functions}
\author{Saiei-Jaeyeong Matsubara-Heo\footnote{ Graduate School of Science, Kobe  University, 1-1 Rokkodai, Nada-ku, Kobe 657-8501, Japan.\newline\indent e-mail: \texttt{saiei@math.kobe-u.ac.jp}}}

\begin{document}

\date{}
\maketitle

\begin{abstract}
\noindent
We introduce an interpolation between Euler integral and Laplace integral: Euler-Laplace integral. We establish a combinatorial method of constructing a basis of the rapid decay homology group associated to Euler-Laplace integral  with a nice intersection property. This construction yields a remarkable expansion formula of cohomology intersection numbers in terms of GKZ hypergeometric series. As an application, we obtain closed formulas of the quadratic relations of Aomoto-Gelfand hypergeometric functions and their confluent analogue in terms of bipartite graphs.
\end{abstract}

\section{Introduction}
In this paper, we focus on the integral of the following form:
\begin{equation}\label{MixedIntegral}
f_{\Gamma}(z)=\int_\Gamma e^{h_{0}(x)}h_{1}(x)^{-\gamma_1}\cdots h_{k}(x)^{-\gamma_k}x^{c}\omega,
\end{equation}
where $h_l(x;z)=h_{l,z^{(l)}}(x)=\sum_{j=1}^{N_l}z_j^{(l)}x^{{\bf a}^{(l)}(j)}$ ($l=0,\dots,k$) are Laurent polynomials in torus variables $x=(x_1,\dots,x_n)$, $\gamma_l\in\C$ and $c={}^t(c_1,\dots,c_n)\in\C^{n\times 1}$ are parameters, $x^c=x_1^{c_1}\dots x_n^{c_n}$, $\Gamma$ is a suitable integration cycle, and $\omega$ is an algebraic differential $n$-form in variables $x$ which has poles along $D=\{ x\in\C^n\mid x_1\dots x_n h_1(x)\dots h_k(x)=0\}$. As a function of the independent variable $z=(z_j^{(l)})_{j,l}$, the integral (\ref{MixedIntegral}) defines a hypergeometric function. We call the integral (\ref{MixedIntegral}) the Euler-Laplace integral representation. 

We can naturally define the twisted de Rham cohomology group associated to the Euler-Laplace integral (\ref{MixedIntegral}). We set $N=N_0+\dots+N_k,$ $\Gm^n={\rm Specm}\;\C[x_1^\pm,\dots,x_n^\pm]$, and $\A^N={\rm Specm}\;\C[z_j^{(l)}]$. For any $z\in\A^N$, we can define an integrable connection $\nabla_z:\mathcal{O}_{\Gm^n}(*D)\rightarrow\Omega^1_{\Gm^n}(*D)$ by $\nabla_z=d_x+d_xh_0\wedge-\sum_{l=1}^k\gamma_l\frac{d_xh_l}{h_l}\wedge+\sum_{i=1}^nc_i\frac{dx_i}{x_i}\wedge$. Setting $U=\Gm^n\setminus D$, the algebraic de Rham cohomology group $\Homo_{\rm dR}^*\left( U;(\mathcal{O}_U,\nabla_z)\right)$ is defined as the hypercohomology group $\mathbb{H}^*\left( \Gm^n;(\cdots\overset{\nabla_z}{\rightarrow}\Omega_{\Gm^n}^\bullet(*D)\overset{\nabla_z}{\rightarrow}\cdots)\right)$. Under a genericity assumption on the parameters $\gamma_l$ and $c$, we have the vanishing result $\Homo_{\rm dR}^m\left( U;(\mathcal{O}_U,\nabla_z)\right)=0$ $(m\neq n)$. Moreover, we can define a perfect pairing $\langle\bullet,\bullet\rangle_{ch}:\Homo_{\rm dR}^n\left( U;(\mathcal{O}_U,\nabla_z)\right)\times \Homo_{\rm dR}^{n}\left( U;(\mathcal{O}_U^\vee,\nabla_z^\vee)\right)\rightarrow\C$ which is called the cohomology intersection form (see (\ref{RDc.i.n.}) and the proof of \cref{thm:QuadraticRelation}). The main result of this paper (\cref{thm:QuadraticRelation}) is on the explicit formula of the cohomology intersection number.

In order to extract information of the cohomology intersection number from (\ref{MixedIntegral}), it is important to observe that (\ref{MixedIntegral}) satisfies a special holonomic system called GKZ system (or $A$-hypergeometric system) introduced by I.M.Gelfand, M.I.Graev, M.M.Kapranov, and A.V.Zelevinsky (\cite{GelfandGraevZelevinsky}, \cite{GKZToral}). Let us recall the definition of GKZ system. The system is determined by two inputs: an $d\times n$ ($d<n$) integer matrix $A=({\bf a}(1)|\cdots|{\bf a}(n))$ and a parameter vector $\delta\in\C^{d}$. GKZ system $M_A(\delta)$ is defined by
\begin{subnumcases}{M_A(\delta):}
E_i\cdot f(z)=0 &($i=1,\dots, d$)\label{EulerEq}\\
\Box_u\cdot f(z)\hspace{-0.8mm}=0& $\left(u={}^t(u_1,\dots,u_{n})\in L_A=\Ker(A\times:\Z^{n\times 1}\rightarrow\Z^{d\times 1})\right)$,\label{ultrahyperbolic}
\end{subnumcases}
where $E_i$ and $\Box_u$ are differential operators defined by 

\begin{equation}\label{HGOperators}
E_i=\sum_{j=1}^{n}a_{ij}z_j\frac{\partial}{\partial z_j}+\delta_i,\;\;\;
\Box_u=\prod_{u_j>0}\left(\frac{\partial}{\partial z_j}\right)^{u_j}-\prod_{u_j<0}\left(\frac{\partial}{\partial z_j}\right)^{-u_j}.
\end{equation}

\noindent
Throughout this paper, we assume an additional condition that the column vectors of $A$ generate the ambient lattice $\Z^{d\times 1}$. Writing $\DD_{\A^n}$ for the Weyl algebra on $\A^n$ and by $H_A(\delta)$ the left ideal of $\DD_{\A^n}$ generated by operators (\ref{HGOperators}), we also call the left $\DD_{\A^n}$-module $M_A(\delta)=\DD_{\A^n}/H_A(\delta)$ GKZ system. The fundamental property of GKZ system $M_A(\delta)$ is that it is always holonomic (\cite{Adolphson}), which implies that the stalk of the sheaf of holomorphic solutions at a generic point is finite dimensional. 

We set $\delta=(\gamma_1,\dots,\gamma_k,c)$, and consider an $(n+k)\times N$ matrix
\begin{equation}\label{IntroCayley}
A
=
\left(
\begin{array}{ccc|ccc|ccc|c|ccc}
0&\cdots&0&1&\cdots&1&0&\cdots&0&\cdots&0&\cdots&0\\
\hline
0&\cdots&0&0&\cdots&0&1&\cdots&1&\cdots&0&\cdots&0\\
\hline
&\vdots& & &\vdots& & &\vdots& &\ddots& &\vdots& \\
\hline
0&\cdots&0&0&\cdots&0&0&\cdots&0&\cdots&1&\cdots&1\\
\hline
&A_0& & &A_1& & &A_2& &\cdots & &A_k& 
\end{array}
\right),
\end{equation}

\noindent
where $A_l=\left({\bf a}^{(l)}(1)|\dots|{\bf a}^{(l)}(N_l)\right)$. The matrix (\ref{IntroCayley}) is a variant of the Cayley configuration (\cite{GKZEuler}). It is proved in \cref{thm:mainDresult} that (\ref{MixedIntegral}) is a solution of $M_A(\delta)$. 

An important combinatorial structure of GKZ system is the secondary fan (\cite[Chapter 7]{GKZbook}, \cite{DeLoeraRambauSantos}). If we denote by $L_A^\vee$ the dual lattice of $L_A$, the secondary fan ${\rm Fan}(A)$ is a special (possibly incomplete) fan in $L_A^\vee\otimes_{\Z} \R$. Moreover, each cone of the secondary fan has a combinatorial interpretation as a convex polyhedral subdivision of $\R_{>0}A$, the positive cone spanned by the column vectors of the matrix $A$. Any  triangulation $T$ of $\R_{>0}A$ arising in this way is called a regular triangulation. It was an important discovery of Gelfand-Kapranov-Zelevinsky (\cite{GKZToral}) which was later sophisticated by M.-C. Fern\'andez-Fern\'andez (\cite{FernandezFernandez}) that $T$ can be interpreted as a set of independent solutions of the GKZ system $M_A(\delta)$. Namely, for each simplex $\s\in T$, we can associate a finite Abelian group $G_{\s}$ and hypergeoemtric series $\{ \varphi_{\s,g}(z;\delta)\}_{g\in G_{\s}}$ which are solutions of $M_A(\delta)$. Though these series $\varphi_{\s,g}(z;\delta)$ may diverge, there is at least one regular triangulation $T$ such that $\varphi_{\s,g}(z;\delta)$ is convergent for any simplex $\s\in T$ and $g\in G_{\s}$. We say  such a regular triangulation $T$ is convergent in this paper. Then, it is known that the set $\Phi_T=\bigcup_{\s\in T}\{\varphi_{\s,g}(z;\delta)\}_{g\in G_{\s}}$ is a basis of solutions of $M_A(\delta)$ when the parameter $\delta$ is generic. Geometrically speaking, the secondary fan ${\rm Fan}(A)$ defines a toric variety $X({\rm Fan}(A))$ which contains the torus ${\rm Specm}\C [L_A]$. Any series $\varphi_{\s,g}(z;\delta)$ can be interpreted as a local solution near the torus fixed point $z_T^\infty$ of the toric variety $X({\rm Fan}(A))$ corresponding to $T$.


In the main statement of this paper, we focus on a particular class of the integral (\ref{MixedIntegral}) that the corresponding GKZ system $M_A(\delta)$ admits a unimodular triangulation. A regular triangulation $T$ is said to be unimodular if for any simplex $\s\in T$, the corresponding Abelian group $G_{\s}$ is trivial. The study of the matrix $A$ admitting a (special) unimodular triangulation is an active area of research in combinatorics (\cite{SantosZiegler}, \cite{HibiOhsugiRootConfiguration}, \cite{HibiOhsugiCentrallySymmetric}). In relation to special functions, the most famous example is Aomoto-Gelfand system (\cite{AomotoKita}, \cite{GelfandGeneralTheory}). In our notation, this corresponds to the case that $h_0\equiv 0$ and other Laurent polynomials $h_l$ are all linear polynomials. Another important class is the configuration matrix $A$ associated to a reflexive polytope (\cite{Batyrev}, \cite{StienstraMirror}, \cite{NarumiyaShiga}). This class of GKZ system has been studied in the context of toric mirror symmetry. These classes define regular holonomic GKZ systems (\cite{Hotta}), and have Euler integral representations, i.e., we can take $h_0(x)\equiv 0$. However, they have other integral representations with non-zero exponential factor $h_0(x)$. Moreover, there are various examples of (\ref{MixedIntegral}) admitting a unimodular triangulation and satisfying an irregular GKZ system. An important class is a ``confluence'' of Aomoto-Gelfand system (\cite{KimuraHaraokaTakano}, \cite{MarkovTarasovVarchenko}) which is discussed in \S\ref{E(21...1)}. The main result of this paper is an explicit formula of the cohomology intersection number in terms
of series solutions associated to a convergent unimodular regular triangulation.



\begin{thm}\label{thm:IntroTheorem}
Suppose that four vectors ${\bf a},{\bf a}^\prime\in\Z^{n\times 1},{\bf b},{\bf b}^\prime\in\Z^{k\times 1}$ and a convergent unimodular regular triangulation T are given. Under the assumption on the parameters $\gamma_l$ and $c$ of \cref{thm:QuadraticRelation}, one has an identity
\begin{align}
&(-1)^{|{\bf b}|+|{\bf b}^\prime|}\gamma_1\cdots\gamma_k(\gamma-{\bf b})_{\bf b}(-\gamma-{\bf b}^\prime)_{{\bf b}^\prime}
\sum_{\s\in T}\frac{\pi^{n+k}}{\sin\pi A_\s^{-1}\delta}\varphi_{\s,0}\left(z;
\begin{pmatrix}
\gamma-{\bf b}\\
c+{\bf a}
\end{pmatrix}
\right)\varphi_{\s,0}^\vee\left(z;
\begin{pmatrix}
\gamma+{\bf b}^\prime\\
c-{\bf a}^\prime
\end{pmatrix}
\right)\nonumber\\
=&\frac{\langle x^{\bf a}h^{\bf b}\frac{dx}{x},x^{{\bf a}^\prime}h^{{\bf b}^\prime}\frac{dx}{x}\rangle_{ch}}{(2\pi\ii)^n}\label{IntroFormula}
\end{align}
when $z$ is in the non-empty open set $U_T$.
\end{thm}

\noindent
In the theorem above, we put $(\gamma-{\bf b})_{\bf b}=\prod_{l=1}^k(\gamma_l-{b}_l)_{ b_l}=\prod_{l=1}^k\frac{\Gamma(\gamma_l)}{\Gamma(\gamma_l-b_l)}$ and $x^{\bf a}h^{\bf b}=x_1^{a_1}\dots x_n^{a_n}h_1^{b_1}\dots h_k^{b_k}$ if ${\bf a}={}^t(a_1,\dots,a_n)$ and ${\bf b}={}^t(b_1,\dots,b_k)$. $A_\sigma$ is a square matrix obtained by aligning column vectors of $A$ labeled by the elements of $\s$. The symbol $\sin\pi A_\s^{-1}\delta$ denotes the product of values of sine functions at each entry of the vector $\pi A_\s^{-1}\delta$. $U_T$ is an open neighbourhood of the point $z_T^\infty$. The formula (\ref{IntroFormula}) gives a convergent Laurent series expansion of the cohomology intersection number $\langle x^{\bf a}h^{\bf b}\frac{dx}{x},x^{{\bf a}^\prime}h^{{\bf b}^\prime}\frac{dx}{x}\rangle_{ch}$. Note that cohomology classes of the form $[x^{\bf a}h^{\bf b}\frac{dx}{x}]$ generate the algebraic de Rham cohomology groups $\Homo_{\rm dR}^*\left( U;(\mathcal{O}_U,\nabla_z)\right)$ and its dual $\Homo_{\rm dR}^*\left( U;(\mathcal{O}_U^\vee,\nabla_z^\vee)\right)$ (\cite{HibiTakayamaNishiyama}). 


The derivation of the formula (\ref{IntroFormula}) is based on the twisted analogue of Riemann-Hodge bilinear relation, commonly referred to as the twisted period relation. The twisted period relation expresses the cohomology intersection number in terms of the homology intersection number and period integrals. The essential part of this paper \S\ref{SectionEuler} and \S\ref{IntersectionNumbers} provides a concrete method of constructing a good basis of the rapid decay homology group associated to the Euler-Laplace  integral (\ref{MixedIntegral}). In particular, we have closed formulas of the homology intersection matrix and the period integrals. Behind the construction, we fully make use of the combinatorial structure of GKZ system. In the rest of the introduction, we explain the background of our study.

\subsection{The role of twisted (co)homology intersection numbers}

It seems that the interest on computing (co)homology intersection numbers is confined in a small community of specialists. Therefore, it is reasonable to recall the importance of computing the exact formula of (co)homology intersection numbers.

It was discovered in \cite{ChoMatsumoto} that a family of functional identities of hypergeometric functions called quadratic relations can be derived in a systematic way from the twisted version of Riemann-Hodge bilinear relation. This relation is a compatibility among cohomology intersection form, homology intersection form, and twisted period pairings. K.Cho and K.Matsumoto developed a method of evaluating cohomology intersection numbers for $1$-dimensional integrals, which was generalized to generic hyperplane arrangement case by Matsumoto in \cite{MatsumotoIntersection}. Another application of cohomology intersection number is the derivation of Pfaffian system that Euler-Laplace integral (\ref{MixedIntegral}) satisfies. In \cite{GotoKanekoMatsumoto}, \cite{GotoMatsumotoContiguity}, \cite{FrellesvigMastroliaMizera}, \cite{MastroliaMizera}, \cite{MizeraIntersection}, the authors deal with various Pfaffian systems from this point of view.


On the other hand, the nature of homology intersection number is more topological. Let us recall that the Euler-Laplace integral (\ref{MixedIntegral}) can be regarded as a result of the period pairing between the algebraic de Rham cohomology group $\Homo_{\rm dR}^n\left( U;(\mathcal{O}_U,\nabla_z)\right)$ and the rapid decay homology group $\Homo_n^{\rm r.d.}\left( U;(\mathcal{O}_{U^{an}}^\vee,\nabla_z^{{\rm an}\vee})\right)$ of M.Hien (\cite{HienRDHomology}). Therefore, the homology intersection pairing $\langle\bullet,\bullet\rangle_h$ can naturally be defined as a perfect pairing between two rapid decay homology groups $\Homo_n^{\rm r.d.}\left( U;(\mathcal{O}_{U^{an}}^\vee,\nabla_z^{{\rm an}\vee})\right)$ and $\Homo_n^{\rm r.d.}\left( U;(\mathcal{O}_{U^{an}},\nabla_z^{\rm an})\right)$, which is compatible with the cohomology intersection pairing $\langle\bullet,\bullet\rangle_{ch}$ through period pairings.  For the precise definition of $\langle\bullet,\bullet\rangle_{ch}$, see \S\ref{RDIntersection}. It is customary in the theory of hypergeometric integrals to call the rapid decay homology groups $\Homo_n^{\rm r.d.}\left( U;(\mathcal{O}_{U^{an}}^\vee,\nabla_z^{{\rm an}\vee})\right)$ and $\Homo_n^{\rm r.d.}\left( U;(\mathcal{O}_{U^{an}},\nabla_z^{\rm an})\right)$ the twisted homology groups.

The most typical application of homology intersection number is probably the monodromy invariant hermitian matrix in the study of period maps (\cite{DeligneMostow}, \cite{MSY}, \cite{YoshidaMyLove}). It plays an essential role when one determines the image of the period map. Therefore, it is crucial that the homology intersection numbers can be evaluated exactly for a given basis of the twisted homology group. In these studies, the integral (\ref{MixedIntegral}) is reduced to Aomoto-Gelfand hypergeoemtric integral (\cite{AomotoKita}). An explicit formula of the homology intersection numbers  can be found in \cite{KitaYoshida}, \cite{KitaYoshida2}.

Another application of the homology intersection number is the global study of the hypergeometric function that the integral (\ref{MixedIntegral}) represents. This is based on the standard fact that the rapid decay homology group $\Homo_n^{\rm r.d.}\left( U;(\mathcal{O}_{U^{an}}^\vee,\nabla_z^{{\rm an}\vee})\right)$ is isomorphic to the solution space of the Gau\ss-Manin connection associated to the integral (\ref{MixedIntegral}). Once this isomorphism is established, one wants to understand the parallel transport of a cycle $[\Gamma(z)]\in\Homo_n^{\rm r.d.}\left( U;(\mathcal{O}_{U^{an}}^\vee,\nabla_z^{{\rm an}\vee})\right)$ from a singular point $z=z_1$ to another one $z=z_2$. This can be achieved once we find a good basis $\{ [\check{\Gamma}_i(z)]\}_{i}$ of $\Homo_n^{\rm r.d.}\left( U;(\mathcal{O}_{U^{an}},\nabla_z^{{\rm an}})\right)$ near $z=z_2$ with which one can compute the homology intersection numbers $\langle [\Gamma(z)], [\check{\Gamma}_i(z)]\rangle_h$. The interested readers may refer to \cite{GotoPhDThesis}, \cite{MatsumotoFD},  \cite{MatsumotoYoshida} and references therein.

In these studies, the most important part is the choice of a good basis of twisted homology groups with which one can compute the exact value of homology intersection numbers. Let us review the method of constructing a basis of the twisted homology group.

\subsection{Constructing a good basis of cycles I: The method of stationary phase}

Let us recall the construction of a good basis of the twisted homology group by means of the method of stationary phase. For the detail, see \cite{AomotoLesEquations}, \cite{AomotoKita}, \cite{ArnoldStationaryPhase}, \cite{Fedoryuk}, or \cite{PhamMethodeDeCol}. Introducing a real parameter $\tau$ and positive rational numbers $\eta_1,\dots,\eta_{k+n}\in\Q_{ >0}$, we consider an integral 
\begin{equation}\label{OscillatoryIntegral}
f_{\Gamma}^{\pm}(z)=\int_{\Gamma}e^{\tau\varphi}h_{1,z^{(1)}}(x)^{\mp\gamma_1}\cdots h_{k,z^{(k)}}(x)^{\mp\gamma_k}x^{\pm c}\omega
\end{equation}
where $\varphi=\varphi(x;z,\eta)=h_{0}(x;z)+\sum_{l=1}^k\eta_{n+l}\log h_{l}(x;z)+\sum_{i=1}^n\eta_i\log x_i$. Note that the integral (\ref{OscillatoryIntegral}) is essentially same as the integral (\ref{MixedIntegral}). The recipe of the method of stationary phase is as follows: We first detect all the critical points of $\varphi(x;z,\eta)$ where $z$ and $\eta$ are regarded as fixed constants. The set of critical points is denoted by ${\rm Crit}$. Assume that any critical point $p$ is Morse. For each critical point $p\in{\rm Crit}$, we associate the contracting (resp. expanding) manifold $L_p^+$ (resp. $L_p^-$) as a set of all trajectories that have the point $p$ as the limit point at $t=+\infty$ (resp. $t=-\infty$) of the gradient vector field of $\Re\varphi$. Assuming the completeness of the gradient flow, we obtain $n$-dimensional cycle $L_p^+$ (resp. $L_p^-$) which is called a positive (resp. negative) Lefschetz thimble (to be more precise, one must replace the Euclidian metric by a complete K\"ahler metric as in \cite[Chapter 4]{AomotoKita}). Under the generic condition that $L_p^\pm$ does not flow into another critical point $q\in {\rm Crit}$, the method of stationary phase tells us that the asymptotic expansion of the integral $f_{L_p^\pm}^{\pm}(z)$ as $\tau\rightarrow \pm\infty$ has the form $e^{\tau\varphi(p;z,\eta)}a_p^\pm\tau^{-\frac{n}{2}}(1+o(\tau))$ where $a_p^\pm$ is a non-zero constant which can be computed from the Hessian of $\varphi(x;z,\eta)$ at p. Moreover, by construction, we have the natural orthogonality relation $\langle L_p^+,L_q^-\rangle_h=0$ if $p\neq q$ and $\langle L_p^+,L_p^-\rangle=1$ (Smale's transversality). In this fashion, it is expected that we obtain a good basis $\{ L_p^+\}_{p\in{\rm Crit}}\subset\Homo_n^{\rm r.d.}\left( U;(\mathcal{O}_{U^{an}}^\vee,\nabla_z^{{\rm an}\vee})\right)$ and $\{ L_p^-\}_{p\in{\rm Crit}}\subset\Homo_n^{\rm r.d.}\left( U;(\mathcal{O}_{U^{an}},\nabla_z^{\rm an})\right)$ of the twisted homology groups. Though this method is simple and widely used, it is not easy in general to compute the exact asymptotic expansion of the integral $f_{L_p^\pm}^\pm$  as $\tau\rightarrow\pm\infty$. Moreover, though it is expected that the global monodromy representation can be computed by means of Picard-Lefschetz formula, it is not clear how one can expand the integral $f^\pm_{L_p^\pm}$ as a convergent series in $z$ variables.

\subsection{Constructing a good basis of cycles I\hspace{-.1em}I: Regularization of chambers}

When the integral (\ref{MixedIntegral}) is associated to a hyperplane arrangement, positive Lefschetz thimbles $\{ L_p^+\}_{p\in{\rm Crit}}$ can be understood in a combinatorial way. For simplicity, let us consider the case when $h_l$ are polynomials of degree $1$, $h_0\equiv 0$, the variable $z$ is real and parameters $\gamma_l$ and $c$ are generic. When $z$ is taken to be real and generic, the real part $U^{\rm an}\cap\R^N$ is decomposed into finitely many connected components which are called chambers. In particular, a relatively compact chamber is called a bounded chamber. Then, each positive Lefschetz thimble is represented by exactly one bounded chamber $\Delta$ with the property $p\in\Delta$. Let us describe it more precisely. We write $\mathcal{L}^+$ for the local system on $U^{\rm an}$ of flat sections of $\nabla_z^{{\rm an}\vee}$. The dual local system of $\mathcal{L}^+$ is denoted by $\mathcal{L}^-$. For each bounded chamber $\Delta$, we write $\Delta^+$ (resp. $\Delta^-$) for an element of the locally finite homology group $\Homo^{\rm lf}_n(U,\mathcal{L}^+)$ (resp. $\Homo^{\rm lf}_n(U^{\rm an},\mathcal{L}^-)$) represented by $\Delta$. Then, the set $\{ \Delta^\pm\}_{\Delta:\text{bounded chambres}}$ is a basis of $\Homo^{\rm lf}_n(U^{\rm an},\mathcal{L}^\pm)$ (\cite{KohnoHomology}, \cite[Proposition 6.4.1]{OrlikTerao}) and $\{ \Delta^+\}_{\Delta:\text{bounded chambres}}$ coincides with the basis $\{ L_p^+\}_{p\in{\rm Crit}}$.

It is important that the intersection numbers with respect to these bases $\{ \Delta^\pm\}_{\Delta:\text{bounded chambres}}$ can be computed explicitly. Let ${\rm reg}:\Homo^{\rm lf}_n(U^{\rm an},\mathcal{L}^\pm)\tilde{\rightarrow}\Homo_n(U^{\rm an},\mathcal{L}^\pm)$ be the inverse of the natural isomorphism $\Homo_n(U^{\rm an},\mathcal{L}^\pm)\tilde{\rightarrow}\Homo_n^{\rm lf}(U^{\rm an},\mathcal{L}^\pm)$. The isomorphism ${\rm reg}$ is called the regularization map. For a bounded chamber $\Delta$, there is an explicit description of the cycle ${\rm reg}(\Delta^+)$. For simplicity, we suppose that the hyperplane arrangement $D$ is generic, i.e., $D$ is a normal crossing divisor. For a given element $[\Delta^+]\in\Homo_n^{\rm lf}(U^{\rm an},\mathcal{L}^+)$ represented by a bounded chamber $\Delta$, we cut off a small neighbourhoods of the faces of $\Delta$ and consider a small chamber $\s$. We can regard $\s$ as a finite chain whose orientation is induced from that of $\Delta^+$. At this stage, $\s$ is not a cycle. Then, to each face of $\s$, we put a ``pipe'' encircling a divisor with a suitable coefficient in the local system $\mathcal{L}^+$. Repeating this process, we obtain a cycle ${\rm reg }(\Delta^+)$ whose homology class in $\Homo_n^{\rm lf}(U^{\rm an},\mathcal{L}^+)$ is $[\Delta^+]$ (\cite{KitaWronskian}). From this construction, we see that the regularized cycle ${\rm reg}(\Delta^+)$ is, up to a constant multiplication, equal to the multidimensional Pochhammer cycle (see e.g. \cite{Beukers}). As for the description of the regularization map for a particular non-generic arrangement, see  \cite[Chapter 5]{TsuchiyaKanie}. With this construction, we have an explicit formula of the homology intersection number $\langle {\rm reg}(\Delta^+_1),\Delta_2^-\rangle_h$, which turns out to be a periodic function of the parameters (\cite{KitaYoshida} and \cite{KitaYoshida2}).





\begin{figure}[h]
\begin{minipage}{0.5\hsize}
\begin{center}
\begin{tikzpicture}
\draw[-] (-1,0)--(4,0);
\draw[-] (0,-1)--(0,3);
\draw[-] (-1,3)--(3,-1);
\draw[-] (-2,-1)--(4,0.1);
\filldraw[fill=lightgray] (0,0) -- (2,0) -- (0,2);
\node at (0.6,0.6) {$\Delta^+$};
\end{tikzpicture}
\end{center}
\end{minipage}
\begin{minipage}{0.5\hsize}
\begin{center}
\begin{tikzpicture}
\draw[-] (-1,0)--(4,0);
\draw[-] (0,-1)--(0,3);
\draw[-] (-1,3)--(3,-1);
\draw[-] (-2,-1)--(4,0.1);
\filldraw[fill=lightgray] (0.2,0.2)--(1.5,0.2)--(0.2,1.5)--cycle;
\draw[-] (0.5,1.8) -- (1.8,0.5);
\draw[-] (0.2,-0.2) -- (1.5,-0.2);
\draw[-] (-0.2,0.2) -- (-0.2,1.5);
\draw[-] (0.2,1.5) to [out=100,in=130] (0.5,1.8);
\draw[-] (1.5,0.2) to [out=100,in=130] (1.8,0.5);
\draw[-] (0.2,1.5) to [out=20,in=-80] (0.5,1.8);
\draw[dashed] (1.5,0.2) to [out=0,in=-80] (1.8,0.5);
\draw[-] (1.5,0.2) to [out=-50,in=30] (1.5,-0.2);
\draw[dashed] (1.5,0.2) to [out=-100,in=100] (1.5,-0.2);
\draw[-] (0.2,0.2) to [out=-50,in=30] (0.2,-0.2);
\draw[-] (0.2,0.2) to [out=-120,in=120] (0.2,-0.2);
\draw[-] (0.2,0.2) to [out=110,in=70] (-0.2,0.2);
\draw[-] (0.2,0.2) to [out=-110,in=-70] (-0.2,0.2);
\draw[-] (0.2,1.5) to [out=110,in=70] (-0.2,1.5);
\draw[dashed] (0.2,1.5) to [out=-110,in=-70] (-0.2,1.5);
\node at (2,2){reg$(\Delta^+)$};
\node at (0.6,0.6) {$\sigma$};
\end{tikzpicture}
\end{center}
\end{minipage}
\caption{A bounded chamber and its regularization}
\end{figure}
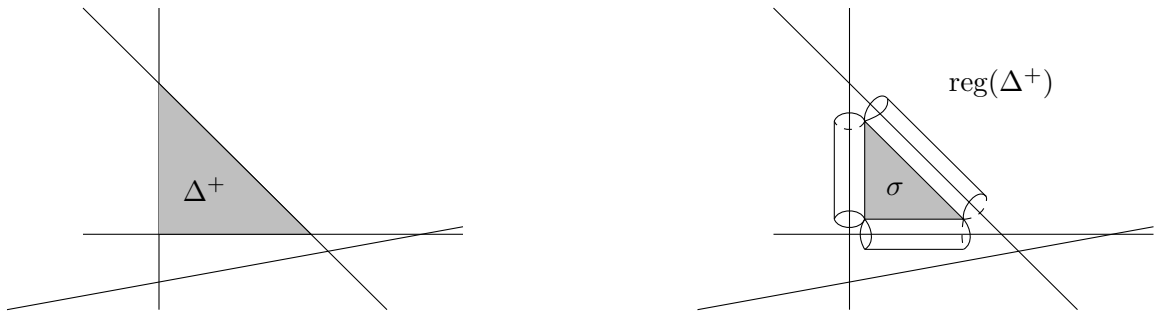

\subsection{Our method: The viewpoint of GKZ system}

In the previous subsection, we explained that there is a combinatorial method of constructing a basis of the homology group whose homology intersection matrix has an explicit formula when the integrand of (\ref{MixedIntegral}) defines a hyperplane arrangement. Now it is natural to ask to what extent the construction can be generalized. There are many papers discussing constructions of bases of the (rapid decay) homology groups in various contexts  (e.g., \cite{AomotoConfigurations}, \cite{AomotoHypersphereSchlafli}, \cite{EsterovTakeuchi}, \cite{GotoPhDThesis},  \cite{MarkovTarasovVarchenko}). To our knowledge, however, there is no systematic method of constructing a basis of the rapid decay homology group whose intersection matrix as well as its period integrals have closed formulas. 


We propose a systematic method of constructing a basis of the rapid decay homology group without any restriction on the Laurent polynomials $h_l(x)$. We no longer expect that the ``visible cycles'' such as chambers or their regularizations are sufficient to construct a basis. Instead, we make use of a convergent regular triangulation $T$ of the cone $\R_{>0}A$. Let us illustrate our construction by a simple example.

Suppose that $n=1$ and the integrand of (\ref{MixedIntegral}) is $e^{z_1x+z_2x^{-1}}(z_3+z_4x)^{-\gamma}x^c$. For simplicity, we assume that $z_1$ and $z_2$ are positive. We consider the real oriented blow-up (\cite[\S 8.2]{SabbahIntroductionToStokes}) $(\C\setminus\{ 0,\zeta=-\frac{z_3}{z_4}\})\sqcup S^1\infty\sqcup S^10$, where $S^1\infty$ (resp. $S^10$) is the circle at $\infty$ (resp. at the origin). Under the standard identification $S^1=\R/2\pi\Z$, the twisted homology group $\Homo_1^{\rm r.d.}\left( U;(\mathcal{O}_{U^{an}}^\vee,\nabla_z^{{\rm an}\vee})\right)$ is simply given by a relative homology group $\Homo_1\left( (\C\setminus\{ 0,\zeta\})\sqcup (\pi,\frac{3\pi}{2})\infty\sqcup (\pi,\frac{3\pi}{2}) 0,(\pi,\frac{3\pi}{2})\infty\sqcup (\pi,\frac{3\pi}{2}) 0;\mathcal{L}\right)$ where $\mathcal{L}$ is the local system of flat sections of the connection $\nabla^{{\rm an}\vee}_z$. As a convergent regular triangulation, we take $T=\{ 14,34,23\}$. For the simplex $14\in T$, we associate a limit $z_2,z_3\rightarrow 0$ with $z_1,z_4$ fixed. At the limit, we observe that the exponential factor $e^{z_1x+z_2x^{-1}}$ becomes $e^{z_1x}$ while the point $\zeta$ approaches the origin. The key observation is that, at the limit, we may pretend as if the rapid decay homology group of $e^{z_1x}x^{c-\gamma}$ is associated to the integrand $e^{z_1x}x^{c-\gamma}$. Since the rank of this rapid decay homology group is $1$, we can take a generator $[\Gamma_{14}(z)]$. Though this cycle is defined when $z_2$ and $z_3$ are small, it can be defined for any generic $z$ after parallel transportation (see the argument in \S\ref{SectionEuler}). Applying this process to other simplices $34,23$, we obtain a basis $\left\{ [\Gamma_{14}(z)],[\Gamma_{34}(z)],[\Gamma_{23}(z)]\right\}$ of the twisted homology group.

After a suitable modification, we can generalize the construction above to the general case: We take a convergent regular triangulation $T$ of $\R_{>0}A$. To each simplex $\s\in T$, we associate a limit $z_j\rightarrow 0$ ($j\notin\s$). At the limit, we can pretend as if the rank of the rapid decay homology group is the cardinality $|G_{\s}|$ of the finite Abelian group $G_{\s}$. We can take out a set of independent cycles $\{\Gamma_{\s,h}\}_{h\in\widehat{G_{\s}}}$ labeled by the dual group $\widehat{G_{\s}}$. After performing the parallel transportation, the union $\Gamma_T=\bigcup_{\s\in T}\{\Gamma_{\s,h}\}_{h\in\widehat{G_{\s}}}$ is a basis of the rapid decay homology group $\Homo_n^{\rm r.d.}\left( U;(\mathcal{O}_{U^{an}}^\vee,\nabla_z^{{\rm an}\vee})\right)$. The same construction works for the dual group $\Homo_n^{\rm r.d.}\left( U;(\mathcal{O}_{U^{an}},\nabla_z^{{\rm an}})\right)$ and this basis is denoted by $\check{\Gamma}_T=\bigcup_{\s\in T}\{\check{\Gamma}_{\s,h}\}_{h\in\widehat{G_{\s}}}$.

Let $z_\s^{\infty}$ be the point near $z_T^\infty$ of which the $j$-th coordinate is $0$ unless $j\in \s$. In a sense, $z_\s^\infty$ plays the role of a critical point in the method of stationary phase. Namely, the cycles $\Gamma_{\s,h}$ are characterized by the behaviour of the integrand of (\ref{MixedIntegral}) near the special point $z_\s^{\infty}$. The general construction above can be carried out in quite an explicit manner. Indeed, after a sequence of coordinate transformations, each cycle $\Gamma_{\s,h}$ can be constructed as a product of Pohhammer cycles and Hankel contour. Thanks to this explicit construction, we can also expand the integral (\ref{MixedIntegral}) along the cycle $\Gamma_{\s,h}$ into a hypergeometric series converging near the point $z_T^\infty$. In \S\ref{SectionEuler}, we give an explicit transformation matrix between two basis of solutions $\Phi_T$ and $\Gamma_T$ in terms of the character matrices of the finite Abelian groups $G_{\s}$ (\cref{thm:fundamentalthm3}). The homology intersection matrix $\left(\langle \Gamma_{\s_1,h_1},\check{\Gamma}_{\s_2,h_2}\rangle_h\right)$ is naturally block-diagonalized in the sense that $\langle \Gamma_{\s_1,h_1},\check{\Gamma}_{\s_2,h_2}\rangle_h=0$ unless $\s_1=\s_2$. Finally, we can evaluate the diagonal term $\langle \Gamma_{\s,{\bf 0}},\check{\Gamma}_{\s,{\bf 0}}\rangle_h$  as a periodic function in parameters when the group $G_{\s}$ is the trivial Abelian group ${\bf 0}$. A table comparing our construction with the method of stationary phase is the following.

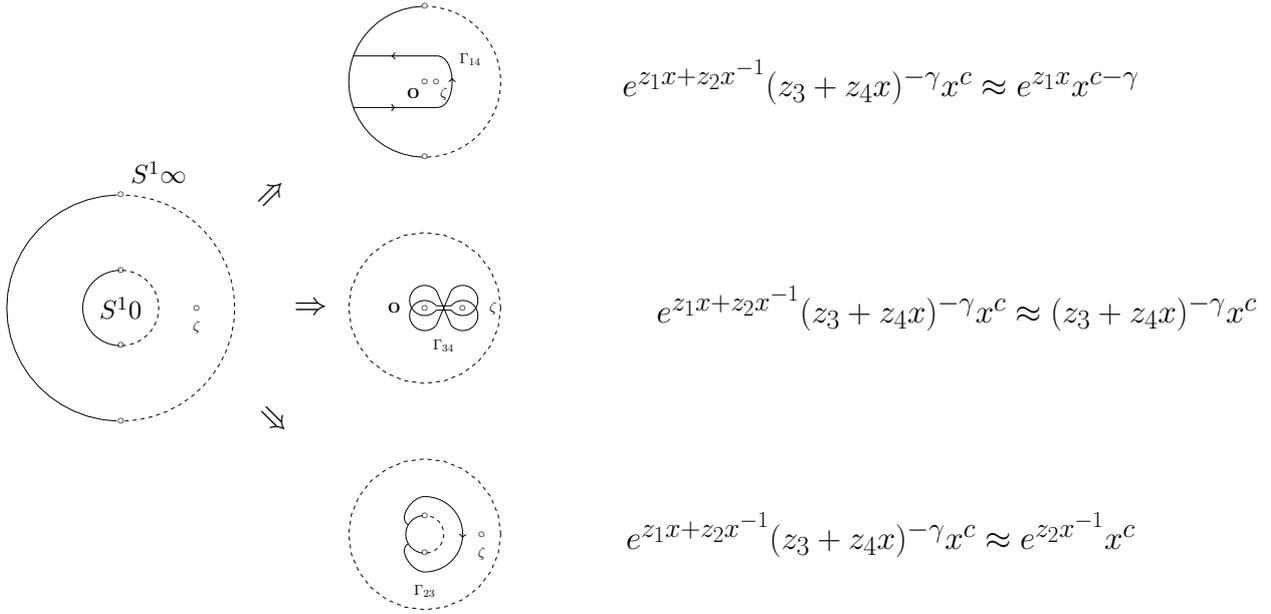
\begin{figure}[t]
\begin{center}
\scalebox{0.5}[0.5]{
\begin{tikzpicture}
\draw[-,domain=92:268] plot ({3*cos(\x)},{3*sin(\x)});
\draw[-,domain=92:268] plot ({cos(\x)},{sin(\x)});
\draw[dashed,domain=-88:88] plot ({3*cos(\x)},{3*sin(\x)});
\draw[dashed,domain=-88:88] plot ({cos(\x)},{sin(\x)});
\node at (0,3){$\circ$};
\node at (0,-3){$\circ$};
\node at (0,1){$\circ$};
\node at (0,-1){$\circ$};
\node at (2,0){$\circ$};
\node at (2,-0.5){$\zeta$};
\node at (0,0){\scalebox{1.8}{$S^10$}};
\node at (1,3.6){\scalebox{1.8}{$S^1\infty$}};
\draw[-,domain=92:268] plot ({8+2*cos(\x)},{6+2*sin(\x)});
\draw[dashed,domain=-88:88] plot ({8+2*cos(\x)},{6+2*sin(\x)});
\node at (8,8){$\circ$};
\node at (8,4){$\circ$};
\node at (8,6){$\circ$};
\node at (8.3,6){$\circ$};
\node at (7.7,5.7){\bf O};
\node at (8.5,5.7){$\zeta$};
\node at (9.2,6.6){$\Gamma_{14}$};
\draw[->-=.5,domain=0:2.2] plot ({8+2*cos(200)+(\x)},{6+2*sin(200)});
\draw[-<-=.5,domain=0:2.2] plot ({8+2*cos(160)+(\x)},{6+2*sin(160)});
\coordinate (A) at ({10.2+2*cos(160)},{6+2*sin(160)});
\coordinate (B) at ({10.2+2*cos(200)},{6+2*sin(200)});
\draw[-<-=.5] (A) to [out=0,in=0] (B);
\draw[dashed,domain=0:360] plot ({8+2*cos(\x)},{2*sin(\x)});
\node at (8,0){$\circ$};
\node at (9,0){$\circ$};
\node at (7.2,0){\bf O};
\node at (9.8,0){$\zeta$};
\node at (8.5,-1){$\Gamma_{34}$};
\draw[-,domain=-140:140] plot ({9+0.4*cos(\x)}, {0.2+0.4*sin(\x)});
\draw[-,domain=40:320] plot ({8+0.4*cos(\x)}, {0.2+0.4*sin(\x)});
\draw[-,domain=-140:140] plot ({9+0.4*cos(\x)}, {-0.2+0.4*sin(\x)});
\draw[-,domain=40:320] plot ({8+0.4*cos(\x)}, {-0.2+0.4*sin(\x)});
    \coordinate (A1) at ({8+0.4*cos(40)}, {0.2+0.4*sin(40)});
    \coordinate (A2) at ({9+0.4*cos(-140)}, {-0.2+0.4*sin(-140)});
    \coordinate (B1) at ({9+0.4*cos(140)}, {-0.2+0.4*sin(140)});
    \coordinate (B2) at ({8+0.4*cos(40)}, {-0.2+0.4*sin(40)});
    \coordinate (C1) at ({8+0.4*cos(320)}, {-0.2+0.4*sin(320)});
    \coordinate (C2) at ({9+0.4*cos(140)}, {0.2+0.4*sin(140)});
    \coordinate (D1) at ({9+0.4*cos(-140)}, {0.2+0.4*sin(-140)});
    \coordinate (D2) at ({8+0.4*cos(320)}, {0.2+0.4*sin(320)});
\draw[-] (A1) -- (A2);
\draw[-] (B1) -- (B2);
\draw[-] (C1) -- (C2);
\draw[-] (D1) -- (D2);
\draw[dashed,domain=0:360] plot ({8+2*cos(\x)},{-6+2*sin(\x)});
\draw[-,domain=96:264] plot ({8+0.5*cos(\x)},{-6+0.5*sin(\x)});
\draw[dashed,domain=-84:84] plot ({8+0.5*cos(\x)},{-6+0.5*sin(\x)});
\node at (8,-5.5){$\circ$};
\node at (8,-6.5){$\circ$};
\node at (9.5,-6){$\circ$};
\node at (9.5,-6.5){$\zeta$};
\node at (8,-7.5){$\Gamma_{23}$};
\draw[-<-=.5,domain=-90:90] plot ({8+cos(\x)},{-6+sin(\x)});
\coordinate (AA) at ({8+cos(90)},{-6+sin(90)});
\coordinate (BA) at ({8+cos(-90)},{-6+sin(-90)});
\coordinate (AB) at ({8+0.5*cos(150)},{-6+0.5*sin(150)});
\coordinate (BB) at ({8+0.5*cos(210)},{-6+0.5*sin(210)});
\draw[-] (AB) to [out=150,in=180] (AA);
\draw[-] (BB) to [out=210,in=180] (BA);
\node at (4,3){$\rotatebox{45}{\Huge $\Rightarrow$}$};
\node at (5,0){\Huge $\Rightarrow$};
\node at (4,-3){$\rotatebox{-45}{\Huge $\Rightarrow$}$};
\node at (20,6){\Huge $e^{z_1x+z_2x^{-1}}(z_3+z_4x)^{-\gamma}x^c\approx e^{z_1x}x^{c-\gamma}$};
\node at (22,0){\Huge $e^{z_1x+z_2x^{-1}}(z_3+z_4x)^{-\gamma}x^c\approx (z_3+z_4x)^{-\gamma}x^c$};
\node at (20,-6){\Huge $e^{z_1x+z_2x^{-1}}(z_3+z_4x)^{-\gamma}x^c\approx  e^{z_2x^{-1}}x^c$};
\end{tikzpicture}
}\label{TheFirstFigure}
\caption{Our construction of cycles}
\end{center}
\end{figure}

\begin{table}[ht]
\begin{tabular}{|c|c|c|}
\hline
 &Method of stationary phase & Our method \\
\hline
Special point &Critical point &$z_\s^\infty$ ($\s\in T$) \\
\hline
Integration cycle &Lefschetz thimble& (Pochhammer cycle)$\times$(Hankel contour)\\
\hline
Intersection matrix & Identity matrix& Periodic function in parameters\\
\hline
Expansion of the integral &Asymptotic series in $\tau\approx\infty$& Hypergeometric series in $z\approx z_T^\infty$ \\
\hline
\end{tabular}
\end{table}


\subsection{Quadratic relations for Aomoto-Gelfand hypergeoemtric functions and its confluence}


In the last two sections, we apply \cref{thm:IntroTheorem} to Aomoto-Gelfand system (\S \ref{QuadraticRelationsForGrassman}) and its confluence (\S \ref{E(21...1)}). The special feature of these classes is that there is a special convergent unimodular regular triangulation $T$ called the staircase triangulation (\cite{GelfandGraevRetakh}). Since there is a one-to-one correspondence between simplices of $T$ and spanning trees of a complete bipartite graph, we can express the homology intersection numbers in terms of these graphs. It is noteworthy that our basis of cycles is different from that consisting of regularizations of bounded chambers $\{ {\rm reg}(\Delta^\pm)\}_{\Delta:\text{bounded chambers}}$. Namely, our cycles may go around several divisors $\{ h_{l}(x)=0\}$ simultaneously so that they are linked in a more complicated way. In this sense, our cycles can be referred to as ``linked cycles''. 

Besides the precise description of the twisted homology group, there is a basis of the algebraic de Rham cohomology group whose cohomology intersection matrix has a closed formula (\cite{MatsumotoIntersection} and \cref{prop:E(21...1)c.i.n.}). Therefore, \cref{thm:IntroTheorem} gives rise to a general quadratic relation for Aomoto-Gelfand hypergeometric functions (\cref{thm:QuadraticRelationsForAomotoGelfand}) and its confluence (\cref{thm:QuadraticRelationsForAomotoGelfand2}). The simplest example of such an identity is the following relation:

\begin{align}
&(1-\gamma+\alpha)(1-\gamma+\beta){}_2F_1\left(\substack{\alpha,\beta\\ \gamma};z\right){}_2F_1\left(\substack{-\alpha,-\beta\\ 2-\gamma};z\right)-\alpha\beta{}_2F_1\left(\substack{\gamma-\alpha-1,\gamma-\beta-1\\ \gamma};z\right)
{}_2F_1\left(\substack{1-\gamma+\alpha,1-\gamma+\beta \\ 2-\gamma};z\right)\nonumber\\
=&(1-\gamma+\alpha+\beta)(1-\gamma).\label{QRGauss}
\end{align}

\noindent
Here, ${}_2F_1\left(\substack{\alpha,\beta\\ \gamma} ;z\right)$ is the usual Gau\ss' hypergeometric series
\begin{equation}\label{series}
{}_2F_1\left(\substack{\alpha,\beta\\ \gamma} ;z\right)=\displaystyle\sum_{n=0}^\infty\frac{(\alpha)_n(\beta)_n}{(\gamma)_n(1)_n}z^n
\end{equation}
with complex parameters $\alpha,\beta\in\C$ and $\gamma\in\C\setminus\Z_{\leq 0}$.

\subsection{The structure of the paper}

We give a plan of this paper. The first three sections \S\ref{SectionDModules},  \S\ref{SectionRapidDecay} and \S\ref{RDIntersection} are devoted to the foundations of the study of the integral (\ref{MixedIntegral}). In \S\ref{SectionDModules}, we establish an isomorphism between the Gau\ss-Manin connection associated to the integral (\ref{MixedIntegral}) and GKZ $\DD$-module $M_A(\delta)$ under the non-resonance assumption of $\delta$ (\cref{thm:mainDresult}). The technique of the proof is an adaptation of the arguments of \cite[\S 3]{AdolphsonSperber} and \cite[\S 9]{DworkLoeser} in the framework of $\DD$-modules. Based on \S\ref{SectionDModules}, we establish an isomorphism between the rapid decay homology group $\Homo_n^{\rm r.d.}\left( U;(\mathcal{O}_{U^{an}}^\vee,\nabla_z^{{\rm an}\vee})\right)$ and the stalk of the solutions of $M_A(\delta)$ at a generic point $z$ in \S\ref{SectionRapidDecay}. The essential part of the construction is a compactification of $U$ in the spirit of \cite{KhovanskiNewton} and \cite{EsterovTakeuchi}. In \S\ref{RDIntersection}, we give a foundation of the (co)homology intersection pairings. One can find a similar discussion in \cite{FSY}. The definitions of the pairing is a natural generalization of that of \cite{MajimaMatsumotoTakayama}. From \S\ref{SectionSeries}, we combine the general theory above with the combinatorial structure of GKZ system. In \S\ref{SectionSeries}, we summarize the construction of hypergeoemtric series developed in \cite{FernandezFernandez}. The readers familiar with this topic can skip this section. In \S\ref{SectionEuler} we construct a basis $\Gamma_T$ of the rapid decay homology group for a given convergent regular triangulation $T$. We give an explicit transformation formula between the basis of series solutions $\Phi_T$ and the basis $\Gamma_T$ (\cref{thm:fundamentalthm3}). In \S\ref{IntersectionNumbers}, we establish a closed formula of intersection matrix with respect to the bases $\Gamma_T$ and $\check{\Gamma}_T$ when the convergent regular triangulation $T$ is unimodular. In \S\ref{QuadraticRelationsForGammaSeries}, we give a proof of \cref{thm:IntroTheorem}. In the last two sections \S\ref{QuadraticRelationsForGrassman} and \S\ref{E(21...1)}, we discuss the application of \cref{thm:IntroTheorem} to Aomoto-Gelfand system and its confluence. After recalling/establishing closed formulas of the cohomology intersection numbers (\cite{MatsumotoIntersection}, \cref{prop:E(21...1)c.i.n.}) we give closed formulas of quadratic relations associated to the staircase triangulation in terms of bipartite graphs (\cref{thm:QuadraticRelationsForAomotoGelfand} and  \cref{thm:QuadraticRelationsForAomotoGelfand2}).



\section{General framework of Euler-Laplace integral representation}\label{SectionDModules}
In this section, we discuss $\DD$-modules corresponding to the integral (\ref{MixedIntegral}). Namely, we establish an isomorphism between GKZ system $M_A(\delta)$ and the Gauss-Manin system associated to the integral (\ref{MixedIntegral}) when the parameter $\delta$ is non-resonant (Theorem \ref{thm:mainDresult}). 

\subsection{Some equivalences of integral representations}\label{subsec:2.1}
Let us recall some basic notation and results of algebraic $\DD$-modules. For their proofs, see \cite{BorelEtAl} or \cite{HTT}. Let $X$ and $Y$ be smooth algebraic varieties over the complex number field $\C$ and let $f:X\rightarrow Y$ be a morphism. Throughout this paper, we write $X$ as $X_x$ when we emphasize that $X$ is equipped with the coordinate $x$. We write $\DD_X$ for the sheaf of linear partial differential operators on $X$ and write $D^b_{q.c.}(\DD_X)$ (resp. $D^b_{coh}(\DD_X)$, resp. $D^b_{h}(\DD_X)$) for the derived category of bounded complexes of left $\DD_X$-modules whose cohomologies are quasi-coherent  (resp. coherent, resp. holonomic). We write $D^b_{*}(\DD_X)$ for either of two categories $D^b_{q.c.}(\DD_X)$ or $D^b_{h}(\DD_X)$. For any coherent $\DD_X$-module $M$, ${\rm Char}(M)$ denotes its characteristic variety in $T^*X$. In general, for any object $M\in D^b_{coh}(\DD_X)$, we define its characteristic variety by ${\rm Char}(M)=\displaystyle\cup_{n\in\Z}{\rm Char}\left(\Homo^n(M)\right)$. We write ${\rm Sing }(M)$ for the image of ${\rm Char}(M)$ by the canonical projection $T^*X\rightarrow X$. For any object $N\in D^b_*(\DD_Y)$, we define its inverse image $\LL f^*N\in D^b_*(\DD_X)$ (resp. its shifted inverse image $f^\dagger N\in D^b_*(\DD_X)$)  by the formula
\begin{equation}
\LL f^*N=\DD_{X\rightarrow Y}\overset{\LL}{\underset{f^{-1}\DD_Y}{\otimes}}f^{-1}N \;\;\;(\text{resp. } f^{\dagger}N=\LL f^*N[\dim X-\dim Y]),
\end{equation} 
where $\DD_{X\rightarrow Y}$ is the transfer module $\mathcal{O}_X\otimes_{f^{-1}\mathcal{O}_Y}f^{-1}\DD_Y.$ For any object $M\in D^b_*(\DD_X)$, we define its holonomic dual $\HD_X M\in D^b_*(\DD_X)^{op}$ by 
\begin{equation}
\HD_X M=\R\mathcal{H}om_{\DD_X}(M,\DD_X)\otimes_{\mathcal{O}_X}\Omega_X^{\otimes -1}.
\end{equation}
Note that $\HD_X$ is involutive, i.e., we have $\HD_X\circ\HD_X\simeq {\rm id}_X$. For any object $M\in D^b_*(\DD_X)$, we define its direct image $\int_fM\in D^b_*(\DD_Y)$  (resp. its proper direct image $\int_{f!}M\in D^b_*(\DD_Y)$) by
\begin{equation}
\int_fM=\R f_*(\DD_{Y\leftarrow X}\overset{\LL}{\underset{\DD_X}{\otimes}}M),\;\;\;(\text{resp. } \int_{f!}=\HD_Y\circ\int_f\circ\HD_X M),
\end{equation}
where $\DD_{Y\leftarrow X}$ is the transfer module $\Omega_X\otimes_{\mathcal{O}_X}\DD_{X\rightarrow Y}\otimes_{f^{-1}\mathcal{O}_Y}f^{-1}\Omega_Y$. If $X$ is a product variety $X=Y\times Z$ and $f:Y\times Z\rightarrow Y$ is the natural projection, the direct image can be computated in terms of (algebraic) relative de Rham complex
\begin{equation}\label{RelativeDeRhamFormula}
\int_fM\simeq \R f_*(\DR_{X/Y}(M)).
\end{equation}
In particular, if $Y=\{*\}$ (one point), and $M$ is a connection $M=(E,\nabla)$ on $Z$, then for any integer $p$, we have a canonical isomorphism
\begin{equation}
\Homo^p\left( \int_fM\right)\simeq\mathbb{H}_{dR}^{p+\dim Z}(Z,(E,\nabla)),
\end{equation} 
where $\mathbb{H}_{dR}^k$ denotes the $k$-th algebraic de-Rham cohomology group. If $\rho_f:X\times_YT^*Y\rightarrow T^*X$ and $\varpi_f:X\times_YT^*Y\rightarrow T^*Y$ denote associated morphisms, we have
\begin{equation}\label{eqn:EstimateOfChar}
{\rm Char}(\int_fM)\subset\varpi_f\rho_f^{-1}\left({\rm Char}(M)\right)
\end{equation}
for any $M\in D^b_{coh}(\mathcal{D}_X)$. Suppose that a cartesian diagram 
\begin{equation}
\xymatrix{
 X^\prime \ar[r]^{f^\prime } \ar[d]_{g^\prime}&Y^\prime\ar[d]^{g}\\
 X \ar[r]^{f}                             &Y
}
\end{equation}
is given. For any object $M\in D^b_*(\DD_X)$, we have the base change formula 
\begin{equation}\label{BaseChange}
g^\dagger\int_fM\simeq\int_{f^\prime}g^{\prime\dagger}M.
\end{equation}
For objects $M,M^\prime\in D^b_*(\DD_X)$ and $N\in D^b_*(\DD_Y)$, the tensor product $M\overset{\mathbb{D}}{\otimes}M^\prime\in D^b_*(\DD_X)$ and external tensor product $M\boxtimes N\in D^b_*(\DD_{X\times Y})$ are defined by
\begin{equation}
M\overset{\mathbb{D}}{\otimes}M^\prime=M\overset{\mathbb{L}}{\underset{\mathcal{O}_X}{\otimes}}M^\prime,\;\; M\boxtimes N=M\underset{\C}{\otimes}N.
\end{equation}
Note that for any objects $N,N^\prime\in D^b_{*}(\DD_Y)$, we have a canonical isomorphism 
\begin{equation}
\mathbb{L}f^*(N\overset{\D}{\otimes}N^\prime)\simeq \mathbb{L}f^*N\overset{\D}{\otimes}\mathbb{L}f^*N^\prime.
\end{equation}
For any objects $M\in D_*^b(\DD_X)$ and $N\in D^b_*(\DD_Y)$, we have the projection formula
\begin{equation}
\int_f\left( M\overset{\D}{\otimes}\mathbb{L}f^*N\right)\simeq\left(\int_fM\right)\overset{\D}{\otimes}N.
\end{equation}

\noindent
Let $Z$ be a smooth closed subvariety of $X$ and let $i:Z\hookrightarrow X$ and $j:X\setminus Z\hookrightarrow X$ be natural inclusions. Then, for any object $M\in D_*^b(\DD_X)$, there is a distinguished triangle
\begin{equation}\label{sdt}
\int_ii^\dagger M\rightarrow M\rightarrow\int_j j^\dagger M\overset{+1}{\rightarrow}.
\end{equation}
If we write $\Gamma_{[Z]}$ for the algebraic local cohomology functor supported on $Z$, we have canonical isomorphisms 
\begin{equation}
\R\Gamma_{[Z]}(\mathcal{O}_X)\overset{\D}{\otimes}M\simeq\R\Gamma_{[Z]}M\simeq\int_ii^\dagger M.
\end{equation}

\noindent
For any (possibly multivalued) function $\varphi$ on $X$ such that $\varphi$ is nowhere-vanishing and that $\frac{d\varphi}{\varphi}$ belongs to $\Omega^1_X(X)$, we define a $\DD_X$-module $\mathcal{O}_X\varphi$ by twisting its action as
\begin{equation}
\theta\cdot h=\Big\{\theta+\Big(\frac{\theta\varphi}{\varphi}\Big)\Big\}h\;\;\;(h\in\mathcal{O}_X,\;\theta\in\Theta_X).
\end{equation}
Here, $\Theta_X$ denotes the sheaf of vector fields on $X$. For any $\DD_X$-module $M,$ we define $M\varphi$ by $M\varphi=M\underset{\mathcal{O}_X}{\otimes}\mathcal{O}_X\varphi.$ We write $\underline{\C}\varphi$ for the local system of flat sections of $\left(\mathcal{O}_X\varphi^{-1}\right)^{an}$ on $X^{an}.$


Now, we are going to prove the isomorphism between Laplace-Gauss-Manin connections associted to Euler-Laplace and Laplace integral. We first prove the following identity which is ``obvious'' from the definition of $\Gamma$ function.

\begin{prop}\label{prop:GammaInt}
Let $h:X\rightarrow\A^1$ be a non-zero regular function such that $h^{-1}(0)$ is smooth, $\pi:X\times (\Gm)_y\rightarrow X$ be the canonical projection, $j:X\setminus h^{-1}(0)\hookrightarrow X$ and $i:h^{-1}(0)\hookrightarrow X$ be inclusions, and let $\gamma\in\C\setminus\Z$ be a parameter. In this setting, for any $M\in D^b_{q.c.}(\DD_X)$, one has a canonical isomorphism

\begin{equation}\label{TheFirstIsomorphism}
\int_{\pi}(\mathbb{L}\pi^*M)y^\gamma e^{yh}\simeq\int_j (j^\dagger M)h^{-\gamma}.
\end{equation}
and a vanishing result
\begin{equation}\label{TheVanishingResult}
\int_{\pi}\R\Gamma_{[h^{-1}(0)\times(\Gm)_y]}(\mathbb{L}\pi^*M)y^\gamma e^{yh}=0.
\end{equation}

\end{prop}
\noindent
For the proof, we insert the following elementary

\begin{lem}\label{lem:gamma}
Let $pt:(\Gm)_y\rightarrow\{ *\}$ be the trivial morphism. If $\gamma\in\C\setminus \Z$ and $h\in\C,$ one has
\begin{equation}
\int_{pt}\mathcal{O}_{(\Gm)_y}y^\gamma e^{hy}=
\begin{cases}
0&(h=0)\\
\C&(h\neq 0).
\end{cases}
\end{equation}
\end{lem}

\begin{proof}
By the formula (\ref{RelativeDeRhamFormula}), we have equalities
\begin{equation}
\int_{pt}\mathcal{O}_{(\Gm)_y}y^\gamma e^{hy}=\Big(\Omega^{\bullet+1}((\Gm)_y), \nabla\Big)=\Big(0\rightarrow \overset{\overset{-1}{\smile}}{\C[y^{\pm}]}\overset{\nabla}{\rightarrow}\overset{\overset{0}{\smile}}{\C[y^\pm]}\rightarrow 0\Big),
\end{equation}
where $\nabla=\frac{\partial}{\partial y}+\frac{\gamma}{y}+h.$ In view of this formula, the lemma is a consequence of an elementary computation.

\end{proof}

\noindent
(Proof of Proposition \ref{prop:GammaInt})

\noindent
By projection formula, we have isomorphisms
\begin{equation}
\int_{\pi}(\mathbb{L}\pi^*M)y^\gamma e^{yh}\simeq M\overset{\D}{\otimes}\int_\pi\mathcal{O}_{X\times (\Gm)_y}y^\gamma e^{yh}
\end{equation}
and

\begin{equation}
\int_j (j^\dagger M)h^{-\gamma}\simeq M\overset{\D}{\otimes }\int_j\mathcal{O}_{X\setminus h^{-1}(0)}h^{-\gamma}.
\end{equation}

\noindent
Therefore, the first isomorphism (\ref{TheFirstIsomorphism}) is reduced to the case when $M=\mathcal{O}_X$. Consider the following cartesian diagram:

\begin{equation}
\xymatrix{
 h^{-1}(0)\times(\Gm)_y \ar[r]^{\tilde{i}} \ar[d]_{\tilde{\pi}}&X\times(\Gm)_y \ar[d]^{\pi}\\
 h^{-1}(0)   \ar[r]^{i}                             &X.}
\end{equation}

\noindent
By base change formula and \cref{lem:gamma}, we have
\begin{equation}
i^\dagger\int_\pi\mathcal{O}_{X\times (\Gm)_y}y^\gamma e^{yh}=\int_{\tilde{\pi}}\tilde{i}^\dagger\mathcal{O}_{X\times (\Gm)_y}y^\gamma e^{yh}=\int_{\tilde{\pi}}\mathcal{O}_{h^{-1}(0)\times (\Gm)_y}y^\gamma [-1]=0.
\end{equation}
Therefore, by the standard distinguished triangle (\ref{sdt}), we have a canonical isomorphism
\begin{equation}\label{eqn:223}
\int_\pi\mathcal{O}_{X\times (\Gm)_y}y^\gamma e^{yh}\simeq\int_jj^\dagger\int_\pi\mathcal{O}_{X\times (\Gm)_y}y^\gamma e^{yh}.
\end{equation}
We are going to compute the right-hand side (\ref{eqn:223}). We consider the following cartesian square:

\begin{equation}
\xymatrix{
 \Big(X\setminus h^{-1}(0)\Big)\times(\Gm)_y \ar[r]^{\;\;\;\;\;\;\tilde{j}} \ar[d]_{\tilde{\pi}^\prime}&X\times(\Gm)_y \ar[d]^{\pi}\\
 X\setminus h^{-1}(0)   \ar[r]^{   \;\;\;\;\;\;   j}                             &X.}
\end{equation}
Again by projection formula, we have
\begin{equation}
j^\dagger\int_\pi\mathcal{O}_{X\times (\Gm)_y}y^\gamma e^{yh}\simeq\int_{\tilde{\pi}^\prime}\tilde{j}^\dagger\mathcal{O}_{X\times (\Gm)_y}y^\gamma e^{yh}.
\end{equation}

\noindent
We consider an isomorphism $\varphi: \Big( X\setminus h^{-1}(0)\Big)\times(\Gm)_y\tilde{\rightarrow}\Big( X\setminus h^{-1}(0)\Big)\times(\Gm)_y$ defined by $\varphi(x,y)=(x,h(x)y).$ Since $\tilde{\pi}^\prime=\tilde{\pi}^\prime\circ\varphi,$ we have

\begin{equation}
\int_{\tilde{\pi}^\prime}\tilde{j}^\dagger\mathcal{O}_{X\times (\Gm)_y}y^\gamma e^{yh}\simeq\int_{\tilde{\pi}^\prime}\int_\varphi\mathcal{O}_{\Big(X\setminus h^{-1}(0)\Big)\times (\Gm)_y}y^\gamma e^{yh}\simeq\int_{\tilde{\pi}^\prime}\mathcal{O}_{X\setminus h^{-1}(0)}h^{-\gamma}\boxtimes \mathcal{O}_{(\Gm)_y}y^\gamma e^{y}\simeq \mathcal{O}_{X\setminus h^{-1}(0)}h^{-\gamma}.
\end{equation}
Thus, the first isomorphism (\ref{TheFirstIsomorphism}) follows. As for the vanishing result (\ref{TheVanishingResult}), we have a sequence of isomorphisms
\begin{align} 
\int_\pi\R\Gamma_{[h^{-1}(0)\times (\Gm)_y]}\left( (\mathbb{L}\pi^*M)y^\gamma e^{yh}\right)&\simeq \int_\pi \int_{\tilde{i}}\tilde{i}^\dagger\left( (\mathbb{L}\pi^*M)y^\gamma e^{yh}\right)\\
 &\simeq \int_{\pi\circ\tilde{i}}\left( \mathbb{L}(\pi\circ\tilde{i})^*My^\gamma\right)[-1]\\
 &\simeq M\overset{\D}{\otimes}\int_{\pi\circ\tilde{i}}\mathcal{O}_{h^{-1}(0)\times (\Gm)_y}y^\gamma[-1]\\
 &\simeq M\overset{\D}{\otimes}\int_{i\circ\tilde{\pi}}\mathcal{O}_{h^{-1}(0)\times (\Gm)_y}y^\gamma[-1]\\
 &\simeq 0.
\end{align}
The last isomorphism is a consequence of Lemma \ref{lem:gamma}.

\qed

\begin{rem}
In the proof above, we have used the following simple fact: Let $f:X\rightarrow X$ be an isomorphism. Then, we have an isomorphism
\begin{equation}\label{Cart}
\int_f\simeq (f^{-1})^{\dagger}=\mathbb{L}(f^{-1})^*.
\end{equation}
Indeed, base change formula applied to the following cartesian diagram gives the isomorphism (\ref{Cart}):
\begin{equation}
\xymatrix{
 X \ar[r]^{{\rm id}_X } \ar[d]_{f^{-1}}&X\ar[d]^{{\rm id}_X}\\
 X \ar[r]^{f}                             &X.
}
\end{equation}
\end{rem}

Repeated applications of the \cref{prop:GammaInt} give the following

\begin{cor}\label{cor:gamma}
Let $h_l:X\rightarrow\A^1$ $(l=1,\cdots, k)$ be non-zero regular functions such that $h_l^{-1}(0)$ are smooth, $\pi:X\times (\Gm)_y^k\rightarrow X$ be the canonical projection, $j:X\setminus \lef h_1\dots h_k=0\righ\hookrightarrow X$ be the inclusion, and let $\gamma_l\in\C\setminus\Z$ be parameters. In this setting, for any object $M\in D^b_{q.c.}(\DD_X)$, one has a natural isomorphism

\begin{equation}
\int_{\pi}(\mathbb{L}\pi^*M)y_1^{\gamma_1}\dots y_k^{\gamma_k} e^{y_1h_1+\dots+y_kh_k}\simeq\int_j(j^\dagger M)h_1^{-\gamma_1}\cdots h_k^{-\gamma_k}.
\end{equation}
\end{cor}

The following theorem proves the equivalence of Laplace integral representation and Euler-Laplace integral representation.
\begin{thm}[Cayley trick for Euler-Laplace integrals]\label{thm:CayleyTrick}
Let $h_{l,z^{(l)}}(x)=\sum_{j=1}^{N_l}z_j^{(l)}x^{{\bf a}^{(l)}(j)}$ $(l=0,1,\dots,k)$ be Laurent polynomials on $(\Gm)^n_x$. We put $N=N_0+\cdots+N_k$, $z=(z^{(0)},\dots,z^{(k)})$, $X_0=\A^N_z\times (\Gm)_x^n\setminus\lef (z,x)\in\A^N\times(\Gm)^n\mid h_{1,z^{(1)}}(x)\cdots h_{k,z^{(k)}}(x)=0\righ$, and $X_k=\A^N_z\times(\Gm)_y^k\times(\Gm)_x^n$. Let $\pi:X_0\rightarrow\A^N_z$ and $\varpi: X_k\rightarrow\A^N_z$ be projections and $\gamma_l\in\C\setminus\Z$ be parameters. Then, one has an isomorphism

\begin{equation}\label{Desired}
\int_\pi\mathcal{O}_{X_0 }e^{h_{0,z^{(0)}}(x)}h_{1,z^{(1)}}(x)^{-\gamma_1}\cdots h_{k,z^{(k)}}(x)^{-\gamma_k}x^c \simeq\int_\varpi\mathcal{O}_{X_k}y^\gamma x^c e^{h_z(y,x)},
\end{equation}
where $h_z(y,x)=h_{0,z^{(0)}}(x)+\displaystyle\sum_{l=1}^ky_lh_{l,z^{(l)}}(x).$
\end{thm}
\begin{proof}
Note first that hypersurfaces $\{ (z,x)\in\A^N\times (\Gm)^n\mid h_{l,z^{(l)}}(x)=0\}\subset\A^N_z\times(\Gm)^n_x$ $(l=1,\dots,k)$ are smooth. Now, consider the following commutative diagram:

\begin{equation}
\xymatrix{
 X_0 \ar[d]_{\pi } \ar[dr]^-{j}& \\
 \A^N_z  & \A^N_z\times(\Gm)_x^n \ar[l]_-{\tilde{\pi}}\\
 X_k   \ar[ur]_-{p}   \ar[u]^-{\varpi}    & .
}
\end{equation}
\noindent
By projection formula,
\begin{align}
 &\int_j\mathcal{O}_{X_0 }h_{1,z^{(1)}}(x)^{-\gamma_1}\cdots h_{k,z^{(k)}}(x)^{-\gamma_k}x^c e^{h_{0,z^{(0)}}(x)}\nonumber\\
 \simeq&\int_j\Big( \mathcal{O}_{X_0 }h_{1,z^{(1)}}(x)^{-\gamma_1}\cdots h_{k,z^{(k)}}(x)^{-\gamma_k}\Big)\overset{\mathbb{D}}{\otimes}\mathcal{O}_{\A^N_z\times(\Gm)^n_x}x^ce^{h_{0,z^{(0)}}(x)}.\label{CIsom1}
\end{align}
\noindent
By \cref{cor:gamma}, we have
\begin{equation}
\int_j\Big( \mathcal{O}_{X_0 }h_{1,z^{(1)}}(x)^{-\gamma_1}\cdots h_{k,z^{(k)}}(x)^{-\gamma_k}\Big)\simeq\int_p\mathcal{O}_{X_k}y^\gamma e^{y_1h_{1,z^{(1)}}(x)+\cdots+y_kh_{k,z^{(k)}}(x)}.
\end{equation}
Again by projection formula, we have
\begin{equation}\Big(\int_p\mathcal{O}_{X_k}y^\gamma e^{y_1h_{1,z^{(1)}}(x)+\cdots+y_kh_{k,z^{(k)}}(x)}\Big)\overset{\mathbb{D}}{\otimes}\mathcal{O}_{\A^N_z\times(\Gm)^n_x}x^ce^{h_{0,z^{(0)}}(x)}\simeq\int_p\mathcal{O}_{X_k}y^\gamma x^c e^{h_z(y,x)}\label{CIsom2}
\end{equation}

\noindent
Since one has canonical isomorphisms 
\begin{equation}
\int_\pi\simeq\int_{\tilde{\pi}}\circ\int_j\;\;\;\;\;\;\;\;\int_\varpi\simeq\int_{\tilde{\pi}}\circ\int_p\;\;\;,
\end{equation}
applying the functor $\int_{\varpi}$ to the left hand side of (\ref{CIsom1}) and to the right hand side of (\ref{CIsom2}) gives the desired isomorphism (\ref{Desired}).
\end{proof}

\begin{cor}\label{cor:Cor2.5}
Under the assumption of \cref{thm:CayleyTrick}, one has a canonical isomorphism

\begin{equation}
\int_{\pi!}\mathcal{O}_{X_0 }e^{h_{0,z^{(0)}}(x)}h_{1,z^{(1)}}(x)^{-\gamma_1}\cdots h_{k,z^{(k)}}(x)^{-\gamma_k}x^c \simeq\int_{\varpi!}\mathcal{O}_{X_k}y^\gamma x^c e^{h_z(y,x)}
\end{equation}

\end{cor}

\begin{proof}
Let $\iota:X_k\rightarrow X_k$ be an involution defined by $\iota(z,y,x)=(z,-y,x)$. Then, we see that $\varpi\circ\iota=\varpi$. This identity implies an equality $\int_\varpi=\int_\varpi\circ\int_\iota$, from which we obtain an identity
\begin{equation}
\int_{\varpi}\mathcal{O}_{X_k}y^\gamma x^c e^{h_z(y,x)}=\int_{\varpi}\mathcal{O}_{X_k}y^\gamma x^c e^{h_z(-y,x)}.
\end{equation}
In view of this identity and two equalities $\HD_{\A^N_z}\circ\int_\pi=\int_{\pi!}\circ\HD_{X_0}$ and $\HD_{\A^N_z}\circ\int_\varpi=\int_{\varpi!}\circ\HD_{X_k}$, we obtain the desired isomorphism by applying $\HD_{\A^N_z}$ to
 (\ref{Desired}) and replace $-\gamma$, $-c$ and $-z^{(0)}$ by $\gamma$, $c$ and $z^{(0)}$.

\end{proof}

\begin{rem}
All the arguments in this subsection can easily be extended to the case when $\gamma_1,\dots,\gamma_k,c_1,\dots,c_n\in {\rm End}_{\C}(V)$ for some finite dimensional $\C$-vector space $V$ and they are mutually commutative. The condition $\gamma_l\notin\Z$ is simply replaced by the condition that any eigenvalue of $\gamma_l$ is non-integral. 

\end{rem}

\subsection{Euler-Laplace integral representation of GKZ system}

Let us refer to the result of M.Schulze and U.Walther (\cite[Corollary 3.8]{SchulzeWaltherLGM}, see also \cite{SchulzeWaltherIrreducibility}) which relates $M_A(c)$ for non-resonant parameters to Laplace-Gauss-Manin connection. It is stated in the following form.

\begin{thm}[\cite{SchulzeWaltherLGM}]
Let $\phi:(\Gm)^d_x\rightarrow\A^n$ be a morphism defined by $\phi(x)=(x^{{\bf a}(1)},\dots,x^{{\bf a}(n)})$. If $c\in\C^{d\times 1}$ is non-resonant, one has a canonical isomorphism 
\begin{equation}\label{SWisom}
M_A(c)\simeq \FL\circ\int_\phi\mathcal{O}_{(\Gm)^d}x^c,
\end{equation}
where $\FL$ stands for Fourier-Laplace transform.
\end{thm}

\noindent
Recall that the parameter $c$ is non-resonant (with respect to $A$) if for any face $\Gamma<\R_{\geq 0}A$ such that $0\in \Gamma$, one has $c\notin\Z^{d\times 1}+\spanning_{\C}\Gamma$.

For readers' convenience, we include a proof of an isomorphism which rewrites the right-hand side of (\ref{SWisom}) as a direct image of an integrable connection. The following result is essentially obtained in \cite{EsterovTakeuchi}.

\begin{prop}\label{prop:EsterovTakeuchi}
Let $f_j\in\mathcal{O}(X)\setminus\C$ $(j=1,\dots,p)$ be non-constant regular functions. Put $f=(f_1,\dots,f_p):X\rightarrow\A^p_\zeta$. Define the Fourier-Laplace transform $\FL:D^b_{q.c.}(\mathcal{D}_{\A^p_\zeta})\rightarrow D^b_{q.c.}(\mathcal{D}_{\A^p_z})$ by the formula 
\begin{equation}
\FL(N)=\int_{\pi_z}(\mathbb{L}\pi^*_\zeta N)\overset{\mathbb{D}}{\otimes}\mathcal{O}_{\A^p_\zeta\times\A^p_z}e^{z\cdot\zeta},
\end{equation}
where $\pi_z:\A^p_z\times\A^p_\zeta\rightarrow\A^p_z$ and $\pi_{\zeta}:\A^p_z\times\A^p_\zeta\rightarrow\A^p_\zeta$ are canonical projections. Let $\pi:X\times\A^p_z\rightarrow\A^p_z$ be the canonical projection. Under these settings, for any object $M\in D^b_{q.c.}(\DD_X),$ one has an isomorphism
\begin{equation}
\FL\Big(\int_f M\Big)\simeq\int_\pi\left\{ (M\boxtimes\mathcal{O}_{\A^p_z})\overset{\mathbb{D}}{\otimes}(\mathcal{O}_{X\times\A^p_z}e^{\sum_{j=1}^pz_jf_j})\right\}.
\end{equation}
\end{prop}
\begin{proof}
Consider the following commutative diagram

\begin{equation}
\xymatrix{
 X\times\A^p_z \ar[d]_{\pi } \ar[r]^{f\times \id}& \A^p_\zeta\times\A^p_z \ar[dl]_{\pi_z}\\
 \A^p_z  &
}.
\end{equation}
\noindent
By the projection formula, we have canonical isomorphisms
\begin{align}
\FL\Big(\int_fM\Big)&\simeq\int_{\pi_z}\left\{ \Big(\left(\int_fM\right)\boxtimes\mathcal{O}_{\A^p_z}\Big)\overset{\mathbb{D}}{\otimes}\mathcal{O}_{\A^p_z\times\A^p_\zeta}e^{z\cdot\zeta }\right\}\\
 &\simeq\int_{\pi_z}\left\{ \Big(\int_{f\times\id_z}M\boxtimes\mathcal{O}_{\A^p_z}\Big)\overset{\mathbb{D}}{\otimes}\mathcal{O}_{\A^p_z\times\A^p_\zeta}e^{z\cdot\zeta }\right\}\\
 &\simeq\int_\pi\left\{ \left(M\boxtimes\mathcal{O}_{\A^p_z}\right)\overset{\mathbb{D}}{\otimes}(\mathcal{O}_{X\times\A^p_z}e^{\sum_{j=1}^pz_jf_j})\right\}.
\end{align}

\end{proof}
\noindent
If we take $X$ to be $(\Gm)^d_x,$ $M$ to be $\mathcal{O}_{(\Gm)^d_x}x^c,$ and $f$ to be $f=(x^{{\bf a}(1)},\dots,x^{{\bf a}(n)}),$ we have
\begin{equation}
\FL\Big(\int_f\mathcal{O}_{(\Gm)^d_x}x^c\Big)\simeq\int_\pi\mathcal{O}_{(\Gm)^d_x\times\A^n_z}x^ce^{h_{z}(x)},
\end{equation}
where $h_z(x)=\sum_{j=1}^nz_jx^{{\bf a}(j)}.$ Therefore, we obtain a

\begin{cor}\label{cor:Cor2.9}
If $c$ is non-resonant, one has a canonical isomorphism
\begin{equation}
M_A(c)\simeq\int_\pi\mathcal{O}_{(\Gm)^d\times\A^n}x^ce^{h_z(x)}.
\end{equation}
\end{cor}

We have a similar result for the Fourier transform of the proper direct image. For the proof, we need a simple

\begin{lem}\label{lem:Duality}

For any objects $M,N\in D^b_{coh}(\DD_X)$, if the inclusion ${\rm Ch}(M)\cap{\rm Ch}(N)\subset T^*_XX$ holds, one has a canonical isomorphism $\D_X(M\overset{\D}{\otimes}N)\simeq \D_X M\overset{\D}{\otimes}\D_X N$ in $D^b_{coh(\mathcal{D}_X)}$.
\end{lem}

\noindent
The proof of this lemma will be given in the appendix.

\begin{prop}\label{prop:Prop2.11}
Under the setting of \cref{prop:EsterovTakeuchi}, for any $M\in D^b_{coh}(\DD_X)$, one has
\begin{equation}
\FL\Big(\int_{f!}M\Big)\simeq\int_{\pi!}\left\{ (M\boxtimes\mathcal{O}_{\A^p_z})\overset{\mathbb{D}}{\otimes}(\mathcal{O}_{X\times\A^p_z}e^{\sum_{j=1}^pz_jf_j})\right\}.
\end{equation}
\end{prop}

\begin{proof}
By \cite[PROPOSITION2.2.3.2.]{Daia}, for any $N\in D^b_{coh}(\DD_{\A^N_\zeta})$, we have a canonical isomorphism ${\rm FL}(N)\simeq\int_{\pi_z!}(\mathbb{L}\pi_\zeta^*N)\overset{\HD}{\otimes}\mathcal{O}_{\A^p_z\times\A_\zeta^p}e^{z\cdot\zeta}.$ We remark that the convention of inverse image functor in \cite{Daia} is different from ours. By \cite[Theorem 2.7.1.]{HTT}, we see that functors $\mathbb{L}\pi^*_\zeta$ and $\HD$ commute. Therefore, by \cref{lem:Duality}, we have 
\begin{align}
\FL\Big(\int_{f!}M\Big)&\simeq\int_{\pi_z!}\left(\mathbb{L}\pi_\zeta^*\left( \int_{f!}M\right)\right)\overset{\HD}{\otimes}\mathcal{O}_{\A^p_z\times\A_\zeta^p}e^{z\cdot\zeta}\\
 &\simeq \HD\circ\int_{\pi_z}\left(\mathbb{L}\pi_\zeta^*\left( \int_{f}\HD M\right)\right)\overset{\HD}{\otimes}\mathcal{O}_{\A^p_z\times\A_\zeta^p}e^{-z\cdot\zeta}\\
 &\overset{\cref{lem:Duality}}{\simeq} \HD\circ\int_{\pi_z}\left(\D \mathbb{L}\pi_\zeta^*\left( \int_{f!} M\right)\right)\overset{\HD}{\otimes}\mathcal{O}_{\A^p_z\times\A_\zeta^p}e^{-z\cdot\zeta}\\
 &\simeq\HD\circ\int_{\pi_z}\left\{ ((\HD M)\boxtimes\mathcal{O}_{\A^p_z})\overset{\mathbb{D}}{\otimes}(\mathcal{O}_{X\times\A^p_z}e^{-\sum_{j=1}^pz_jf_j})\right\}\\
 &\simeq \HD\circ\int_{\pi_z}\left\{ \HD (M\boxtimes\mathcal{O}_{\A^p_z})\overset{\mathbb{D}}{\otimes}(\mathcal{O}_{X\times\A^p_z}e^{-\sum_{j=1}^pz_jf_j})\right\}\\
 &\overset{\cref{lem:Duality}}{\simeq} \int_{\pi_z!}\left\{ (M\boxtimes\mathcal{O}_{\A^p_z})\overset{\mathbb{D}}{\otimes}(\mathcal{O}_{X\times\A^p_z}e^{\sum_{j=1}^pz_jf_j})\right\}.
\end{align}

\end{proof}

\noindent
Now, we use the same notation as \cref{thm:CayleyTrick}. We put 
\begin{equation}
\Phi=\Phi(z,x)=e^{h_{0,z^{(0)}}(x)}h_{1,z^{(1)}}(x)^{-\gamma_1}\cdots h_{k,z^{(k)}}(x)^{-\gamma_k}x^c,\;\; \Phi_k=y^\gamma x^c e^{h_z(y,x)}
\end{equation}
to simplify the notation. Let us formulate the main theorem of this section. We put $N=N_0+N_1+\cdots+N_k$ and define an $n\times N_l$ matrix $A_l$ by $A_l=({\bf a}^{(l)}(1)\mid\cdots\mid {\bf a}^{(l)}(N_l))$. Then, we define the Cayley configuration $A$ as an $(n+k)\times N$ matrix by
\begin{equation}\label{CayleyConfig}
A
=
\left(
\begin{array}{ccc|ccc|ccc|c|ccc}
0&\cdots&0&1&\cdots&1&0&\cdots&0&\cdots&0&\cdots&0\\
\hline
0&\cdots&0&0&\cdots&0&1&\cdots&1&\cdots&0&\cdots&0\\
\hline
&\vdots& & &\vdots& & &\vdots& &\ddots& &\vdots& \\
\hline
0&\cdots&0&0&\cdots&0&0&\cdots&0&\cdots&1&\cdots&1\\
\hline
&A_0& & &A_1& & &A_2& &\cdots & &A_k& 
\end{array}
\right).
\end{equation}
We define a morphism $j_A:(\Gm)^k_y\times (\Gm)^n_x\rightarrow\A^N_z$ by $j_A(y,x)=(y,x)^A$. In view of the proof of \cite[LEMMA 4.2]{EsterovTakeuchi}, one has a canonical isomorphism $\int_{j_A!}\mathcal{O}_{(\Gm)^k_y\times (\Gm)^n_x}y^\gamma x^c\overset{\sim}{\rightarrow}\int_{j_A}\mathcal{O}_{(\Gm)^k_y\times (\Gm)^n_x}y^\gamma x^c$. Combining \cref{thm:CayleyTrick}, \cref{cor:Cor2.5}, \cref{cor:Cor2.9}, and \cref{prop:Prop2.11}, we have the following main result of this section.

\begin{thm}\label{thm:mainDresult}
Suppose that the parameter $\delta={}^t(\gamma_1,\cdots,\gamma_k,c)\in\C^{(n+k)\times 1}$ is non-resonant and $\gamma_l\notin\Z$ for $l=1,\dots,k$. Then, one has a sequence of canonical isomorphisms of $\DD_{\A^{N}_z}$-modules
\begin{equation}\label{EulerLaplaceGaussManin}
M_A(\delta)\simeq\int_\varpi\mathcal{O}_{X_k}\Phi_k\simeq\int_\pi\mathcal{O}_{X_0 }\Phi.
\end{equation}

\noindent
Moreover, one has canonical isomorphisms
\begin{equation}\label{RegularizationCondition}
\int_\varpi\mathcal{O}_{X_k}\Phi_k\simeq\int_{\varpi!}\mathcal{O}_{X_k}\Phi_k
\;\;\;\text{  and  }\;\;\;
\int_\pi\mathcal{O}_{X_0}\Phi\simeq\int_{\pi!}\mathcal{O}_{X_0 }\Phi.
\end{equation}
The isomorphism (\ref{RegularizationCondition}) is called the regularization condition.
\end{thm}

\begin{rem}
In general, non-resonance of the parameter $\delta$ does not imply $\gamma_l\notin\Z$. The simplest example is $A=\begin{pmatrix}1&1\\ 0&1\end{pmatrix}$. If $\gamma_l\in \Z$ for some $l$ and $A_0=\varnothing$, the regularization condition (\ref{RegularizationCondition}) is no longer true. Indeed, if (\ref{RegularizationCondition}) is true, the analytic de Rham functor applied to (\ref{RegularizationCondition}) shows that the canonical morphism ${\rm can}:\Homo_1(\pi^{-1}(z)^{an};\C\Phi)\rightarrow\Homo_1^{lf}(\pi^{-1}(z)^{an};\C\Phi)$ is an isomorphism. Let $C_l$ be a $1$-dimensional cycle which turns around the divisor $\{ h_l(x)=0\}\subset(\Gm)^n$ once. $C_l$ defines a non-zero homology class $[C_l]$ since we have $\frac{1}{2\pi\ii}\int_{C_l}d_x\log h_l=\pm1$. On the other hand, it is easy to see that the class ${\rm can}[C_l]$ is homologous to $0$, which is a contradiction. The condition $\gamma_l\notin\Z$ is missing in \cite[Theorem 2.15]{GKZEuler}.
\end{rem}

Let $\Delta_A$ be the convex hull of column vectors of $A$ and the origin in $\R^{(n+k)\times 1}$. Recall that a point $z\in\A^{N}$ is said to be Newton non-degenerate if for any face $\Gamma$ of $\Delta_A$ which does not contain the origin, the set $\{ (y,x)\in(\Gm)^k\times (\Gm)^n\mid d_{(y,x)}h^\Gamma_z(y,x)=0\}$ is empty. Here, $h^\Gamma_z(y,x)$ denotes the Laurent polynomial associated to the face $\Gamma$ (before LEMMA 3.3 of \cite{Adolphson}). We set $d_x+\sum_{i=1}^nc_i\frac{dx_i}{x_i}\wedge-\sum_{l=1}^k\gamma_l\frac{d_xh_{l,z^{(l)}}(x)}{h_{l,z^{(l)}}(x)}\wedge+d_xh_{0,z^{(0)}}(x)\wedge$. The following result is also obtained in \cite{AdolphsonSperber}

\begin{cor}
Under the same assumption of \cref{thm:mainDresult}, if $z$ is a Newton non-degenerate point, the algebraic de Rham cohomology group $\Homo^*_{dR}\left( \pi^{-1}(z);(\mathcal{O}_{\pi^{-1}(z)},\nabla_z)\right)$ is purely $n$-codimensional.
\end{cor}

\begin{proof}
We write $\iota_z:\{ z\}\hookrightarrow\A^N$ for the canonical inclusion. Since $\iota_z$ is non-characteristic with respect to the $\DD$-module $M_A(\delta)$ (\cite[LEMMA 3.3]{Adolphson}), the complex $\mathbb{L}\iota_z^*M_A(\delta)$ is concentrated in degree $0$. On the other hand, we have $\mathbb{L}\iota_z^*\int_\pi\mathcal{O}_{X_0 }\Phi=DR_{\pi^{-1}(z)/\{ z\}}(\mathcal{O}_{\pi^{-1}(z)},\nabla_z)$ by the projection formula. Therefore, the result follows from the isomorphism (\ref{EulerLaplaceGaussManin}).

\end{proof}

\section{Description of the rapid decay homology groups of Euler-Laplace integrals}\label{SectionRapidDecay}
\subsection{A concrete version of Theorem \ref{thm:mainDresult}}\label{subsec:3.1}
We inherit the notation of \S\ref{subsec:2.1}. We begin with proving an explicit version of \cref{thm:mainDresult}. Let $Y$ be a smooth product variety $Y=X\times Z$, $X$ be affine and let $M=(E,\nabla)$ be a (meromorphic) integrable connection on $Y$. We write $\pi_Z:Y\rightarrow Z$ for the canonical projection. We revise the explicit $\DD_Z$-module structure of $\int_{\pi_Z}M$. We can assume that $Z$ is affine since the argument is local. From the product structure of $Y$, we can naturally define a decomposition $\Omega_{Y}^1(E)=\Omega^1_{Y/X}(E)\oplus\Omega^1_{Y/Z}(E)$. Here, $\Omega^1_{Y/X}(E)$ and $\Omega^1_{Y/Z}(E)$ are the sheaves of relative differential forms with values in $E$. By taking a local frame of $E$, we see that $\nabla$ can locally be expressed as $\nabla=d+\Omega\wedge$ where $\Omega\in\Omega^1(\End (E))$. We see that $\Omega$ can be decomposed into $\Omega=\Omega_x+\Omega_z$ with $\Omega_x\in \Omega^1_{Y/Z}(\End (E))$ and $\Omega_z\in \Omega^1_{Y/X}(\End (E))$. Then, $\nabla_{Y/Z}=d_x+\Omega_x\wedge$ and $\nabla_{Y/X}=d_z+\Omega_z\wedge$ are both globally well-defined and we have $\nabla=\nabla_{Y/X}+\nabla_{Y/Z}$. Here, $\nabla_{Y/X}:\mathcal{O}_{Y}(E)\rightarrow\Omega^1_{Y/X}(E)$ and $\nabla_{Y/Z}:\mathcal{O}_{Y}(E)\rightarrow\Omega^1_{Y/Z}(E)$. Note that the integrability condition $\nabla^2=0$ is equivalent to three conditions $\nabla_{Y/X}^2=0, \nabla_{Y/Z}^2=0,$ and $\nabla_{Y/X}\circ\nabla_{Y/Z}+\nabla_{Y/Z}\circ\nabla_{Y/X}=0$. For any (local algebraic) vector field $\theta$ on $Z$ and for any form $\omega\in\Omega_{Y/Z}^*(E)$, we define the action $\theta\cdot \omega$ by $\theta\cdot \omega=\iota_\theta(\nabla_{Y/X}\omega)$, where $\iota_\theta$ is the interior derivative. In this way, ${\rm DR}_{Y/Z}(E,\nabla)=(\Omega^{\dim X+*}_{Y/Z}(E),\nabla_{Y/Z})$ is a complex of $\DD_Z$-modules. It can be shown that ${\rm DR}_{Y/Z}(E,\nabla)$ represents $\int_{\pi_Z}M$ (\cite[pp.45-46]{HTT}). 

For any non-constant regular function $h$ on $Y$ such that $h^{-1}(0)$ is smooth and a parameter $\gamma\in\C\setminus\Z$, we are going to give an explicit version of the isomorphism 
\begin{equation}
\int^0_{\pi_Z\circ\pi}(\mathbb{L}\pi^*M)y^\gamma e^{yh(x,z)}\simeq \int^0_{\pi_Z}\int_j(j^\dagger M)h^{-\gamma},
\end{equation}
where $\pi:Y\times(\Gm)_y\rightarrow Y$ is the canonical projection, $j:Y\setminus h^{-1}(0)\rightarrow Y$ is the canonical inclusion, and $\int^0_*$ is the $0$-th cohomology group of $\int_*$. We denote $(E_1, \nabla_1)$ the integrable connection $(\mathbb{L}\pi^*M)y^\gamma e^{yh}$. We set $D=h^{-1}(0)\times (\Gm)_y$ and consider a short exact sequence of complexes of $\DD_Z$-modules
\begin{equation}
0\rightarrow{\rm DR}_{Y\times(\Gm)_y/Z}(E_1, \nabla_1)\rightarrow
{\rm DR}_{Y\times(\Gm)_y/Z}((E_1, \nabla_1)(*D))
\rightarrow
\frac{
{\rm DR}_{Y\times(\Gm)_y/Z}((E_1, \nabla_1)(*D))
}
{
{\rm DR}_{Y\times(\Gm)_y/Z}(E_1, \nabla_1)
}
\rightarrow 0.
\end{equation}
Here, the first and the second morphism are natural inclusion and projection respectively. Since the third complex is quasi-isomorphic to $\int_{\pi_Z\circ\pi}\R\Gamma_{[D]}((\mathbb{L}\pi^*M)y^\gamma e^{yh})$, this is quasi-isomorphic to $0$ by (\ref{TheVanishingResult}). 



Now, we consider an isomorphism $\varphi:(Y\setminus h^{-1}(0))\times(\Gm)_y\rightarrow (Y\setminus h^{-1}(0))\times(\Gm)_y$ defined by $\varphi(z,x,y)=(z,x,\frac{y}{h(z,x)})$. For any $\omega\in \Omega^p_{Y\times(\Gm)_y/Z}(E_1)(*D)$, we define $\varphi_z^*\omega$ to be the pull-back of $\omega$ by $\varphi$ freezing the variable $z$. More precisely, we consider the decomposition $\Omega_{Y\times(\Gm)_y}^p=\Omega_{Y\times(\Gm)_y/Z}^p\oplus\Omega^1_Z\wedge\Omega^{p-1}_{Y\times(\Gm)_y}$. Then, $\varphi_z^*\omega$ is defined to be the projection of $\varphi^*\omega$ to the component $\Omega_{Y\times(\Gm)_y/Z}^p$. We put $(E_2,\nabla_2)$ to be the meromorphic integrable connection $\left(\int_j (j^\dagger M)h^{-\gamma}\right)\boxtimes \mathcal{O}_{(\Gm)_y}y^\gamma e^y.$ By a direct computation, we can verify that $\varphi_z^*$ induces a $\C$-linear isomorphism of complexes $\varphi_z^*:{\rm DR}_{Y\times(\Gm)_y/Z}((E_1, \nabla_1)(*D))\rightarrow {\rm DR}_{Y\times(\Gm)_y/Z}(E_2, \nabla_2)$. However, this is not a morphism of $\DD_Z$-modules. Nonetheless, we can prove the following 

\begin{prop}\label{prop:Prop312}

\begin{equation}
\Homo^{0}(\varphi_z^*):\Homo^{0}({\rm DR}_{Y\times(\Gm)_y/Z}((E_1, \nabla_1)(*D)))\rightarrow \Homo^{0}({\rm DR}_{Y\times(\Gm)_y/Z}((E_2, \nabla_2))
\end{equation}
 is an isomorphism of $\DD_Z$-modules.
\end{prop}

\begin{proof}
Remember that the connection $(E,\nabla)$ can locally be expressed as $\nabla=d+\Omega\wedge=d_x+\Omega_x\wedge+d_z+\Omega_z\wedge$. Therefore, we locally have $\nabla_1=\nabla+\gamma\frac{dy}{y}\wedge +d(yh)\wedge=(d_{x,y}+\Omega_x\wedge+\gamma\frac{dy}{y}\wedge+h{dy}\wedge +yd_xh\wedge)+(d_z+\Omega_z\wedge+yd_zh\wedge)$ and 
$\nabla_2=\nabla-\gamma\frac{dh}{h}\wedge+\gamma\frac{dy}{y}\wedge+dy\wedge=(d_{x,y}+\Omega_x\wedge-\gamma\frac{d_xh}{h}\wedge+\gamma\frac{dy}{y}\wedge+dy\wedge)+(d_z+\Omega_z-\gamma\frac{d_zh}{h}\wedge)$.

Let us take any element $\xi\in {\rm DR}^{\dim X+1}_{Y\times(\Gm)_y/Z}((E_1, \nabla_1)(*D))$. By definition, $\xi$ can be written in the form $\xi=a(z,x,y)\frac{dy}{y}\wedge\omega(x)$ where $\omega(x)\in\Omega^{\dim X}_{Y/Z}(E)$ and $a(z,x,y)$ is a  regular function on $Y\times(\Gm)_y$ having poles along $h^{-1}(0)$. In the following we fix a vector field $\theta$ on $Z$ and compute its actions to $\xi$ and $\varphi^*_z\xi$. In order to emphasize that the actions are different, we write the resulting elements as $\theta\overset{(1)}{\bullet}\xi$ and $\theta\overset{(2)}{\bullet}(\varphi^*_z\xi)$. Firstly, we have an equality
\begin{equation}\label{EquationAbove}
\theta\overset{(1)}{\bullet}\xi=(\theta a)(z,x,y)\frac{dy}{y}\wedge\omega(x)+\Omega_z(\theta)\xi+y(\theta h)(z,x)\xi.
\end{equation}
Applying $\varphi_z^*$ to (\ref{EquationAbove}), we have
\begin{align}
\varphi_z^*(\theta\overset{(1)}{\bullet}\xi)=&(\theta a)(z,x,\frac{y}{h(z,x)})\frac{dy}{y}\wedge\omega(x)+\Omega_z(\theta)a(z,x,\frac{y}{h(z,x)})\frac{dy}{y}\wedge\omega(x)\nonumber\\
&+\frac{y}{h(z,x)}(\theta h)(z,x)a(z,x,\frac{y}{h(z,x)})\frac{dy}{y}\wedge\omega(x).
\end{align}

\noindent
Secondly, by a direct computation, we have an equality 
\begin{equation}
\theta\overset{(2)}{\bullet}(\varphi^*_z\xi)=(\theta a)(z,x,\frac{y}{h(z,x)})\omega(x)\wedge\frac{dy}{y}-\frac{y (\theta h)(z,x)}{h(z,x)^2}a_y(z,x,\frac{y}{h(z,x)})\omega(x)\wedge\frac{dy}{y}+\Omega_z(\theta)\varphi_z^*\xi-\gamma\frac{(\theta h)(z,x)}{h(z,x)}\varphi_z^*\xi,
\end{equation}

\noindent
where $a_y$ is the partial derivative of $a$ in $y$. Finally, we have an equality
\begin{align}
& (\nabla_2)_{Y\times(\Gm)_y/Z}\left(\frac{(\theta h)(z,x)}{h(z,x)}a(z,x,\frac{y}{h(z,x)})\omega(x)\right)\nonumber\\
=&\frac{y(\theta h)(z,x)}{h(z,x)^2}a_y(z,x,\frac{y}{h(z,x)})\frac{dy}{y}\wedge\omega(x)+\gamma\frac{(\theta h)(z,x)}{h(z,x)}a(z,x,\frac{y}{h(z,x)})\frac{dy}{y}\wedge\omega(x)\nonumber\\
 &+\frac{y (\theta h)(z,x)}{h(z,x)}a(z,x,\frac{y}{h(z,x)})\frac{dy}{y}\wedge\omega(x),
\end{align}
from which we obtain a relation
\begin{equation}
\varphi_z^*(\theta\overset{(1)}{\bullet}\xi)+(\nabla_2)_{Y\times(\Gm)_y/Z}\left(\frac{(\theta h)(z,x)}{h(z,x)}a(z,x,\frac{y}{h(z,x)})\omega(x)\right)=\theta\overset{(2)}{\bullet}(\varphi^*_z\xi).
\end{equation}
Taking the cohomology groups, we can conclude that $\varphi^*_z$ is a morphism of $\DD_Z$-modules.

\end{proof}

We write $(E_3,\nabla_3)$ for the meromorphic connection $\int_j(j^\dagger M)h^{-\gamma}$. The relative de Rham complex for $\int_{\pi_Z}(E_3,\nabla_3)$ is explicitly given by the formula ${\rm DR}_{Y/Z}(E_3,\nabla_3)=\left( \Omega^{\dim X+*}_{Y/Z}(*h^{-1}(0)),\nabla-\gamma\frac{d h}{h}\wedge\right).$ 

\begin{prop}\label{prop:TrivialProp}
Wedge product induces an isomorphism of complexes of $\DD_Z$-modules
\begin{equation}
{\rm DR}_{(\Gm)_y/{\rm pt}}\left(\mathcal{O}_{(\Gm)_y},d_y+\gamma\frac{dy}{y}\wedge+dy\wedge\right)\boxtimes {\rm DR}_{Y/Z}(E_3,\nabla_3) \overset{\sim}{\rightarrow} {\rm DR}_{Y\times(\Gm)_y/Z}(E_2, \nabla_2).
\end{equation}
\end{prop}

\noindent
The proof of the \cref{prop:TrivialProp} is straightforward. Therefore, in view of \cref{lem:gamma}, we have a quasi-isomorphism of complexes of $\DD_Z$-modules ${\rm DR}_{Y/Z}(E_3,\nabla_3) \overset{\sim}{\rightarrow} {\rm DR}_{Y\times(\Gm)_y/Z}(E_2, \nabla_2)$ which sends any relative $p$-form $\xi\in\Omega^p_{Y/Z}(E)(*h^{-1}(0))$ to $\frac{dy}{y}\wedge\xi$.

Now, we apply the construction above to Euler-Laplace integral representation. For given Laurent polynomials $h_{l,z^{(l)}}(x)$ ($l=0,1,\dots,k$), we put $D_l=\{ h_{l,z^{(l)}}(x)=0\}\subset X_0$. Then, $\int_{\pi}\mathcal{O}_{X_{0} }\Phi$ is isomorphic to the complex 
\begin{equation}
{\rm DR}_{\A^N_z\times (\Gm)^n_x/\A^N_z}\left( \mathcal{O}_{\A^N_z\times (\Gm)^n_x}\left(*\left(\sum_{l=1}^kD_l\right)\right),d+\sum_{i=1}^nc_i\frac{dx_i}{x_i}\wedge-\sum_{l=1}^k\gamma_l\frac{dh_{l,z^{(l)}}(x)}{h_{l,z^{(l)}}(x)}\wedge+dh_{0,z^{(0)}}(x)\wedge\right).
\end{equation} 

\noindent
In the same way, $\int_{\varpi}\mathcal{O}_{X_{k} }\Phi_k$ is isomorphic to the complex 
\begin{equation}
{\rm DR}_{\A^N_z\times (\Gm)^n_x\times (\Gm)^k_y/\A^N_z}\left( \mathcal{O}_{X_k},d+\sum_{i=1}^nc_i\frac{dx_i}{x_i}\wedge+\sum_{l=1}^k\gamma_l\frac{dy_l}{y_l}\wedge+dh_{z}(y,x)\wedge\right). 
\end{equation}

\noindent
We set $\frac{dx}{x}=\frac{dx_1\wedge\dots\wedge dx_n}{x_1\dots x_n}$ for brevity. Applying \cref{prop:Prop312} and \cref{prop:TrivialProp} repeatedly, we obtain a

\begin{thm}\label{thm:CyclicGenerator}
There is an isomorphism 
\begin{equation}
\int_{\pi}^0\mathcal{O}_{X_{0} }\Phi
\rightarrow
\int_{\varpi}^0\mathcal{O}_{X_{k} }\Phi_k,
\end{equation}
of $\DD_{\A^N_z}$-modules which sends $[\frac{dx}{x}]$ to $[\frac{dy}{y}\wedge\frac{dx}{x}]$.
\end{thm}

\noindent
Let $A$ be as in (\ref{CayleyConfig}). We obtain an explicit version of Theorem \ref{thm:mainDresult}.

\begin{cor}\label{cor:Corollary216}
If the parameter $\delta$ is non-resonant and $\gamma_l\notin\Z$ for any $l=1,\dots,k$, $M_A(\delta)\ni[1]\mapsto [\frac{dx}{x}]\in \int_{\pi}^0\mathcal{O}_{X_{0} }\Phi$ defines an isomorphism of $\DD_{\A_z^N}$-modules.
\end{cor}

\begin{proof}
In \cite[Lemma 4.7]{EsterovTakeuchi}, it was proved that $[\frac{dy}{y}\wedge\frac{dx}{x}]$ is a cyclic generator (Gauss-Manin vector) of $\int_{\pi_{k}}^0\mathcal{O}_{X_{k} }\Phi_k$. Therefore, by \cref{thm:CyclicGenerator}, $[\frac{dx}{x}]$ is a cyclic generator of $\int_{\pi}^0\mathcal{O}_{X_{0} }\Phi$. On the other hand, it can easily be proved that $M_A(\delta)\ni[1]\mapsto [\frac{dx}{x}]\in \int_{\pi}^0\mathcal{O}_{X_{0} }\Phi$ defines a morphism of $\DD_{\A_z^N}$-modules. When the parameter $d$ is non-resonant, this is an isomorphim since $M_A(\delta)$ is irreducible by \cite{SchulzeWaltherIrreducibility}.
\end{proof}

\subsection{A description of the rapid decay homology group}\label{subsec:3.2}
Now we discuss the solutions of Laplace-Gauss-Manin connection $\int_{\pi}\mathcal{O}_{X_{0} }\Phi$. For the convenience of the reader we repeat the relevant material from \cite{EsterovTakeuchi} and \cite{HienRDHomology} without proofs, thus making our exposition self-contained. Let $U$ be a smooth complex affine variety, let $f:U\rightarrow \A^1$ be a non-constant morphism, and let $M=(E,\nabla)$ be a regular integrable connection on $U$. We consider an embedding of $U$ into a smooth projective variety $X$ with a meromorphic prolongation $f:X\rightarrow\mathbb{P}^1$. We assume that $D=X\setminus U$ is a normal crossing divisor. We decompose $D$ as $D=f^{-1}(\infty)\cup D_{irr}$. Then, we write $\widetilde{X^D}=\widetilde{X}$ for the real oriented blow-up of $X$ along $D$ and write $\varpi:\widetilde{X}\rightarrow X$ for the associated morphism (\cite[\S 8.2]{SabbahIntroductionToStokes}). Let $\widetilde{\PP^1}$ denote the real oriented blow-up of $\PP^1$ at infinity and $\varpi_\infty:\widetilde{\PP^1}\rightarrow\PP^1$ the associated morphism. Note that the closure of the ray $[0,\infty)e^{\ii\theta}$ in $\widetilde{\PP^1}$ intersects with $\widetilde{\PP^1}\setminus\C$ at a unique point which is denoted by $e^{\ii\theta}\infty$. Now, a morphism $\tilde{f}:\widetilde{X}\rightarrow \widetilde{\PP^1}$ is naturally induced so that it fits into a commutative diagram
\begin{equation}
\xymatrix{
 \widetilde{X} \ar[r]^{\tilde{f} } \ar[d]_{\varpi}&\widetilde{\PP^1}\ar[d]^{\varpi_\infty}\\
 X \ar[r]^{f}                             &\PP^1.
}
\end{equation}

\noindent
We set $\widetilde{D^{r.d.}}=\tilde{f}^{-1}\left( \{e^{\ii\theta}\infty\mid \theta\in(\frac{\pi}{2},\frac{3\pi}{2})\}\right)\setminus\varpi^{-1}({D_{irr}})\subset\widetilde{X}$. Let $\mathcal{L}$ be the dual local system of $\Ker \left(\nabla^{an}:\mathcal{O}_{X^{an}}(E^{an})\rightarrow\Omega_{X^{an}}^1(E^{an})\right)$. We consider the natural inclusion $U^{an}\overset{j}{\hookrightarrow} U^{an}\cup \widetilde{D^{r.d.}}$. Then, the rapid decay homology group of M.Hien $\Homo^{r.d.}_*\left(U^{an},(E^\vee,\nabla^\vee)\right)$ is defined in this setting by 
\begin{equation}
\Homo^{r.d.}_*\left(U^{an},(Me^f)^\vee\right)=\Homo_*\left(U^{an}\cup \widetilde{D^{r.d.}},\widetilde{D^{r.d.}};j_*\mathcal{L}\right)
\end{equation}
(\cite{HienRDHomology}, see also \cite{EsterovTakeuchi} and \cite{MatsubaraRapidDecay}). Note that $U^{an}\cup \widetilde{D^{r.d.}}$ is a topological manifold with boundary and that $j_*\mathcal{L}$ is a local system on $U^{an}\cup \widetilde{D^{r.d.}}$. The main result of \cite{HienRDHomology} states that the period pairing $\Homo^{r.d.}_*\left(U^{an},(Me^f)^\vee\right)\times \Homo_{\rm dR}^{*}(U,Me^f)\rightarrow \C$ is perfect.

\begin{rem}\label{rem:RemarkRDHomology}
We put $\widetilde{D^{r.d.}_0}=\tilde{f}^{-1}\left( \{ e^{\ii\theta}\infty\mid\theta\in (\frac{\pi}{2},\frac{3\pi}{2})\}\right)$ and denote by $\bar{j}$ the natural inclusion $U^{an}\hookrightarrow U^{an}\cup \widetilde{D^{r.d.}_0}$. It can easily be seen that the inclusion $\left( U^{an}\cup \widetilde{D^{r.d.}}, \widetilde{D^{r.d.}}\right)\hookrightarrow \left( U^{an}\cup\widetilde{D_0^{r.d.}},\widetilde{D_0^{r.d.}}\right)$ is a homotopy equivalence (\cite[Lemma 2.3]{MatsubaraRapidDecay}). Therefore, the rapid decay homology group can be computed by the formula $\Homo^{r.d.}_*\left(U^{an},(E^\vee,\nabla^\vee)\right)=\Homo_*\left(U^{an}\cup \widetilde{D_0^{r.d.}},\widetilde{D_0^{r.d.}};\bar{j}_*\mathcal{L}\right).$ Note that this realization is compatible with the period pairing.
\end{rem}

\begin{rem}
The formulation of \cite{HienRoucairol} is not suitable in our setting. In their formulation, $\widetilde{X}$ is taken to be the fiber product $X\underset{\PP^1}{\times}\widetilde{\PP^1}$. However, the corresponding embedding $j:U^{an}\hookrightarrow U^{an}\cup \widetilde{D^{r.d.}}$ may have higher cohomology groups $R^pj_*\underline{\C}_{U^{an}}$. None the less, under a suitable genericity condition of eigenvalues of monodromies of $\mathcal{L}$, we can recover the vanishing of higher direct images $R^pj_*\mathcal{L}$. We do not discuss this aspect in this paper.
\end{rem}

We construct a family of good compactifications $X$ associated to the Laplace-Gauss-Manin connection $\int_{\pi}^0\mathcal{O}_{X_{0} }\Phi.$ Let us use the notation of Theorem \ref{thm:CayleyTrick}. First, we put $\Delta_0=\text{\rm convex hull}\{ 0,{\bf a}^{(0)}(1),\dots,{\bf a}^{(0)}(N_0)\}$ and $\Delta_l=\text{\rm convex hull}\{ {\bf a}^{(l)}(1),\dots,{\bf a}^{(l)}(N_l)\}$ ($l=1,\dots,k$). For any covector $\xi\in (\R^n)^*$, we set $\Delta^\xi_l=\{ v\in \Delta_l\mid \langle \xi,v\rangle=\underset{w\in\Delta_l}{\rm min}\langle \xi,w\rangle\}$ and $h^\xi_{l,{z}^{(l)}}(x)=\sum_{{\bf a}(j)\in\Delta^\xi_l}z_jx^{{\bf a}^{(l)}(j)}$. Here, $\langle\bullet,\bullet\rangle$ denotes the duality pairing. Now, we consider the dual fan $\Sigma$ of the Minkowski sum $\Delta_0+\Delta_1+\dots+\Delta_k$. By taking a refinement if necessary, we may assume that $\Sigma$ is a smooth fan. Then, the associated toric variety $X=X(\Sigma)$ is sufficiently full for any $\Delta_l$ in the sense of \cite{KhovanskiNewton}. We write $\{ D_j\}_{j\in J}$ for the set of torus invariant divisors of $X$. 
\begin{defn}\label{defn:Nonsingularity}
We say that a point $z=(z^{(0)},z^{(1)},\dots,z^{(k)})\in\C^N$ is nonsingular if the following two conditions are both satisfied:
\begin{enumerate}
\item For any $1\leq l_1<\dots <l_s\leq k$, the Laurent polynomials $h_{l_1,z^{(l_1)}}(x),\dots,h_{l_s,z^{(l_s)}}(x)$ are nonsingular in the sense of \cite{KhovanskiNewton}, i.e., for any covector $\xi\in (\R^n)^*$, the $s$-form $d_xh^\xi_{l_1,z^{(l_1)}}(x)\wedge\dots\wedge d_xh^\xi_{l_s,z^{(l_s)}}(x)$ never vanishes on the set $\{ x\in(\C^\times)^n\mid h^\xi_{l_1,z^{(l_1)}}(x)=\dots=h^\xi_{l_s,z^{(l_s)}}(x)=0\}$.
\item For any covector $\xi\in (\R^n)^*$ such that $0\notin\Delta_0^\xi$ and for any $1\leq l_1<\dots <l_s\leq k$ ($s$ can be $0$), the $s+1$-form $dh^\xi_{0,z^{(0)}}(x)\wedge dh^\xi_{l_1,z^{(l_1)}}(x)\wedge\dots\wedge dh^\xi_{l_s,z^{(l_s)}}(x)$ never vanishes on the set $\{ x\in(\C^\times)^n\mid h^\xi_{l_1,z^{(l_1)}}(x)=\dots=h^\xi_{l_s,z^{(l_s)}}(x)=0\}$.
\end{enumerate}
\end{defn}

\begin{prop}
The set of nonsingular points is Zariski open and dense.
\end{prop}

\begin{proof}
We say $z\in\A^N$ is singular if it is not nonsingular. We prove that the set $Z\overset{def}{=}\{ z\in\A^N\mid z\text{ is singular}\}\subset\A^N$ is Zariski closed. For this purpose, it is enough to prove that there is a Zariski closed subset $\widetilde{Z}\subset\A^N\times X$ such that $\pi_{\A^N}(\widetilde{Z})=Z$, where $\pi_{\A^N}:\A^N\times X\rightarrow\A^N$ is the canonical projection. Indeed, since $\Sigma$ is a complete fan, $X\rightarrow pt$ is a proper morphism, its base change $\pi_{\A^N}$ is also a closed morphism. We consider the case when the condition 1 of \cref{defn:Nonsingularity} fails. We take a maximal cone $\tau\in\Sigma$. Since $\Sigma$ is taken to be smooth, there are exactly $n$ primitive vectors $\kappa_1,\dots,\kappa_n\in\Z^{n\times 1}\setminus \{0\}$ such that $\tau\cap\Z^{n\times 1}=\Z_{\geq 0}\kappa_1+\dots+\Z_{\geq 0}\kappa_n$. We set $m_i^{(l)}=\underset{a\in\Delta_l}{\rm min}\langle\kappa_i,a\rangle$ for $l=0,\dots,k$, $i=1,\dots,n$ and $m^{(l)}=(m_1^{(l)},\dots,m_n^{(l)})$. We choose a coordinate $\xi=(\xi_1,\dots,\xi_n)$ so that the equality $\C[\tau^\vee\cap\Z^{n\times 1}]=\C[\xi]$ holds. Then, $\tilde{h}_{l,z^{(l)}}(\xi)=\xi^{-m^{(l)}}h_{l,z^{(l)}}(\xi)$ ($l=1,\dots,k$) is a polynomial with non-zero constant term. For any subset $I\subset\{ 1,\dots,n\}$, we set $\tilde{h}^I_{l,z^{(l)}}(\xi_{\bar{I}})=\tilde{h}_{l,z^{(l)}}(\xi)\restriction_{\cap_{i\in I}\{ \xi_i=0\}}$. Then, the condition 1 of \cref{defn:Nonsingularity} fails if and only if $d_{\xi_{\bar{I}}}\tilde{h}^I_{l_1,z^{(l_1)}}(\xi_{\bar{I}})\wedge\dots\wedge d_{\xi_{\bar{I}}}\tilde{h}^I_{l_s,z^{(l_s)}}(\xi_{\bar{I}})=0$ for some $\xi_{\bar{I}}\in\left\{ \xi_{\bar{I}}\in\C^{\bar{I}}\mid \tilde{h}^I_{l_1,z^{(l_1)}}(\xi_{\bar{I}})=\dots=\tilde{h}^I_{l_s,z^{(l_s)}}(\xi_{\bar{I}})=0\right\}$. This condition is clearly a Zariski closed condition. 

As for condition 2 of \cref{defn:Nonsingularity}, we rearrange the index $\{ 1,\dots,n\}=\{ 1,\dots,i_0,i_0+1,\dots, n\}$ so that $m_i^{(0)}<0$ for $i=1,\dots,i_0$ and $m_i^{(0)}=0$ for $i=i_0+1,\dots,n$. For any subset $I\subset\{ 1,\dots,n\}$ such that $I\cap\{1,\dots,i_0\}\neq\varnothing$, we set $\tilde{h}^I_{0,z^{(0)}}(\xi_{\bar{I}})=\displaystyle\prod_{i\in I}\xi_i^{-m_i^{(0)}} h_{0,z^{(0)}}(\xi)\restriction_{\cap_{i\in I}\{ \xi_i=0\}}$. Then, condition 2 of \cref{defn:Nonsingularity} fails if and only if $\tilde{h}^I_{0,z^{(0)}}(\xi_{\bar{I}})d_{\xi_{\bar{I}}}\tilde{h}^I_{l_1,z^{(l_1)}}(\xi_{\bar{I}})\wedge\dots\wedge d_{\xi_{\bar{I}}}\tilde{h}^I_{l_s,z^{(l_s)}}(\xi_{\bar{I}})=0$ and 
$d_{\xi_{\bar{I}}}\tilde{h}^I_{0,z^{(0)}}(\xi_{\bar{I}})\wedge d_{\xi_{\bar{I}}}\tilde{h}^I_{l_1,z^{(l_1)}}(\xi_{\bar{I}})\wedge\dots\wedge d_{\xi_{\bar{I}}}\tilde{h}^I_{l_s,z^{(l_s)}}(\xi_{\bar{I}})=0$ for some $\xi_{\bar{I}}\in\left\{ \xi_{\bar{I}}\in\C^{\bar{I}}\mid \tilde{h}^I_{l_1,z^{(l_1)}}(\xi_{\bar{I}})=\dots=\tilde{h}^I_{l_s,z^{(l_s)}}(\xi_{\bar{I}})=0\right\}$. This is also a Zariski closed condition. Finally, the non-emptiness of nonsingular points follows immediately from the description above and Bertini-Sard's lemma.
\end{proof}

\begin{rem}
If $k=0$, the nonsingularity condition is equivalent to the non-degenerate condition of \cite[p274]{Adolphson}. In general, nonsingularity condition is stronger than non-degenerate condition. Never the less, it is still a Zariski open dense condition as we saw above.
\end{rem}

\begin{rem}
By classical elimination theory, the defining equations of the set of singular points are polynomials with rational coefficients.
\end{rem}

In the following, we fix a nonsingular $z$ and a small positive real number $\varepsilon.$ Let $\Omega$ be the Zariski open subset of $\A^N$ consisting of nonsingular points. By abuse of notation, we write $D_j$ for the product $\Omega\times D_j$. By the condition 1 of \cref{defn:Nonsingularity}, for any subset $I\subset\{ 1,\dots,k\}$, the closure $Z_I=\overline{\displaystyle\bigcap_{l\in I}\{ (z,x)\in\Omega\times(\C^\times)^n_x\mid h_{l,z^{(l)}}(x)=0\}}$ $\subset \Omega\times X$ intersects transversally with $D_{J^\prime}=\displaystyle\bigcap_{j\in J^\prime}D_j$ for any $J^\prime\subset J$. Let us rename the divisors $D_j$ so that $D_j$ with $j\in J_1$ is a part of the pole divisor of $h_{0,z^{(0)}}(x)$ on $X$ and that any $D_j$ with $j\in J_2$ is not. Then by the condition 2 of \cref{defn:Nonsingularity}, the closure $Z_0=\overline{\{ (z, x)\in\Omega\times(\C^\times)^n_x\mid h_{0,z^{(0)}}(x)=0\}}\subset \Omega\times X$ intersects transversally with $Z_I\cap D_{J^\prime}$ such that $J^\prime\cap J_1\neq \varnothing$. 

Now we consider the canonical projection $p:\Omega\times X\rightarrow\Omega$. We recall the blowing up process of \cite{EsterovTakeuchi} (see also \cite{MatsuiTakeuchi}). We consider a sequence of blow-ups of $\Omega\times X$ along codimension $2$ divisors $Z_0\cap D_j$ ($j\in J_1$). If the pole order of $h_{0,z^{(0)}}(x)$ along $D_j$ is $m_j\in\Z_{>0}$, one needs at most $m_j$ blow-ups along $Z_0\cap D_j$. Repeating this process finitely many times, we obtain a non-singular complex variety $\bar{X}$ so that the meromorphic map $h_{0,z^{(0)}}:\Omega\times X\dashrightarrow \mathbb{P}^1$ extends to a regular map $\bar{h}_{0,z^{(0)}}:\bar{X}\rightarrow \mathbb{P}^1$. We write $\bar{p}:\bar{X}\rightarrow \Omega$ for the composition of the natural morphism $\bar{X}\rightarrow \Omega\times X$ with the canonical projection $\Omega\times X\rightarrow \Omega$. We also write $\bar{Z}_l$ and $\bar{D}_j$ for the proper transforms of $Z_l$ and $D_j$. We equip $\bar{X}$ with the Whitney stratification coming from the normal crossing divisors $\bar{D}=\{ \bar{Z}_l\}_{l=1}^k\cup\{ \bar{D}_j\}_{j\in J}\cup\{\text{exceptional divisors of blow-ups}\}$. By construction, we see that $\bar{h}_{0,z^{\prime(0)}}^{-1}(\infty)$ intersects transversally with any stratum of $\bar{p}^{-1}(z)$ for any $z\in\Omega$. Let us consider a real oriented blow-up $\widetilde{X}=\widetilde{\bar{X}^{\bar{D}}}$ of $\bar{X}$ along $\bar{D}$. We naturally have the following commutative diagram
\begin{equation}
\xymatrix{
 \widetilde{X} \ar[r]^{\tilde{h}_{0,z^{(0)}} } \ar[d]_{\varpi}&\widetilde{\PP^1}\ar[d]^{\varpi_\infty}\\
 \bar{X} \ar[r]^{\bar{h}_{0,z^{(0)}}}                             &\PP^1.
}
\end{equation}
We also equip $\widetilde{X}$ with the Whitney stratification coming from the pull-back of  the normal crossing divisor $\bar{D}$. We set $\tilde{p}=\bar{p}\circ\varpi$. Then, $\tilde{p}^{-1}(z)$ for any $z\in\Omega$ is naturally equipped with an induced Whitney stratification. By construction, $\tilde{h}_{0,z^{(0)}}^{-1}(e^{\ii\theta}\infty)$ intersects transversally with any stratum of  $\tilde{p}^{-1}(z)$ for any $\theta$. Now it is routine to take a ruguous vector field $\Theta$ on $\tilde{p}^{-1}(\Delta(z;\ve))$ for a small positive number $\ve>0$ and a point $z\in\Omega$ with an additional condition 
\begin{equation}\label{AdditionalCondition}
\Theta (\tilde{h}_{0,z^{(0)}}(x))=0
\end{equation}
near $\tilde{h}_{0,z^{(0)}}^{-1}(S^1\infty )$ (\cite[\S4]{Verdier}, see also \cite[\S3.3.]{HienRoucairol}). Taking the flow of $\Theta$, we have a stratified local trivialization of the morphism $\tilde{p}:\widetilde{X}\rightarrow\Omega$. We define $\widetilde{D^{r.d.}}\subset\widetilde{X}$ by the formula $\widetilde{D^{r.d.}}=\tilde{h}_{0,z^{(0)}}^{-1}\left( (\frac{\pi}{2},\frac{3\pi}{2})\infty\right)\setminus\varpi^{-1}\bar{D}^{irr}$ and put $\widetilde{D^{r.d.}_{z}}=\widetilde{D^{r.d.}}\cap \tilde{p}^{-1}(z)$. With the aid of the additional condition (\ref{AdditionalCondition}), we have a local trivialization 
\begin{equation}\label{Trivialization}
\xymatrix{
 \left(\pi^{-1}(z)^{an}\cup \widetilde{D^{r.d.}_{z}}\right)\times \Delta(z;\varepsilon) \ar[d]_{ } \ar[r]^-{\Lambda}& \pi^{-1}(\Delta(z;\varepsilon))^{an}\cup \widetilde{D^{r.d.}} \ar[dl]_{\tilde{p}}\\
\Delta(z;\varepsilon)  &
}
\end{equation}
with an additional condition $\Lambda\left(\widetilde{D^{r.d.}_{z}}\times \Delta(z;\varepsilon)\right)\subset \widetilde{D^{r.d.}}$. Here, the first vertical arrow is the canonical projection. It is clear that $\tilde{p}^{-1}(z)$ is a good compactification for any $z\in\Omega$. For any $z\in\C^N$, we write $\Phi_{z}$ for the multivalued function on $\pi^{-1}(z)$ defined by $\pi^{-1}(z)\ni x\mapsto \Phi(z,x)$. Writing $j_{z}:\pi^{-1}(z)^{an}\hookrightarrow \pi^{-1}(z)^{an}\cup \widetilde{D^{r.d.}_{z}}$ for the natural inclusion,  we set 
\begin{equation}
\Homo^{r.d.}_{*,z}=\Homo_*\left( \pi^{-1}(z)^{an}\cup \widetilde{D^{r.d.}_{z}},\widetilde{D^{r.d.}_{z}};j_{z*}\left(\underline{\C} \Phi_{z}\right)\right).
\end{equation}

\begin{thm}\label{thm:SolutionDescription}
For any $z\in\Omega$, the map
\begin{equation}
\int:\Homo^{r.d.}_{n,z}\ni[\Gamma]\mapsto\left( [\omega]\mapsto\int_\Gamma \Phi \omega \right)\in\Hom_{\DD_{\C^N}}\left( \left(\int_{\pi}^0\mathcal{O}_{X_{0} }\Phi\right)^{an},\mathcal{O}_{\C^N}\right)_{z}\label{Morphism}
\end{equation}
is a well-defined isomorphism of complex vector spaces.
\end{thm}

\begin{proof}

Note first that, for any $[\omega]\in \int_{\pi}^0\mathcal{O}_{X_{0} }\Phi$, the integral $\int_\Gamma \Phi\omega$ is well-defined in a neighborhood of $z$. Indeed, with the aid of the trivialization (\ref{Trivialization}), one can construct a continuous family $\{\Gamma_{z^\prime}\}_{z^\prime\in\Delta(z;\varepsilon)}$ of rapid decay cycles such that $\Gamma_{z}=\Gamma$. For any $z^\prime$ close to $z$, $\Gamma_{z^\prime}$ is homotopic to $\Gamma$. This argument shows that (\ref{Morphism}) is well-defined. 

Let us show that (\ref{Morphism}) is an isomorphism. If $\int_\Gamma \Phi\omega=0$ for any $[\omega]$, we have $[\Gamma]=0$ by the perfectness of the period pairing. Thus, (\ref{Morphism}) is injective. By abuse of notation, we write $\pi$ for the composition of $\pi^{-1}(\Omega)\overset{i}{\hookrightarrow} \bar{X}\overset{\bar{p}}{\rightarrow}\Omega$. Since $i$ is an open embedding, we have an estimate ${\rm Char}(\int_i\mathcal{O}_{\pi^{-1}(\Omega)}\Phi)\subset T_{\bar{D}}\bar{X}$. In view of the estimate (\ref{eqn:EstimateOfChar}) and the fact that each fiber of $\bar{p}$ meets transversally with each component of $\bar{D}$, we obtain an estimate ${\rm Char}(\int_\pi\mathcal{O}_{\pi^{-1}(\Omega)}\Phi)={\rm Char}(\int_{\bar{p}}\circ\int_i\mathcal{O}_{\pi^{-1}(\Omega)}\Phi)\subset T^*_\Omega\Omega$. This estimate shows that $\int_{\pi}\mathcal{O}_{X_0}\Phi$ is a connection on $\Omega$. Thus, by Cauchy-Kovalevsky-Kashiwara theorem (\cite[Theorem 4.3.2]{HTT}) and base change formula (\ref{BaseChange}), the complex dimension of the right-hand side of (\ref{Morphism}) is equal to that of $\Homo^n_{dR}(\pi^{-1}(z);\nabla_z)$, which in turn equals that of $\Homo^{r.d.}_{n,z}$.
\end{proof}

\begin{rem}
The assumption that $z$ is nonsingular is important. As a simple example, we consider a Laplace-Gauss-Manin connection $\int_\pi \mathcal{O}_{\A^2_z\times \Gm}e^{z_1x+z_2x^2}x^c$ with $c\notin\Z$. In this case, we can easily see that $z$ is nonsingular (non-degenerate) if $z_2\neq 0$. Let us fix a point $z=(1,0).$ Then, the Hankel contour $\Gamma$ which begins from $-\infty$ turns around the origin and goes back to $-\infty$ belongs to $\Homo^{r.d.}_{1,z}.$ However, as soon as ${\rm Re}(z_2)> 0$, the integral $\int_\Gamma e^{x+z_2x^2}x^c\frac{dx}{x}$ diverges. 

\end{rem}

\noindent
By \cref{cor:Corollary216}, an isomorphism
\begin{equation}
\Hom_{\DD_{\C_z^N}}(\int_{\pi}^0\mathcal{O}_{X_{0} }\Phi,\mathcal{O}_{\C^N})\rightarrow\Hom_{\DD_{\C^N_z}}(M_A(\delta),\mathcal{O}_{\C^N})
\end{equation}
is induced. In view of \cref{thm:SolutionDescription}, we obtain the second main result of this section. 

\begin{thm}\label{thm:EulerLaplaceRepresentationTheorem}
Suppose the parameter vector $\delta$ is non-resonant and $\gamma_l\notin\Z$ for any $l=1,\dots,k$. For any $z\in \Omega$, one has an isomorphism
\begin{equation}\label{Integration}
\Homo^{r.d.}_{n,z}\overset{\int}{\rightarrow}\Hom_{\DD_{\C^N_z}}(M_A(\delta),\mathcal{O}_{\C^N})_{z}
\end{equation}
given by
\begin{equation}
[\Gamma]\mapsto \int_\Gamma \Phi \frac{dx}{x}.
\end{equation}
\end{thm}

\begin{rem}
It is straightforward to construct a local system $\mathcal{H}^{r.d.}_{n}=\displaystyle\bigcup_{z\in\Omega^{an}}\Homo^{r.d.}_{n,z}\rightarrow\Omega^{an}$ and an isomorphism $\mathcal{H}^{r.d.}_{n}\overset{\int}{\rightarrow}\Hom_{\DD_{\C^N_z}}(M_A(\delta),\mathcal{O}_{\C^N})\restriction_{\Omega^{an}}$ whose stalks are identical with (\ref{Integration}). See the proofs of \cite[Proposition 3.4. and Theorem 3.5.]{HienRoucairol}.

\end{rem}


\section{Rapid decay intersection pairing and its localization}\label{RDIntersection}

In this section, we develop an intersection theory of rapid decay homology groups along the line of the preceding studies \cite{ChoMatsumoto}, \cite{GotoCyclesForFC}, \cite{IwasakiWittenLaplacian}, \cite{KitaYoshida2}, \cite{MajimaMatsumotoTakayama}, \cite{MimachiYoshida}, and \cite{OharaSugikiTakayama}. We use the same notation as \S\ref{subsec:3.1}. Namely, we consider a smooth complex Affine variety $U$ and a regular singular connection $(E,\nabla)$ on $U$. In order to simplify the discussion and the notation, we assume that $E$ is a trivial bundle and $\nabla$ is given by $\nabla=d+\sum_{i=1}^k\alpha_i\frac{df_i}{f_i}\wedge$ for some regular functions $f_i\in\mathcal{O}(U)\setminus\C$ and complex numbers $\alpha_i$. For another regular function $f\in\mathcal{O}(U)$, we set $\nabla_f=\nabla+df\wedge$. We take a standard orientation of $\C^n$ so that for any holomorphic coordinate $(z_1,\dots,z_n)$, the real form $\left(\frac{\ii}{2}\right)^ndz_1\wedge\dots\wedge dz_n\wedge d\bar{z}_1\wedge\dots\wedge d\bar{z}_n$ is positive. Note that this choice of orientation is not compatible with the product orientation. For example, our orientation of $\C^2$ is different from the product orientation of $\C\times\C$.

Let us recall several sheaves on the real oriented blow-up $\widetilde{X}$. We write $\mathcal{P}^{<D}_{\widetilde{X}}$ (resp. $\mathcal{P}^{{\rm mod}D}_{\widetilde{X}}$) for the sheaf of $C^\infty$ functions on $U^{an}$ which are flat (resp. moderate growth) along $D$. We write $\Omega^{(p,q)}_{X^{an}}$ for the sheaf of smooth $(p,q)$-forms on $X^{an}$. We set $\mathcal{A}^{?D}_{\widetilde{X}}=\Ker\left( \bar{\partial}:\mathcal{P}^{?D}_{\widetilde{X}}\rightarrow \mathcal{P}^{?D}_{\widetilde{X}}\otimes_{\varpi^{-1}\mathcal{O}_{X^{an}}}\varpi^{-1}\Omega^{(0,1)}_{X^{an}}\right)$ for $?=<,{\rm mod}$. As in \cite{HienRDHomology}, we set
\begin{equation}\label{RapidDecayComplex}
DR^{?D}_{\widetilde{X}}(\nabla_f):=\mathcal{A}^{?D}_{\widetilde{X}}\otimes_{\varpi^{-1}\mathcal{O}_{X^{an}}}\varpi^{-1}DR_{X^{an}}(\nabla_f).
\end{equation}
With this notation, we set $\mathcal{S}^{?D}:=\mathcal{H}^0\left( DR^{?D}_{\widetilde{X}}(\nabla_f)\right)$ and ${}^\vee\mathcal{S}^{?D}:=\mathcal{H}^0\left( DR^{?D}_{\widetilde{X}}(\nabla^\vee_f)\right)$.

The main result of \cite{HienRDHomology} is an explicit description of the perfect pairing between $\mathbb{H}^{2n-*}(\widetilde{X},{}^\vee\mathcal{S}^{<D})$ and $\mathbb{H}^*(\widetilde{X},\mathcal{S}^{{\rm mod}D})$. Namely, $\mathbb{H}^{2n-*}(\widetilde{X},{}^\vee\mathcal{S}^{<D})$ is described by the rapid decay homology group, $\mathbb{H}^*(\widetilde{X},\mathcal{S}^{{\rm mod}D})$ is the algebraic de Rham cohomology group, and the pairing is given by the exponential period pairing. For later use, we need other realizations of these cohomology groups and the period pairing.

We first remark that $DR_{\widetilde{X}}^{?D}(\nabla_f)$ is a resolution of ${}^\vee\mathcal{S}^{?D}$ (\cite[Proposition 1]{HienRDHomology}). Combining this result with the quasi-isomorphism
\begin{equation}
\mathcal{A}^{?D}_{\widetilde{X}}\otimes_{\varpi^{-1}\mathcal{O}_{X^{an}}}\varpi^{-1}\Omega^{r}_{X^{an}}\tilde{\rightarrow}\left(\mathcal{P}^{?D}_{\widetilde{X}}\otimes_{\varpi^{-1}\mathcal{O}_{X^{an}}}\varpi^{-1}\Omega^{(r,\bullet)}_{X^{an}},\bar{\partial}\right),
\end{equation}
we see that $\mathcal{S}^{?D}$ is quasi-isomorphic to
$
\mathcal{P}DR^{?D}_{\widetilde{X}}(\nabla_f)\overset{def}{=}\left(\mathcal{P}^{?D}_{\widetilde{X}}\otimes_{\varpi^{-1}\mathcal{O}_{X}}\varpi^{-1}\Omega^{(\bullet,\bullet)}_{X},\nabla_f,\bar{\partial}\right)$. We set 
\begin{equation}
\Homo_{r.d.}^*\left( U,\nabla_f^\vee\right)\overset{def}{=}\mathbb{H}^{*}\left( \widetilde{X}; {}^\vee\mathcal{S}^{<D}\right).
\end{equation}

\noindent
Since $\mathcal{P}DR^{<D}_{\widetilde{X}}(\nabla^\vee_f)$ is a soft resolution of ${}^\vee\mathcal{S}^{<D}$, $\Homo_{r.d.}^*\left( U,(\mathcal{O}_U,\nabla_f^\vee)\right)$ can be computed by taking global sections of the complex $\mathcal{P}DR^{<D}_{\widetilde{X}}(\nabla^\vee_f)$. We also remark that the factorization $\Gamma_{\widetilde{X}}=\Gamma_{X^{an}}\circ\R\varpi$ and the relation $\R\varpi_*\left( \mathcal{A}^{{\rm mod}D}_{\widetilde{X}}\right)=\mathcal{O}_{X^{an}}(*D)$ (\cite[CHAPITRE I\vspace{-.1em}I, Corollaire 1.1.8]{SabbahDim2}) implies $\mathbb{H}^*(\widetilde{X},\mathcal{S}^{{\rm mod}D})=\Homo^*\left(\R\Gamma_{X^{an}}\R\varpi_*\left( DR^{{\rm mod}D}_{\widetilde{X}}(\nabla_f)\right)\right)=\mathbb{H}^*\left(X^{an};\mathcal{O}_{X^{an}}(*D)\otimes_{\mathcal{O}_{X^{an}}}DR_{X^{an}}(\nabla_f)\right)=\Homo_{\rm dR}^*\left( U,\nabla_f\right)$. Here, the last equality is a consequence of GAGA (\cite{SerreGAGA}).

Hence, if we write $\tilde{j}:U^{an}\rightarrow \widetilde{X}$ for the natural inclusion, the canonical duality pairing $DR^{<D}_{\widetilde{X}}(\nabla^\vee_f)\otimes DR^{{\rm mod}D}_{\widetilde{X}}(\nabla_f)\rightarrow\tilde{j}_!\C_{U^{an}}$ (\cite{HienRDHomology}[Theorem 3]) yields a perfect pairing
\begin{equation}\label{RDc.i.n.}
\begin{array}{cccc}
\langle\bullet,\bullet\rangle_{ch} &\Homo_{\rm dR}^{*}\left( U; \nabla_f\right)\times\Homo_{r.d.}^{2n-*}\left( U,\nabla_f^\vee\right)  &\rightarrow&\C\\
&\rotatebox{90}{$\in$}& &\rotatebox{90}{$\in$}\\
&([\omega],[\eta]) &\mapsto&\int\omega\wedge\eta.
\end{array}
\end{equation}

\noindent
We often write $\langle\omega,\eta\rangle_{ch}$ instead of $\langle[\omega],[\eta]\rangle_{ch}$. We also give a realization of $\mathbb{H}^*(\widetilde{X},\mathcal{S}^{{\rm mod}D})$ in terms of a certain relative homology group. We set $\widetilde{D^{mod}}=\left( \widetilde{D_{irr}}\setminus\tilde{f}^{-1}(S^1\infty)\right)\cup\tilde{f}^{-1}\left( \left\{ e^{\theta\ii}\infty\mid -\frac{\pi}{2}<\theta<\frac{\pi}{2}\right\}\right)$. We consider a sequence of natural inclusions $U^{an}\overset{l}{\rightarrow} U^{an}\cup \widetilde{D^{mod}}\overset{k}{\rightarrow}\widetilde{X}$. By the local description of $\mathcal{S}^{{\rm mod}D}$ (\cite{HienRDHomology}[p12]), we can easily confirm that the equality $\mathcal{S}^{{\rm mod}D}=k_!l_*\mathcal{L}^\vee$ holds. Moreover, if we set 
\begin{equation}
\mathcal{C}^{-p}_{U^{an}\cup \widetilde{D^{mod}},\widetilde{D^{mod}}}(\mathcal{L}^\vee)=\left( V\mapsto S_p\left( U^{an}\cup \widetilde{D^{mod}},\left(U^{an}\cup \widetilde{D^{mod}}\setminus V\right)\cup\widetilde{D^{mod}};l_*\mathcal{L}^\vee\right)\right)^\dagger,
\end{equation}
we see that $\mathcal{S}^{{\rm mod}D}_{\widetilde{X}}[2n]\simeq k_!\mathcal{C}^{-*}_{U^{an}\cup \widetilde{D^{mod}},\widetilde{D^{mod}}}(\mathcal{L}^\vee)$ as in the arguments after Proposition 2.1 of \cite{MatsubaraRapidDecay}. Here, the superscript $\dagger$ denotes the sheafification. Therefore, we have a realization $\mathbb{H}^*\left( \widetilde{X};\mathcal{S}^{{\rm mod}D}\right)=\Homo^{mod}_{2n-*}\left( U,\nabla_f\right)\overset{def}{=}\Homo_{2n-*}\left( U^{an}\cup \widetilde{D^{mod}},\widetilde{D^{mod}};l_*\mathcal{L}^\vee\right)$. Moreover, the same argument as \S 5 of \cite{HienRDHomology} proves the perfectness of the pairing

\begin{equation}
\begin{array}{ccc}
 \Homo_{r.d.}^{*}\left( U,\nabla_f^\vee\right)\times \Homo^{mod}_{2n-*}\left( U,\nabla_f\right) &\rightarrow&\C\\
\rotatebox{90}{$\in$}& &\rotatebox{90}{$\in$}\\
(\eta,\delta^\vee) &\mapsto&\int_{\delta^\vee}e^{-f}\prod_{i=1}^kf_i^{-\alpha_i}\eta.
\end{array}
\end{equation}

With these setups, we can naturally define the Poincar\'e duality isomorphism $\Phi:\Homo^{r.d.}_{*}\left( U^{an},\nabla_f^\vee\right)\overset{\sim}{\rightarrow}\Homo_{r.d.}^{2n-*}\left( U,\nabla_f^\vee\right)$ and $\Phi^\vee:\Homo^{mod}_{*}\left( U,\nabla_f\right)\overset{\sim}{\rightarrow}\Homo^{2n-*}_{dR}\left(U,\nabla_f\right)$. Namely, for any element $[\gamma]\in\Homo^{r.d.}_{*}\left( U^{an},\nabla_f^\vee\right)$, $\Phi(\gamma)\in\Homo_{r.d.}^{2n-*}\left( U,\nabla_f^\vee\right)$ is the unique element such that the equality $\int_{\gamma}e^f\prod_{i=1}^kf_i^{\alpha_i}\omega=\int \Phi(\gamma)\wedge\omega$ holds for any $\omega\in\Homo^{*}_{dR}\left(U,\nabla_f\right)$. In the same way, for any element $[\delta^\vee]\in\Homo^{mod}_{*}\left( U,\nabla_f\right)$, $\Phi^\vee(\delta^\vee)\in\Homo^{2n-*}_{dR}\left(U,\nabla_f\right)$ is the unique element such that the equality $\int_{\delta^\vee}e^{-f}\prod_{i=1}^kf_i^{-\alpha_i}\eta=\int \Phi^\vee(\delta^\vee)\wedge\eta$ holds for any $\eta\in \Homo_{r.d.}^{*}\left( U,\nabla_f^\vee\right)$. We define the homology intersection pairing $\langle\bullet,\bullet\rangle_h$ by

\begin{equation}\label{eqn:4.7}
\begin{array}{cccc}
 \langle\bullet,\bullet\rangle_h:&\Homo_*^{r.d.}\left( U,\nabla^\vee_f\right)\times\Homo^{mod}_{2n-*}\left( U,\nabla_f\right)&\rightarrow&\C\\
&\rotatebox{90}{$\in$}& &\rotatebox{90}{$\in$}\\
&(\gamma,\delta^\vee) &\mapsto&\int\Phi(\gamma)\wedge\Phi^\vee(\delta^\vee).
\end{array}
\end{equation}

We are ready to state the twisted period relation (cf. \cite[Theorem 2]{ChoMatsumoto}) for rapid decay homology groups. Let us fix four bases $\{\omega_i\}_{i=1}^r\subset\Homo_{dR}^{*}\left( U; \nabla_f\right)$, $\{\gamma_i\}_{i=1}^r\subset\Homo_*^{r.d.}\left( U;\nabla_f^\vee\right)$, $\{\eta_i\}_{i=1}^r\subset\Homo_{r.d.}^{2n-*}\left( U;\nabla_f^\vee\right)$, and $\{\delta_i^\vee\}_{i=1}^r\subset\Homo^{mod}_{2n-*}\left( U;\nabla_f\right)$. We set $I_{ch}=(\langle \omega_i, \eta_j\rangle_{ch})_{i,j}$, $I_h=(\langle \gamma_i,\delta_j^\vee\rangle_h)_{i,j}$, $P=\left( \int_{\gamma_j}e^{f}\prod_{l=1}^kf_l^{\alpha_l}\omega_i\right)_{i,j}$, and $P^\vee=\left( \int_{\delta_j^\vee}e^{-f}\prod_{l=1}^kf_l^{-\alpha_l}\eta_i\right)_{i,j}$. With these setups, we can state a basic 
\begin{prop}
The following identity is true:
\begin{equation}\label{GeneralQuadraticRelation}
I_{ch}=P{}^tI_h^{-1}{}^tP^\vee.
\end{equation}
\end{prop}

\noindent
The proof is exactly same as that of the twisted period relation \cite[Theorem 2]{ChoMatsumoto}.

Now, we need to prove an important technique to compute homology intersection numbers. For any open subset $\widetilde{V}$ of $\widetilde{X}$, we set $V=\widetilde{V}\cap U$, $DR^{?D}_{\widetilde{V}}(\nabla_f)=DR^{?D}_{\widetilde{X}}(\nabla_f)\restriction_{\widetilde{V}}$ and $\mathcal{S}^{?D}_{V}=\mathcal{S}^{?D}_{\widetilde{X}}\restriction_{\widetilde{V}}$. Let $a_{\widetilde{V}}:\widetilde{V}\rightarrow pt$ denote the  morphism to a point. We set $\widetilde{D_V^{r.d.}}=\widetilde{D^{r.d.}}\cap\widetilde{V}$, $\widetilde{D_V^{mod}}=\widetilde{D^{mod}}\cap\widetilde{V}$, $\Homo^{mod}_*\left( V; \nabla_f\right)=\Homo_{*}^{\rm lf}\left( V\cup \widetilde{D_V^{mod}},\widetilde{D_V^{mod}};l_*\mathcal{L}^\vee\right)$, $\Homo^*_{mod}\left(V; \nabla_f\right)=\mathbb{H}^{*}\left( \widetilde{V};DR^{{\rm mod}D}_{\widetilde{V}}(\nabla_f)\right)$, 
$\Homo^{r.d.}_*\left(V;\nabla_f^\vee\right)=\Homo_{*}\left( V\cup \widetilde{D_V^{r.d.}},\widetilde{D_V^{r.d.}};\mathcal{L}\right)$, and $\Homo_{r.d.}^*\left(V;\nabla_f^\vee\right)=\mathbb{H}_c^{*}\left( \widetilde{V};DR^{<D}_{\widetilde{V}}(\nabla_f^\vee)\right)$. We write $\tilde{j}_V:V\rightarrow\widetilde{V}$ for the natural inclusion. By the perfectness of the duality pairing $DR_{\widetilde{V}}^{<D}(\nabla_f^\vee)\otimes DR^{\rm mod D}_{\widetilde{V}}(\nabla_f)\rightarrow \tilde{j}_{V!}\C$ and the identity $a_{\widetilde{V}}^!\C\simeq\tilde{j}_{V!}\C[2n]$ we get a sequence of isomorphisms

\begin{align}
\R\Gamma_{\widetilde{V}}\left( DR^{\rm mod D}_{\widetilde{V}}(\nabla_f)[2n]\right)&\simeq\R\Gamma_{\widetilde{V}}\R\mathcal{H}om\left( DR_{\widetilde{V}}^{<D}(\nabla_f^\vee),\tilde{j}_{V!}\C_V[2n]\right)\\
&\simeq\R\mathcal{H}om\left( a_{\widetilde{V}!}DR_{\widetilde{V}}^{<D}(\nabla_f^\vee),\C\right).
\end{align}

\noindent
Here, the last isomorphism is a result of Poincar\'e-Verdier duality. Since we have identities $\mathbb{H}^*\left( \widetilde{V};\mathcal{S}^{{\rm mod}D}_{V}\right)=\Homo^{mod}_{2n-*}\left( V; \nabla_f\right)$ and $\mathbb{H}^*\left( \widetilde{V};\mathcal{S}^{<D}_{V}\right)=\Homo^{r.d.}_{2n-*}\left(V;\nabla_f^\vee\right)$, this isomorphism gives rise to perfect pairings
\begin{equation}
\begin{array}{ccc}
\Homo^{r.d.}_*\left( V;\nabla_f^\vee\right)\times \Homo^{*}_{mod}\left( V;\nabla_f\right)&\rightarrow&\C\\
\rotatebox{90}{$\in$}& &\rotatebox{90}{$\in$}\\
(\gamma,\omega)&\mapsto&\int_{\gamma}e^{f}\prod_{i=1}^kf_i^{\alpha_i}\omega,
\end{array}
\end{equation}

\begin{equation}
\begin{array}{ccc}
\Homo_{r.d.}^*\left( V;\nabla_f^\vee\right)\times \Homo_{*}^{mod}\left(V;\nabla_f\right)&\rightarrow&\C\\
\rotatebox{90}{$\in$}& &\rotatebox{90}{$\in$}\\
(\eta,\delta^\vee)&\mapsto&\int_{\delta^\vee}e^{-f}\prod_{i=1}^kf_i^{-\alpha_i}\eta,
\end{array}
\end{equation}

and

\begin{equation}
\begin{array}{cccc}
\langle\bullet,\bullet\rangle_{ch}:&\Homo^{*}_{mod}\left(V;\nabla_f\right)\times \Homo_{r.d.}^{2n-*}\left( V;\nabla_f^\vee\right)&\rightarrow&\C\\
&\rotatebox{90}{$\in$}& &\rotatebox{90}{$\in$}\\
&(\omega,\eta)&\mapsto&\int\omega\wedge\eta.
\end{array}
\end{equation}

\noindent
Then, the Poincar\'e duality morphisms $\Phi_{V}:\Homo^{r.d.}_*\left( V;\nabla_f^\vee\right)\tilde{\rightarrow}\Homo_{r.d.}^{2n-*}\left( V;\nabla_f^\vee\right)$ and its dual counterpart $\Phi_{V}^\vee:\Homo^{mod}_*\left( V;\nabla_f\right)\tilde{\rightarrow}\Homo_{mod}^{2n-*}\left( V;\nabla_f\right)$ are naturally defined with the aid of the perfect pairings above. Namely, the definitions of $\Phi_V$ and $\Phi^\vee_V$ are parallel to those of $\Phi$ and $\Phi^\vee.$  Therefore, we define the homology intersection pairing by

\begin{equation}
\begin{array}{cccc}
 \langle\bullet,\bullet\rangle_h:&\Homo^{r.d.}_*\left( V;\nabla_f^\vee\right)\times\Homo^{mod}_{2n-*}\left( V;\nabla_f\right)&\rightarrow&\C\\
&\rotatebox{90}{$\in$}& &\rotatebox{90}{$\in$}\\
&(\gamma,\delta^\vee) &\mapsto&\int\Phi(\gamma)\wedge\Phi^\vee(\delta^\vee).
\end{array}
\end{equation}

With these set-ups, we can discuss localization of intersection pairings. If we denote by $\iota_{\widetilde{X}\widetilde{V}}:\widetilde{V}\rightarrow\tilde{X}$ the natural inclusion, the natural transform 
$
id_{\widetilde{X}}\rightarrow\iota_{\widetilde{X}\widetilde{V}*}\iota_{\widetilde{X}\widetilde{V}}^{-1}
$ 
induces a commutative diagram

\begin{equation}
\xymatrix{
 \Homo_{*}^{mod}\left( U;\nabla_f\right) \ar[r]^{rest } \ar[d]_{\Phi^\vee}&\Homo_{*}^{mod}\left( V;\nabla_f\right)\ar[d]^{\Phi^\vee}\\
 \Homo_{\rm dR}^{2n-*}\left( U;\nabla_f\right) \ar[r]^{rest}                             &\Homo^{2n-*}_{mod}\left( V;\nabla_f\right).
}
\end{equation}
Here, the first horizontal morphism is nothing but the usual restriction morphism of locally finite homology groups and the second horizontal morphism is induced by taking pull-backs of differential forms. 
On the other hand, the natural transform $\iota_{\widetilde{X}\widetilde{V}!}\iota_{\widetilde{X}\widetilde{V}}^{-1}\rightarrow id_{\widetilde{X}}$ induces a commutative diagram

\begin{equation}
\xymatrix{
 \Homo_{*}^{r.d.}\left( U;\nabla^\vee_f\right) \ar[d]_{\Phi}&\Homo_{*}^{r.d.}\left( V;\nabla_f^\vee\right)\ar[l]^{\iota_{\widetilde{X}\widetilde{V}!} } \ar[d]^{\Phi}\\
 \Homo^{2n-*}_{r.d.}\left( U;\nabla_f^\vee\right)  &\Homo^{2n-*}_{r.d.}\left( V;\nabla_f^\vee\right)\ar[l]^{\iota_{\widetilde{X}\widetilde{V}!}}                            .
}
\end{equation}
\noindent
Both horizontal morphisms are given by extension by zero. By definition, we see that the morphisms $rest$ and $\iota_{\widetilde{X}\widetilde{V}!} $ satisfy an adjoint relation
\begin{equation}\label{Adjoint}
\langle \iota_{\widetilde{X}\widetilde{V}!}(\gamma),\delta^\vee\rangle_h=\langle \gamma,rest(\delta^\vee)\rangle_h\;\;\;\;\left(\gamma\in\Homo_{*}^{r.d.}\left( V;\nabla_f^\vee\right),\delta^\vee\in\Homo_{2n-*}^{mod}\left( U;\nabla_f\right)\right).
\end{equation}

\noindent
We consider a commutative diagram

\begin{equation}\label{ThisCD}
\xymatrix{
\mathcal{S}^{<D}_{\widetilde{X}}\ar[r]^{can}&\mathcal{S}^{{\rm mod}D}_{\widetilde{X}}\ar[dr]&\\
\iota_{\widetilde{X}\widetilde{V}!}\mathcal{S}^{<D}_{\widetilde{V}}\ar[u]\ar[r]&\iota_{\widetilde{X}\widetilde{V}*}\mathcal{S}^{<D}_{\widetilde{V}}\ar[r]&\iota_{\widetilde{X}\widetilde{V}*}\mathcal{S}^{{\rm mod}D}_{\widetilde{V}},
}
\end{equation}
where the morphism $can:\mathcal{S}^{<D}_{\widetilde{X}}\rightarrow\mathcal{S}^{{\rm mod}D}_{\widetilde{X}}$ is induced by the canonical morphism $\mathcal{A}^{<D}_{\widetilde{X}}\rightarrow\mathcal{A}^{{\rm mod}D}_{\widetilde{X}}$ and the vertical arrows are induced by natural transforms 
$
id_{\widetilde{X}}\rightarrow\iota_{\widetilde{X}\widetilde{V}*}\iota_{\widetilde{X}\widetilde{V}}^{-1}
$ and 
$\iota_{\widetilde{X}\widetilde{V}!}\iota_{\widetilde{X}\widetilde{V}}^{-1}\rightarrow id_{\widetilde{X}}$. By taking hypercohomologies, the diagram (\ref{ThisCD}) induces another commutative diagram
\begin{equation}\label{RDCommutativity}
\xymatrix{
\Homo_*^{r.d.}\left( U;\nabla_f\right)\ar[r]^{can_U}&\Homo_{*}^{mod}\left( U;\nabla_f\right)\ar[d]^{rest}\\
\Homo_{*}^{r.d.}\left( V;\nabla_f\right)\ar[u]^{\iota_{\widetilde{X}\widetilde{V}!}}\ar[r]^{can_V}&\Homo_{*}^{mod}\left( V;\nabla_f\right).
}
\end{equation}

\noindent
Combining (\ref{Adjoint}) and (\ref{RDCommutativity}), we have the following localization formula

\begin{prop}\label{prop:TheLocalizationFormula}
For any $\gamma\in\Homo_{*}^{r.d.}\left( V;\nabla_f^\vee\right)$ and $\gamma^\vee\in\Homo_{2n-*}^{r.d.}\left( V;\nabla_f\right)$, the identity
\begin{equation}
\langle\iota_{\widetilde{X}\widetilde{V}!}(\gamma),can_U\circ\iota_{\widetilde{X}\widetilde{V}!}(\gamma^\vee)\rangle_h=\langle\gamma,can_V(\gamma^\vee)\rangle_h
\end{equation}
holds.
\end{prop}

Finally, we discuss cross products of chains. Let us consider another smooth complex affine variety $W$ and a regular connection $\nabla^\prime=d+\sum_{l=1}^m\beta_l\frac{dg_l}{g_l}\wedge:\mathcal{O}_W\rightarrow\Omega^1_W$. We compactify $W$ into a smooth projective variety $Y$ so that $D^\prime=Y\setminus W$ is a normal crossing divisor. Then, we see that $X\times Y$ is a good compactification of $(\mathcal{O}_{U\times W},\nabla_f+\nabla^\prime)$. We set $D_{U\times W}=X\times Y\setminus U\times W.$  By considering K\"unneth isomorphism $\Homo_{\rm dR}^*\left(U\times W,(\mathcal{O}_{U\times W},\nabla_f+\nabla^\prime)\right)\simeq\bigoplus_{p+q=*}\Homo_{\rm dR}^p\left(U,(\mathcal{O}_{U},\nabla_f)\right)\boxtimes\Homo_{dR}^q\left(W,(\mathcal{O}_{ W},\nabla^\prime)\right)$ induced by the wedge product, we have an isomorphism
\begin{equation}
\mathbb{H}^{*}\left( \widetilde{X\times Y};DR^{<D_{U\times W}}_{\widetilde{X\times Y}}(\nabla_f+\nabla^\prime)\right) \simeq \bigoplus_{p+q=*}\mathbb{H}^{p}\left( \widetilde{X};DR^{<D}_{\widetilde{X}}(\nabla_f)\right)\boxtimes\mathbb{H}^{q}\left( \widetilde{Y};DR^{<D_{W}}_{\widetilde{Y}}(\nabla^\prime)\right)
\end{equation}
induced again by the wedge product. Therefore, if we set $\mathcal{L}_W^\vee=\Ker\left( \nabla^\prime:\mathcal{O}_{W^{an}}\rightarrow\Omega^1_{W^{an}}\right)$, for any $\delta^\vee\in\Homo^{mod}_p\left( U, (\mathcal{O}_U,\nabla_f)\right)$ and $\delta_W^\vee\in\Homo^{mod}_q\left( W, (\mathcal{O}_W,\nabla^\prime)\right)$, we can define the cross product $\delta^\vee\times\delta^\vee_W$ so that the formula
\begin{equation}
\int_{\delta^\vee\times\delta^\vee_W}e^{-f}\prod_{l=1}^kf_l^{-\alpha_l}\prod_{l=1}^mg_l^{-\beta_l}\eta\wedge\eta_W=\left(\int_{\delta^\vee}e^{-f}\prod_{l=1}^kf_l^{-\alpha_l}\eta\right)\left(\int_{\delta_W^\vee}\prod_{l=1}^mg_l^{-\beta_l}\eta_W\right)
\end{equation}

\noindent
is true for any $\eta\in\mathbb{H}^{p}\left( \widetilde{X};DR^{<D}_{\widetilde{X}}(\nabla_f)\right)$ and $\eta_W\in\mathbb{H}^{q}\left( \widetilde{Y};DR^{<D^\prime}_{\widetilde{Y}}(\nabla^\prime)\right)$. Likewise, we can also define the cross product $\gamma\times \gamma_W$ for any $\gamma\in\Homo^{r.d.}_p\left( U,\nabla_f\right)$ and $\gamma_W\in\Homo^{r.d.}_p\left( W,\nabla^\prime\right)$ (\cite[Lemma 2.4]{MatsubaraMellinBarnesKyushu}). If we write $n^\prime $ for the complex dimension of $W$, we have a

\begin{prop}\label{prop:TheProductFormula}
For $\gamma\in\Homo^{r.d.}_p\left( U,(\mathcal{O}_U,\nabla_f)\right)$, $\gamma_W\in\Homo^{r.d.}_q\left( W,(\mathcal{O}_W,\nabla^\prime)\right)$, $\delta^\vee\in\Homo^{mod}_{2n-p}\left( U, (\mathcal{O}_U,\nabla_f)\right)$ and $\delta_W^\vee\in\Homo^{mod}_{2n^\prime-q}\left( W, (\mathcal{O}_W,\nabla^\prime)\right)$, one has an identity
\begin{equation}
\langle \gamma\times\gamma_W,\delta^\vee\times\delta^\vee_W\rangle_h=(-1)^{nn^\prime+pq}\langle \gamma,\delta^\vee\rangle_h\langle \gamma_W,\delta^\vee_W\rangle_h.
\end{equation}
In particular, if $p=n$ and $q=n^\prime$, one has $\langle \gamma\times\gamma_W,\delta^\vee\times\delta^\vee_W\rangle_h=\langle \gamma,\delta^\vee\rangle_h\langle \gamma_W,\delta^\vee_W\rangle_h$.
\end{prop}
\noindent
The readers should be aware of our choice of the orientation of $\C^n.$

\section{Review on the combinatorial structure of series solutions}\label{SectionSeries}

In this section, we briefly recall the construction of a basis of solutions of GKZ system in terms of $\Gamma$-series following the exposition of M.-C. Fern\'andez-Fern\'andez (\cite{FernandezFernandez}). For any commutative ring $R$ and for any pair of finite sets $I$ and $J$, we denote by $R^{I\times J}$ the set of matrices with entries in $R$ whose rows (resp. columns) are indexed by $I$ (resp. $J$). For any univariate function $F$ and for any vector $w={}^t(w_1,\dots,w_n)\in\C^{n\times 1}$, we define $F(w)$ by $F(w)=F(w_1)\cdots F(w_n)$. In this section, $A$ is a $d\times n$ $(d<n)$ integer matrix whose column vectors generate the lattice $\Z^{d\times 1}$. Under this notation, for any vector $v\in\C^{n\times 1}$ such that $Av=-\delta,$ we put
\begin{equation}\label{GammaSeries1}
\varphi_v(z)=\displaystyle\sum_{u\in L_A}\frac{z^{u+v}}{\Gamma(1+u+v)}.
\end{equation}
It can readily be seen that $\varphi_{v}(z)$ is a formal solution of $M_A(\delta)$ (\cite{GKZToral}). We call (\ref{GammaSeries1}) a $\Gamma$-series solution of $M_A(\delta)$.
For any subset $\tau\subset\{1,\dots,n\}$, $A_\tau$ denotes the matrix given by the columns of $A$ indexed by $\tau.$ In the following, we take $\sigma\subset\{1,\dots,n\}$ such that the cardinality $|\sigma|$ is equal to $d$ and $\det A_\sigma\neq 0.$ Note that $A_\s\in\Z^{d\times\s}$ and $A_\s^{-1}\in\Q^{\s\times d}$. Taking a vector ${\bf k}\in\Z^{\bar{\sigma}\times 1},$ we put
\begin{equation}
v_\sigma^{\bf k}=
\begin{pmatrix}
-A_{\sigma}^{-1}(\delta+A_{\bar{\sigma}}{\bf k})\\
{\bf k}
\end{pmatrix},
\end{equation}
where $\bs$ denotes the complement $\{ 1,\dots,n\}\setminus \s$. Then, by a direct computation, we have
\begin{equation}\label{seriesphi}
\varphi_{\s,{\bf k}}(z;\delta)\overset{\rm def}{=}\varphi_{v_\sigma^{\bf k}}(z)=
z_\sigma^{-A_\sigma^{-1}\delta}
\sum_{{\bf k+m}\in\Lambda_{\bf k}}\frac{(z_\sigma^{-A_\sigma^{-1}A_{\bar{\sigma}}}z_{\bar{\sigma}})^{\bf k+m}}{\Gamma({\bf 1}_\sigma-A_\sigma^{-1}(\delta+A_{\bar{\sigma}}({\bf k+m}))){\bf (k+m)!}},
\end{equation}
where $\Lambda_{\bf k}$ is given by
\begin{equation}\label{lambdak}
\Lambda_{\bf k}=\Big\{{\bf k+m}\in\Z^{\bar{\sigma}\times 1}_{\geq 0}\mid A_{\bar{\sigma}}{\bf m}\in\Z A_\sigma\Big\}.
\end{equation}
By (\cite[Lemma 3.1,3.2, Remark 3.4.]{FernandezFernandez}), a complete set of representatives $\{ [A_{\barsigma}{\bf k}(i)]\}_{i=1}^{r_\s}$ of the finite Abelian group $\Z^{d\times 1}/\Z A_\sigma$ induces a decomposition $\Z^{\barsigma\times 1}_{\geq 0}=\bigsqcup_{j=1}^{r_\s}\Lambda_{{\bf k}(j)}.$  Therefore, we can observe that $\{\varphi_{\sigma,{\bf k}(i)}(z;\delta)\}_{i=1}^{r_\s}$ is a set of $r_\s$ linearly independent formal solutions of $M_A(\delta)$ unless $\varphi_{\sigma,{\bf k}(i)}(z;\delta)=0$ for some $i$. In order to ensure that $\varphi_{\s,{\bf k}(i)}(z;\delta)$ does not vanish, we say that a parameter vector $\delta$ is very generic with respect to $\sigma$ if $A_\sigma^{-1}(\delta+A_{\bar{\sigma}}{\bf m})$ does not contain any integer entry for any ${\bf m}\in\mathbb{Z}_{\geq 0}^{\bar{\sigma}\times 1}.$ Using this terminology, we can rephrase the observation above as follows:

\begin{prop}\label{prop:independence}
If $\delta\in\C^{d\times 1}$ is very generic with respect to $\sigma$, $\Big\{\varphi_{\s,{\bf k}(i)}(z;\delta)\Big\}_{i=1}^{r_\s}$ is a linearly independent set of formal solutions of $M_A(\delta)$.
\end{prop}

\noindent

As is well-known in the literature, under a genericity condition, we can construct a basis of holomorphic solutions of  GKZ system $M_A(\delta)$ consisting of $\Gamma$-series with the aid of a regular triangulation. Let us recall the definition of a regular triangulation. In general, for any subset $\sigma$ of $\{1,\dots,n\},$ we write $\cone(\sigma)$ for the positive span of the column vectors of $A$ $\{{\bf a}(1),\dots,{\bf a}(n)\}$ i.e., $\cone(\sigma)=\sum_{i\in\sigma}\R_{\geq 0}{\bf a}(i).$ We often identify a subset $\sigma\subset\{1,\dots,n\}$ with the corresponding set of vectors $\{{\bf a}(i)\}_{i\in\sigma}$ or with the set $\cone(\s)$. A collection $T$ of subsets of $\{1,\dots,n\}$ is called a triangulation if $\{\cone(\sigma)\mid \sigma\in T\}$ is the set of cones in a simplicial fan whose support equals $\cone(A)$. We regard $\Z^{1\times n}$ as the dual lattice of $\Z^{n\times 1}$ via the standard dot product. We write $\pi_A:\Z^{1\times n}\rightarrow L_A^\vee$ for the dual of the natural inclusion $L_A\hookrightarrow \Z^{n\times 1}$ where $L_A^\vee$ is the dual lattice $\Hom_{\Z}(L_A,\Z)$. By abuse of notation, we still write $\pi_A:\R^{1\times n}\rightarrow L_A^\vee\underset{\Z}{\otimes}\R$ for the linear map $\pi_A\underset{\Z}{\otimes}{\rm id}_{\R}$ where ${\rm id}_{\R}:\R\rightarrow\R$ is the identity map. Then, for any generic choice of a vector $\omega\in\pi_A^{-1}\left(\pi_A(\R^{1\times n}_{\geq 0})\right),$ we can define a triangulation $T(\omega)$ as follows: A subset $\sigma\subset\{1,\dots,n\}$ belongs to $T(\omega)$ if there exists a vector ${\bf n}\in\R^{1\times d}$ such that ${\bf n}\cdot{\bf a}(i)=\omega_i$ if $i\in\sigma$ and ${\bf n}\cdot{\bf a}(j)<\omega_j$ if $ j\in\barsigma.$ A triangulation $T$ is called a regular triangulation if $T=T(\omega)$ for some $\omega\in\R^{1\times n}.$ For a fixed regular triangulation $T$, we say that the parameter vector $\delta$ is very generic if it is very generic with respect to any $\sigma\in T$. It is easy to see that if $\delta$ is very generic, $\delta$ must be non-resonant. Now suppose $\delta$ is very generic. Then, it was shown in \cite{FernandezFernandez} that we have $\rank M_A(\delta)=\vol_\Z(\Delta_A).$ Let us put $H_\sigma=\{ j\in\{ 1,\dots, n \}\mid |A_\sigma^{-1}{\bf a}(j)|=1\}$. Here, $|A_\sigma^{-1}{\bf a}(j)|$ denotes the sum of all entries of the vector $A_\sigma^{-1}{\bf a}(j)$. We set 
\begin{equation}
U_\sigma=\left\{z\in(\C^*)^n\mid {\rm abs}\left(z_\sigma^{-A_\sigma^{-1}{\bf a}(j)}z_{j}\right)<R, \text{for all } a(j)\in H_\sigma\setminus\sigma\right\},
\end{equation}
where $R>0$ is a small positive real number and abs stands for the absolute value.

\begin{defn}
A regular triangulation $T$ is said to be convergent if for any $n$-simplex $\s\in T$ and for any $j\in \bs$, one has the inequality $|A_\sigma^{-1}{\bf a}(j)|\leq 1$.
\end{defn}

\begin{rem}
By \cite[Remark 2.1]{FernandezFernandezLocalMonodromy}, there exists at least one convergent regular triangulation.
\end{rem}

With this terminology, the following result is a special case of \cite[Theorem 6.7.]{FernandezFernandez}.

\begin{prop}
Fix a convergent regular triangulation $T$. Assume $\delta$ is very generic. Then, the set 
$\displaystyle\bigcup_{\sigma\in T}\left\{ \varphi_{\sigma,{\bf k}(i)}(z;\delta)\right\}_{i=1}^{r_\s}$ 
is a basis of holomorphic solutions of $M_A(\delta)$ on $U_{T}\overset{def}{=}\displaystyle\bigcap_{\sigma\in T}U_\sigma\neq\varnothing$ where $r_\s=\vol_\Z(\sigma)=|\Z^{d\times 1}/\Z A_\sigma|.$ 

\end{prop}

\begin{rem}
We define an $n\times\bs$ matrix $B_\sigma$ by
\begin{equation}
B_\sigma=
\begin{pmatrix}
-A_\sigma^{-1}A_{\barsigma}\\
{\bf I}_{\barsigma}
\end{pmatrix}
\end{equation}
and a cone $C_\sigma$ by
\begin{equation}
C_\sigma=\Big\{ \omega\in\R^{1\times n}\mid \omega\cdot B_\sigma>0\Big\}.
\end{equation}
Here, ${\bf I}_{\barsigma}\in\Z^{\{d+1,\dots,n\}\times\bs}$ is the identity matrix. Then, $T$ is a regular triangulation if and only if $C_{T}\overset{def}{=}\displaystyle\bigcap_{\sigma\in T}C_\sigma$ is a non-empty open cone. In this case, the cone $C_T$ is  characterized by the formula
\begin{equation}
C_T=\Big\{ \omega\in\R^{1\times n}\mid T(\omega)=T\Big\}.
\end{equation}
From the definition of $U_\sigma$, we can confirm that $z$ belongs to $U_T$ if $(-\log|z_1|,\dots,-\log|z_n|)$ belongs to a sufficiently far translation of $C_T$ inside itself, which implies $U_T\neq\varnothing.$
\end{rem}

We conclude this section by quoting a result of Gelfand, Kapranov, and Zelevinsky (\cite[Chapter 7, Proposition 1.5.]{GKZbook},\cite[Theorem 5.2.11.]{DeLoeraRambauSantos}).

\begin{thm}[\cite{GKZbook},\cite{DeLoeraRambauSantos}]
There exists a polyhedral fan ${\rm Fan}(A)$ in $\R^{1\times n}$ whose support is \newline $\pi_A^{-1}\left(\pi_A(\R^{1\times n}_{\geq 0})\right)$ and whose maximal cones are exactly $\{ C_T\}_{T: \text{regular triangulation}}$. The fan ${\rm Fan}(A)$ is called the secondary fan.
\end{thm}

\begin{rem}
Let $F$ be a fan obtained by applying the projection $\pi_A$ to each cone of ${\rm Fan}(A)$. By definition, each cone of ${\rm Fan}(A)$ is a pull-back of a cone of $F$ through the projection $\pi_A$. Therefore, the fan $F$ is also called the secondary fan.
\end{rem}

\section{Combinatorial construction of integration contours via regular triangulations}\label{SectionEuler}
In this section, we construct integration contours associated to Euler-Laplace integral representation
\begin{equation}\label{EulerInt2}
f_{\Gamma}(z)=\frac{1}{(2\pi\ii)^{n+k}}\int_\Gamma e^{h_{0,z^{(0)}}(x)} h_{1,z^{(1)}}(x)^{-\gamma_1}\cdots h_{k,z^{(k)}}(x)^{-\gamma_k}x^{c}\frac{dx}{x}.
\end{equation}
with the aid of a convergent regular triangulation. Without loss of generality, we may assume $N_l\geq 2$ for any $l=1,\dots,k$. This is because $N_l=1$ implies that the corresponding Laurent polynomial $h_{l,z^{(l)}}$ is a monomial hence (\ref{EulerInt2}) is reduced to the integral with $k-1$ powers of Laurent polynomials. 

\noindent
We write ${\bf e}_l\;(l=1,\dots,k)$ for the standard basis of $\Z^{k\times 1}$, and put ${\bf e}_0=0\in \Z^{k\times 1}$. We set $I_l=\{ N_0+\dots+N_{l-1}+1,\dots,N_0+\dots+N_l\}$ or equivalently, $I_l=\left\{
\begin{pmatrix}
{\bf e}_l\\
\hline
{\bf a}^{(l)}(j)
\end{pmatrix}
\right\}_{j=1}^{N_l}$ ($l=0,\dots,k$). In the following we fix an $(n+k)$-simplex $\s\subset\{1,\dots,N\}$, i.e., a subset with cardinality $n+k$ and $\det A_\s\neq 0$. We also assume an additional condition $|A_\s^{-1}{\bf a}(j)|\leq 1$ for any $j\in\bs$. According to the partition $\{1,\dots,N\}=I_0\cup\dots\cup I_k$, we have an induced partition $\s=\s^{(0)}\cup\dots\cup\s^{(k)}$, where $\s^{(l)}=\s\cap I_l$. The symbol $\bs^{(l)}$ denotes the complement $I_l\setminus\s^{(l)}$. Since $\det A_\s\neq 0$, we have $\s^{(l)}\neq \varnothing$ for any $l=1,\dots,k$. For any finite set $S$, we write $|S|$ for the cardinality of $S$.

Let us consider an $n$-dimensional projective space $\mathbb{P}^n$ with a homogeneous  coordinate $\tau=[\tau_0:\cdots:\tau_n]$. Let $\alpha_0,\dots,\alpha_{n+1}\in\C$ be parameters such that $\alpha_0+\dots+\alpha_{n+1}=1$ and $\omega(\tau)$ be the section of $\Omega_{\mathbb{P}^n}^n(n+1)$ defined by $\omega(\tau)=\sum_{i=0}^n(-1)^i\tau_i d\tau_0\wedge\dots\wedge\widehat{d\tau_i}\wedge\dots\wedge d\tau_n$. We consider an affine open set $\{\tau_0\neq 0\}$. We define the coordinate $t=(t_1,\dots,t_n)$ of $\{\tau_0\neq 0\}$ by $\frac{\tau_i}{\tau_0}=e^{\pi\ii}t_i$ and $t_{n+1}$ by $t_{n+1}=1-t_1\dots- t_n$. Let $P_{\tau}$ denote the $n$-dimensional Pochhammer cycle constructed as in \cite[\S 6]{Beukers} (see also the appendix of this paper). Then we have the following
\begin{lem}\label{lemma:ProjBeukers}(\cite[Proposition 6.1]{Beukers})
For any complex numbers $\alpha_0,\dots,\alpha_{n+1}\in\C$ such that $\alpha_0+\cdots+\alpha_{n+1}=1,$ one has
\begin{align}
\int_{P_\tau}\tau_0^{\alpha_0-1}\cdots\tau_n^{\alpha_n-1}(\tau_0+\dots+\tau_n)^{\alpha_{n+1}-1}\omega(\tau)
&=\frac{(2\pi\ii)^{n+1}e^{-\pi\ii\alpha_{n+1}}}{\Gamma(1-\alpha_0)\cdots\Gamma(1-\alpha_{n+1})}.
\end{align}
\end{lem}

Now we consider projective spaces $\PP^{|\s^{(l)}|-1}$. Writing $\s^{(l)}=\{i^{(l)}_0,\dots,i^{(l)}_{|\s^{(l)}|-1}\}$ so that $i^{(l)}_0<\dots<i^{(l)}_{|\s^{(l)}|-1}$, we equip $\PP^{|\s^{(l)}|-1}$ with a homogeneous coordinate $[\tau_{\s^{(l)}}]=\left[\tau_{i^{(l)}_0}:\dots:\tau_{i^{(l)}_{|\s^{(l)}|-1}}\right]$. Here, we use the convention $\PP^0=\{ *\}$ (one point). We define a covering map 
$
p:(\C^\times)^n_x\rightarrow(\C^\times)^{\s^{(0)}}_{\xi_{\s^{(0)}}}\times\prod_{l=1}^k\left(\PP^{|\s^{(l)}|-1}_{\tau_{\s^{(l)}}}\setminus\bigcup_{i\in\s^{(l)}}\{ \tau_i=0\}\right)
$
by $p(x)=\left( z_{\s^{(0)}}({\bf 1}_k,x)^{A_{\s^{(0)}}}, \left([z_{\s^{(l)}}\cdot({\bf 1}_k,x)^{A_{\s^{(l)}}}]\right)_{l=1}^k\right)$, where ${\bf 1}_k=\overbrace{(1,\dots,1)}^k$ and $z_{\s^{(l)}}({\bf 1}_k,x)^{A_{\s^{(l)}}}=$ $\left( z_{i}({\bf 1}_k,x)^{{\bf a}^{(l)}(i)}\right)_{i\in\s^{(l)}}$ for $l=0,\dots,k$. We define $\omega(\tau_{\s^{(l)}})$ by $\omega(\tau_{\s^{(l)}})=\sum_{j=0}^{|\s^{(l)}|-1}(-1)^j\tau_{i_j} d\tau_{i_0}\wedge\dots\wedge\widehat{d\tau_{i_j}}\wedge\dots\wedge d\tau_{i_{|\s^{(l)}|-1}}.$ We write $\tau_{\s^{(l)}}$ for the product $\prod_{j=0}^{|\s^{(l)}|-1}\tau_{i_j}$. We set $\tau_\s=\prod_{l=1}^k\tau_{\s^{(l)}}.$ By a direct computation employing Laplace expansion, we have the identity
\begin{equation}
p^*\left(\frac{d\xi_{\s^{(0)}}}{\xi_{\s^{(0)}}}\wedge\frac{\omega(\tau_\s)}{\tau_\s}\right)=p^*\left(\frac{d\xi_{\s^{(0)}}}{\xi_{\s^{(0)}}}\wedge\frac{\omega(\tau_{\s^{(1)}})}{\tau_{\s^{(1)}}}\wedge\cdots\wedge\frac{\omega(\tau_{\s^{(k)}})}{\tau_{\s^{(k)}}}\right)=\text{sgn} (A,\s)(\det A_\s)\frac{dx}{x},
\end{equation}
where we have put $\text{sgn} (A,\s)=(-1)^{k|\s^{(0)}|+(k-1)|\s^{(1)}|+\dots+|\s^{(k-1)}|+\frac{k(k-1)}{2}}$.

Now we use the plane wave expansion coordinate. Let us introduce a coordinate transform of $\xi_{\s^{(0)}}$ by $\xi_i=\rho u_i\;\;(i\in\sigma^{(0)}),$
 where $\rho$ and $u_i$ are coordinates of $\C^\times$ and $\{u_{\s^{(0)}}=(u_i)_{i\in\s^{(0)}}\in(\C^\times)^{\s^{(0)}}\mid\displaystyle\sum_{i\in\sigma^{(0)}}u_i=1\}$ respectively. Then, it is standard that we have an equality of volume forms $d\xi_{\s^{(0)}}=\rho^{|\s^{(0)}|-1}d\rho du_{\s^{(0)}},$ 
where $du_{\s^{(0)}}=\sum_{j=1}^{|\s^{(0)}|}(-1)^{j-1}u_{i_j}du_{\widehat{i_j}}$ with $du_{\widehat{i_j}}=du_{i_1}\wedge\cdots\wedge \widehat{du_{i_j}}\wedge\cdots\wedge du_{i_{|\s^{(0)}|}}$ and $\s^{(0)}=\{ i_1,\dots,i_{|\s^{(0)}|}\}$ ($i_1<\dots<i_{|\s^{(0)}|}$). For any vector $v\in\C^{(n+k)\times 1}$, the $i$($\in\s$)-th entry of $A_\s^{-1}v\in\C^{\s\times 1}$ is denoted by the symbol $p_{\s i}(v)$. For any $v\in\C^{(n+k)\times 1}$ and $l=0,1,\dots,k$, we set $S_l(v)\overset{def}{=}\sum_{i\in\s^{(l)}}p_{\s i}(v)$. Using these formulae and notation, we obtain 

\begin{align}
f_{\Gamma}(z)
=&\frac{sgn(A,\s)}{\det A_\s}\frac{z_\s^{-A_\s^{-1}\delta}}{(2\pi\ii)^{n+k}}\displaystyle\int_{p_*\Gamma}\prod_{l=1}^k\left(\sum_{i\in\s^{(l)}}\tau_i+\sum_{j\in\bar{\s}^{(l)}}z_\s^{-A_\s^{-1}{\bf a}(j)}z_j(\xi_{\s^{(0)}},\tau_\s)^{A_\s^{-1}{\bf a}(j)}\right)^{-\gamma_l}\times\nonumber\\
 &\exp\left\{ \sum_{i\in\s^{(0)}}\xi_i+\sum_{j\in\bs^{(0)}}z_\s^{-A_\s^{-1}{\bf a}(j)}z_j(\xi_{\s^{(0)}},\tau_\s)^{A_\s^{-1}{\bf a}(j)}\right\}
(\xi_{\s^{(0)}},\tau_\s)^{A_\s^{-1}\delta}\frac{d\xi_{\s^{(0)}}\omega(\tau_\s)}{\xi_{\s^{(0)}}\tau_\s}\\
 =&\frac{sgn(A,\s)}{\det A_\s}\frac{z_\s^{-A_\s^{-1}\delta}}{(2\pi\ii)^{n+k}}\displaystyle\int_{p_*\Gamma}\prod_{l=1}^k\left(\sum_{i\in\s^{(l)}}\tau_i+\sum_{j\in\bar{\s}^{(l)}}z_\s^{-A_\s^{-1}{\bf a}(j)}z_j\rho^{S_0({\bf a}(j))}(u_{\s^{(0)}},\tau_\s)^{A_\s^{-1}{\bf a}(j)}\right)^{-\gamma_l}\times\nonumber\\
 &\exp\left\{ 
\rho+\sum_{j\in\bs^{(0)}}z_\s^{-A_\s^{-1}{\bf a}(j)}z_j
\rho^{S_0({\bf a}(j))}
(u_{\s^{(0)}},\tau_\s)^{A_\s^{-1}{\bf a}(j)}
\right\}
\rho^{S_0(\delta)}
(u_{\s^{(0)}},\tau_\s)^{A_\s^{-1}\delta}\nonumber\\
&
\frac{d\rho d u_{\s^{(0)}} \omega(\tau_\s) }{\rho u_{\s^{(0)}} \tau_\s},\label{PlaneWaveIntegral}
\end{align}
where $\Gamma$ is an integration contour to be clarified below. We have also used the convention that $\tau_i$ for $i\in\s^{(l)}$ with $|\s^{(l)}|=1$ is equal to $z_i({\bf 1}_k,x)^{{\bf a}(i)}$.

Let us construct the cycle $\Gamma$. For this purpose, we consider a degeneration of the  the integrand $\Phi$. Namely, we consider the following limit: variables $z_j\neq 0$ with $j\in\bs $ are very small while variables $z_j\neq 0$ with $j\in\s$ are frozen. Symbolically, we write this limit as $z\approx z_{\infty}^\s$. The corresponding degeneration of the integrand is 
\begin{equation}
\Phi\approx e^{\sum_{i\in\s^{(0)}}z_ix^{{\bf a}^{(0)}(i)}}\left(\displaystyle\sum_{i\in\s^{(1)}}z_ix^{{\bf a}^{(1)}(i)}\right)^{-\gamma_1}\dots \left(\displaystyle\sum_{i\in\s^{(k)}}z_ix^{{\bf a}^{(k)}(i)}\right)^{-\gamma_k}.
\end{equation} We first set $\rho=e^{-\pi\ii}$ and construct a cycle in $u_{\s^{(0)}}$ and $\tau_\s$ directions. We take a cycle $\Gamma_{0}$ in $\{\rho=e^{-\pi\ii}\}\times\displaystyle\prod_{l=1}^k\left(\PP^{|\s^{(l)}|-1}_{\tau_{\s^{(l)}}}\setminus\bigcup_{i\in\s^{(l)}}\{ \tau_i=0\}\cup\left\{ \sum_{i\in\s^{(l)}}\tau_i=0\right\}\right)$ as a product cycle $\Gamma_0=P_{u_\s^{(0)}}\times\displaystyle\prod_{l=1}^kP_{\tau_{\s^{(l)}}}$.  We take a $(n-1)$-dimensional twisted cycle $\tilde{\Gamma}_{\s,0}$ in $\{ \rho=e^{-\pi\ii}\}\subset(\C)^n_x$ so that $p_*\tilde{\Gamma}_{\s,0}=\Gamma_0.$ For the construction of such a cycle, see Appendix 3. Note that we  determine the branch of multivalued functions $h_{l,z^{(l)}}(x)^{-\gamma_l}$ so that the expansion 

\begin{equation}\label{TheExpansion}
h_{l,z^{(l)}}^{-\gamma_l}(x)=
\sum_{{\bf m}_l\in\Z^{\bar{\s}^{(l)}}_{\geq 0}}
\frac{(-1)^{|{\bf m}_l|}(\gamma_l)_{|{\bf m}_l|}}{{\bf m}_l!}
\left(\sum_{i\in\s^{(l)}}z_ix^{{\bf a}^{(l)}(i)}\right)^{-\gamma_l-|{\bf m}_l|}
z_{\bs^{(l)}}^{{\bf m}_l}
({\bf 1}_k,x)^{A_{\bs}{\bf m}_l}
\end{equation}
is valid when $z\approx z^\s_\infty$. Thus, the branch of $h_{l,z^{(l)}}(x)^{-\gamma_l}$ is determined by that of $\left(\displaystyle\sum_{i\in\s^{(l)}}z_ix^{{\bf a}^{(l)}(i)}\right)^{-\gamma_l}$, which is determined by the choice of $\tilde{\Gamma}_{\s,0}$. 

In $\rho$ direction, we take the so-called Hankel contour $C_0$. $C_0$ is given by the formula
$C_0=(-\infty,-\delta]e^{-\pi\sqrt{-1}}+l_{(0+)}-(-\infty,-\delta]e^{\pi\sqrt{-1}},$
where $e^{\pm\pi\sqrt{-1}}$ stands for the argument of the variable and $l_{(0+)}$ is a small loop which encircles the origin in the counter-clockwise direction starting from and ending at the point $-\delta$ for some small positive $\delta$. Using this notation, we have 

\begin{lem}\label{lemma:lemma}
Suppose $\alpha\in\mathbb{C}.$ One has an identity
\begin{equation}
\int_{C_{0}}\xi^{\alpha-1}e^\xi d\xi=\frac{2\pi\sqrt{-1}}{\Gamma(1-\alpha)}.
\end{equation}
\end{lem}

\begin{figure}[t]
\begin{minipage}{0.5\hsize}
\begin{center}
\begin{tikzpicture}
\draw[->-=.5,domain=-175:175] plot ({-1+cos(\x)}, {sin(\x)});
\draw[-<-=.5] ({-1+cos(-175)},{sin(-175)}) -- (-6, {sin(-175)});
\draw[->-=.5] ({-1+cos(175)},{sin(175)}) -- (-6, {sin(175)});
\node at (-1,0){$\cdot$};
\draw (-1,0) node[below right]{O};
\end{tikzpicture}
\caption{Hankel contour}
\end{center}
\end{minipage}
\begin{minipage}{0.5\hsize}
\begin{center}
\begin{tikzpicture}
\node at (0,0){$\cdot$};
\node at (4,0){$\cdot$};
\draw (0,0) node[below]{$t_1=0$};
\draw (4,0) node[below]{$t_1=1$};
\draw[-<-=.5,domain=-140:140] plot ({4+cos(\x)}, {0.4+sin(\x)});
\draw[-<-=.5,domain=40:320] plot ({cos(\x)}, {0.4+sin(\x)});
\draw[->-=.5,domain=-140:140] plot ({4+cos(\x)}, {-0.4+sin(\x)});
\draw[->-=.5,domain=40:320] plot ({cos(\x)}, {-0.4+sin(\x)});
    \coordinate (A1) at ({cos(40)}, {0.4+sin(40)});
    \coordinate (A2) at ({4+cos(-140)}, {-0.4+sin(-140)});
    \coordinate (B1) at ({4+cos(140)}, {-0.4+sin(140)});
    \coordinate (B2) at ({cos(40)}, {-0.4+sin(40)});
    \coordinate (C1) at ({cos(320)}, {-0.4+sin(320)});
    \coordinate (C2) at ({4+cos(140)}, {0.4+sin(140)});
    \coordinate (D1) at ({4+cos(-140)}, {0.4+sin(-140)});
    \coordinate (D2) at ({cos(320)}, {0.4+sin(320)});
\draw[->-=.75] (A1) -- (A2);
\draw[->-=.5] (B1) -- (B2);
\draw[->-=.75] (C1) -- (C2);
\draw[->-=.5] (D1) -- (D2);
\end{tikzpicture}
\caption{Pochhammer cycle $P_1$}
\end{center}
\end{minipage}
\end{figure}

\noindent
We wish to integrate the integrand along the product contour $C_0\times \Gamma_0$. To do this, we need a simple

\begin{lem}\label{lem:sum2}
For any $l=1,\cdots,k$ and for any $j\in\barsigma^{(l)}$, one has
\begin{equation}\label{Formula4.17}
S_m({\bf a}(j))=
\begin{cases}
1\; (m=l)\\
0\; (m\neq 0,l).
\end{cases}
\end{equation}
Moreover, if $j\in\barsigma^{(0)}$, one has
\begin{equation}\label{WeirdEq}
S_m({\bf a}(j))=
0\;\; (m=1,\dots,k). 
\end{equation}
\end{lem}

\begin{proof}
Observe first that, if we write $A$ as $A=({\bf a}(1)|\cdots|{\bf a}(N)),$ then for any $j\in\bar{\sigma}^{(l)}$ ($l=1,\dots,k$) and $m=1,\dots,k$, we have
\begin{equation}
\transp{
\begin{pmatrix}
{\bf e}_m\\
\hline
O
\end{pmatrix}
}
{\bf a}(j)=
\begin{cases}
1\; (m=l)\\
0\; (m\neq l)
\end{cases}
\end{equation}

\noindent
This can be written as 
\begin{equation}
\left(
\begin{array}{c|c}
I_k& \\
\hline
 &\mathbb{O}_n
\end{array}
\right)
{\bf a}(j)=
\begin{pmatrix}
{\bf e}_l\\
\hline
O
\end{pmatrix}.
\end{equation}

\noindent
We thus have
$
\begin{pmatrix}
{\bf e}_l\\
\hline
O
\end{pmatrix}
=
\left(
\begin{array}{c|c}
I_k& \\
\hline
 &\mathbb{O}_n
\end{array}
\right)
{\bf a}(j)
=
\left(
\begin{array}{c|c}
I_k& \\
\hline
 &\mathbb{O}_n
\end{array}
\right)
A_\sigma
A_\sigma^{-1}
{\bf a}(j),
$
which clearly shows (\ref{Formula4.17}). On the other hand, for any $j\in\bar{\sigma}^{(0)}$ we have
\begin{equation}
\transp{
\begin{pmatrix}
{\bf e}_m\\
\hline
O
\end{pmatrix}
}
{\bf a}(j)=
0\;\;(m=1,\dots,k).
\end{equation}
Thus, the same argument as above shows (\ref{WeirdEq}).
\end{proof}

\noindent
From \cref{lem:sum2} and the equality
$
\sum_{m=0}^kS_m({\bf a}(j))=|A_\sigma^{-1}{\bf a}(j)|,
$
we obtain two inequalities on the degree of divergence
$
S_0({\bf a}(j))\leq 0\;\;(j\in\barsigma^{(l)},\; l=1,\dots,k)
$
and
$
S_0({\bf a}(j))\leq 1\;\;(j\in\barsigma^{(0)}).
$
From these inequalities we can verify that the expansion (\ref{TheExpansion}) is valid uniformly along $C_0\times \Gamma_0$ and the integral (\ref{PlaneWaveIntegral}) is convergent if $z\approx z_\infty^\s$.
In order to define the lift of the product cycle $C_0\times\Gamma_0$ to $x$ coordinate, we need a 
\begin{lem}\label{lem:ThomMather}
Let $z_j\neq 0$ ($j=1,\dots,N$) be complex numbers and let $\varphi(x)=\sum_{j=1}^Nz_jx^{{\bf a}(j)}$ be a Laurent polynomial in $x=(x_1,\dots,x_n)$. If there is a vector $w=(w_1,\dots,w_n)\in\Z^{1\times n}$ and an integer $m\in\Z\setminus\{ 0\}$ such that for any $j$, one has $w\cdot {\bf a}(j)=m$, then the smooth map $\varphi:\varphi^{-1}(\C^\times)\rightarrow\C^\times$ is a fiber bundle.
\end{lem}

\begin{proof}
Define an action of a torus $\C^\times_\tau$ on $(\C^\times)^n_x$ (resp. on $\C^\times_t$) by $\tau\cdot x=(\tau^{w_1}x_1,\dots,\tau^{w_n}x_n)$ (resp. by $\tau\cdot t=\tau^mt$). Then, it can readily be seen that for any $\tau\in\C^\times$ and $t\in\C^\times$, we have $\tau\cdot\varphi^{-1}(t)=\varphi^{-1}(\tau\cdot t)$. Therefore, if $\varphi$ is a trivial fiber bundle on an open set $U\subset\C^\times_t$, it is again trivial on the open subset $\tau\cdot U$. By Thom-Mather's 1st isotopy lemma (\cite[(4.14) Th\'eor\`eme]{Verdier}), $\varphi$ defines a locally trivial fiber bundle on a non-empty Zariski open subset of $\C^\times_t$. Thus, we can conclude that $\varphi$ is locally trivial on $\C^\times_t$.
\end{proof}

\noindent
In view of \cref{lem:ThomMather}, let us define the twisted cycle $\Gamma_{\s,0}$ as the prolongation of $\tilde{\Gamma}_{\s,0}$ along the Hankel contour $C_0$ with respect to the map $\rho=\sum_{i\in\s^{(0)}}z_j({\bf 1}_k,x)^{{\bf a}(i)}:(\C^\times)^n_x\rightarrow\C$. Computing the integral on this contour, we obtain

\begin{align}
&f_{\s,0}(z;\delta)\nonumber\\
\overset{\rm def}{=}&f_{\Gamma_{\s,0}}(z)\\
=&\frac{sgn(A,\s)}{\det A_\s}\frac{z_\s^{-A_\s^{-1}\delta}}{(2\pi\ii)^{n+k}}\displaystyle\sum_{{\bf m}\in\Z_{\geq 0}^{\bs}}\frac{\prod_{l=1}^k(-1)^{|{\bf m}_l|}(\gamma_l)_{|{\bf m}_l|}}{{\bf m}!}(z_\s^{-A_\s^{-1} A_{\bs}}z_{\bs})^{{\bf m}}\nonumber\\
 &\int_{C_0\times \Gamma_0}\prod_{l=1}^k\left(\sum_{i\in\s^{(l)}}\tau_i\right)^{-\gamma_l-|{\bf m}_l|}e^{\rho}
\rho^{
S_0(\delta+A_{\bs}{\bf m})
}
(u_{\s^{(0)}},\tau_\s)^{A_\s^{-1}(\delta+A_{\bs}{\bf m})}
\frac{d\rho du_{\s^{(0)}}\omega(\tau_\s)}{\rho u_{\s^{(0)}}\tau_\s}.
\end{align}

\noindent
We put 
$
\tilde{\bf e}_l=
\begin{pmatrix}
{\bf e}_l\\
{\bf O}
\end{pmatrix}
\in\Z^{(k+n)\times 1}
$. Since
$
{}^t\tilde{\bf e}_l={}^t\tilde{\bf e}_lA_\s A_\s^{-1}=\displaystyle\sum_{i\in\s^{(l)}}{}^t{\bf e}_iA_\s^{-1},
$
we have
$
S_l(\delta+A_{\bs}{\bf m})={}^t\tilde{\bf e}_l(\delta+A_{\bs}{\bf m})=\gamma_l+|{\bf m}_l|
$ for any $l=1,\dots,k$.
Therefore, the assumption on the parameters in \cref{lemma:ProjBeukers} is satisfied. Moreover, in view of \cref{lem:sum2}, for any $l\geq 1$ such that $|\s^{(l)}|=1$, we also have that if $\{ i\}=\s^{(l)}$ then $p_{\s i}(v)=v_l$ for any $v\in\C^{(n+k)\times 1}$ and $\Gamma(1-p_{\s i}(\delta+A_{\bs}{\bf m}))=\Gamma(1-\gamma_l-|{\bf m}_l|)$. Let $\{ A_{\bs}{\bf k}(i)\}_{i=1}^{r_\s}$ be a complete system of representatives of $\Z^{(n+k)\times 1}/\Z A_\s$. Using \cref{lemma:ProjBeukers}, Lemma {lemma:lemma} and employing the formula
\begin{equation}
(\gamma_l)_{|{\bf m}_l|}=\frac{2\pi\ii e^{-\pi\ii\gamma_l}(-1)^{|{\bf m}_l|}}{\Gamma(\gamma_l)\Gamma(1-\gamma_l-|{\bf m}_l|)(1-e^{-2\pi\ii\gamma_l})},
\end{equation}
we obtain the basic formula
\begin{equation}
f_{\s,0}(z;\delta)=C_\s(\gamma)\sum_{i=1}^{r_{\s}}\varepsilon_\s(\delta,{\bf k}(i))\varphi_{\s,{\bf k}(i)}(z).
\end{equation}

\noindent
Here, we have put
\begin{equation}
C_\s(\gamma)=
\frac{{\rm sgn}(A,\s)\displaystyle\prod_{l:|\s^{(l)}|>1}e^{-\pi\ii(1-\gamma_l)}\displaystyle\prod_{l:|\s^{(l)}|=1}e^{-\pi\ii\gamma_l}}{\det A_\s\Gamma(\gamma_1)\dots\Gamma(\gamma_k)\displaystyle\prod_{l:|\s^{(l)}|=1}(1-e^{-2\pi\ii\gamma_l})}
\end{equation}
and
\begin{equation}
\varepsilon_\s(\delta,{\bf k})=
\begin{cases}
1&(|\s^{(0)}|\leq 1)\\
 1-\exp\left\{ -2\pi\ii S_0(\delta+A_{\bs}{\bf k})\right\}&(|\s^{(0)}|\geq 2).
\end{cases}
\end{equation}

To any integer vector $\tilde{\bf k}\in\Z^{\s\times 1}$, we associate a deck transform $\Gamma_{\s,\tilde{\bf k}}$ of $\Gamma_{\s,0}$ along the loop $(\xi_{\s^{(0)}},[\tau_\s])\mapsto e^{2\pi\ii{}^t\tilde{\bf k}}(\xi_{\s^{(0)}},[\tau_\s])$. If there is a label $\s^{(l)}$ such that $\s^{(l)}=\{ i\}$, the corresponding entry $\tilde{k}_{i}$ of $\tilde{\bf k}$ does not contribute to the deck transformation $\Gamma_{\s,\tilde{\bf k}}$ in the usual sense. However, we associate the scalar multiplication $e^{2\pi\ii k_i\gamma_l}$ for such an entry in the following discussion. By a direct computation, we have
\begin{align}
f_{\s,\tilde{\bf k}}(z;\delta)\overset{def}{=}&f_{\Gamma_{\s,\tilde{\bf k}}}(z)\nonumber\\
 =&e^{2\pi\ii{}^t\tilde{\bf k}A_\s^{-1}\delta}C_\s(\gamma)\sum_{i=1}^{r_\s}e^{2\pi\ii{}^t\tilde{\bf k}A_\s^{-1}A_{\bs}{\bf k}(i)}
\varepsilon_\s(\delta,{\bf k}(i))
\varphi_{\s,{\bf k}(i)}(z;\delta).
\end{align}
\noindent
We take a complete system of representatives $\{ \tilde{\bf k}(i)\}_{i=1}^{r_\s}$ of $\Z^{\s\times 1}/\Z{}^tA_\s$. Since it can readily be seen that the pairing $\Z^{\s\times 1}/\Z {}^tA_\s\times\Z^{(n+k)\times 1}/\Z A_\s\ni ([\tilde{\bf k}],[{\bf k}])\mapsto {}^t\tilde{\bf k}A_\s^{-1}{\bf k}\in\Q/\Z$ is perfect in the sense of Abelian groups, we can easily see that the matrix 
$U_\s\overset{def}{=}\Big(
\exp\left\{
2\pi\ii\transp{
\tilde{\bf k}(i)
}
A_{\s}^{-1}A_{\bs}{\bf k}(j)
\right\}
\Big)_{i,j=1}^{r_\s}$ is the character matrix of the finite Abelian group $\Z^{(n+k)\times 1}/\Z A_\s$, hence it is invertible. 

Let us take a convergent regular triangulation $T$. With the aid of the trivialization (\ref{Trivialization}), we can take a parallel transport of $\Gamma_{\s,\tilde{\bf k}(j)}$ constructed near $z^\s_\infty$ to a point $z_\infty\in U_T$. The resulting cycle is also denoted by $\Gamma_{\s,\tilde{\bf k}(j)}$. It is worth pointing out that the cycles $\Gamma_{\s,\tilde{\bf k}(j)}$ constructed above are locally finite cycles rather than finite ones. It is routine to regard $\Gamma_{\s,\tilde{\bf k}(j)}$ as a rapid decay cycle: We use the notation of \S\ref{subsec:3.2}. For simplicity, let us assume that $z\approx z^\s_\infty$ is nonsingular. Then, we regard $\Gamma_{\s,\tilde{\bf k}(j)}$ as a subset of $\tilde{p}^{-1}(z)$ and take its closure $\overline{\Gamma_{\s,\tilde{\bf k}(j)}}\subset\tilde{p}^{-1}(z)$. By construction, $\overline{\Gamma_{\s,\tilde{\bf k}(j)}}\subset \pi^{-1}(z)\cup\widetilde{D^{r.d.}_0}$. This is (a closure of) a semi-analytic set. By \cite[THEOREM 2.]{LojasiewiczTriangulations}, we can obtain a semi-analytic triangulation of $\overline{\Gamma_{\s,\tilde{\bf k}(j)}}$ which makes it an element of $\Homo_{n,z}^{r.d.}$ in view of \cref{rem:RemarkRDHomology}. 

\begin{figure}[b]
\begin{center}
\begin{tikzpicture}
\draw[-] (0,0)--(4,-1);
\draw[-] (0,0)--(4,3);
\draw[-] (0,0)--(3,4);
\draw[-] (0,0)--(-1,4);
\node at (3,3){$\cdot$};
\node at (3.3,3){$z_\infty$};
\node at (3,0){$\cdot$};
\node at (0,3){$\cdot$};
\node at (3,-0.5){$z_\infty^\s$};
\node at (0,2.5){$z_\infty^{\s^\prime}$};
\node at (1,0){$U_\s$};
\node at (0.3,1.3){$U_{\s^\prime}$};
\node at (1.6,1.6){$U_T$};
\draw[->-=.5] (3,0) to [out=120,in=200] (3,3);
\draw[->-=.5] (0,3) to [out=40,in=100] (3,3);
\draw[<->,domain=-15:50] plot ({0.5*cos(\x)}, {0.5*sin(\x)});
\draw[<->,domain=40:100] plot ({cos(\x)}, {sin(\x)});
\draw[<->,domain=40:50] plot ({1.5*cos(\x)}, {1.5*sin(\x)});
\end{tikzpicture}
\caption{Parallel transport}
\end{center}
\end{figure}
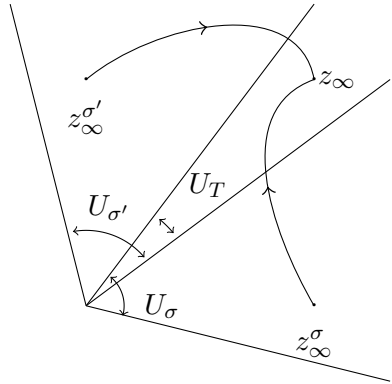

 Summing up all the arguments above and taking into account \cref{thm:EulerLaplaceRepresentationTheorem}, we obtain the main 

\begin{thm}\label{thm:fundamentalthm3}
Take a convergent regular triangulation $T$. Assume that the parameter vector $\delta$ is very generic and that $\gamma_l\notin\Z_{\leq 0}$ for any $l=1,\dots,k$. Then, if one puts
\begin{equation}
f_{\s,\tilde{\bf k}(j)}(z;\delta)=\frac{1}{(2\pi\ii)^{n+k}}\int_{\Gamma_{\s,\tilde{\bf k}(j)}} e^{h_{0,z^{(0)}}(x)}h_{1,z^{(1)}}(x)^{-\gamma_1}\cdots h_{k,z^{(k)}}(x)^{-\gamma_k}x^{c}\frac{dx}{x},
\end{equation}
$\displaystyle\bigcup_{\s\in T}\{ f_{\sigma,\tilde{\bf k}(j)}(z)\}_{j=1}^{r_\s}$ is a basis of solutions of $M_A(\delta)$ on the non-empty open set $U_T$, where $\{\tilde{\bf k}(j)\}_{j=1}^{r_\s}$ is a complete system of representatives of $\Z^{\s\times 1}/\Z{}^tA_\s$. Moreover, for each $\sigma\in T,$ one has a transformation formula 
\begin{equation}
\begin{pmatrix}
f_{\sigma,\tilde{\bf k}(1)}(z;\delta)\\
\vdots\\
f_{\sigma,\tilde{\bf k}(r_\s)}(z;\delta)
\end{pmatrix}
=
T_\sigma
\begin{pmatrix}
\varphi_{\sigma,{\bf k}(1)}(z;\delta)\\
\vdots\\
\varphi_{\sigma,{\bf k}(r_\s)}(z;\delta)
\end{pmatrix}.
\end{equation}
Here, $T_\sigma$ is an $r_\s\times r_\s$ matrix given by 
\begin{align}
T_\sigma=&C_\s(\gamma)
\diag\Big( \exp\left\{
2\pi\ii\transp{
\tilde{\bf k}(i)
}
A_{\s}^{-1}\delta
\right\}\Big)_{i=1}^{r_\s}
U_\s
\diag \left( \varepsilon_\s(\delta,{\bf k}(j))\right)_{j=1}^{r_\s}.
\end{align}
In particular, if $z$ is nonsingular and $\gamma_l\notin\Z$ for $l=1,\dots,k$, $\displaystyle\bigcup_{\s\in T}\left\{ [\Gamma_{\s,\tilde{\bf k}(j)}]\right\}_{j=1}^{r_\s}$ is a basis of the rapid decay homology group $\Homo_{n,z}^{r.d.}.$
\end{thm}


For later use, we also give a formula for dual period integral. Consider an integral of the form
\begin{equation}\label{DualEulerInt}
f_{\check{\Gamma}}^\vee(z)=\frac{1}{(2\pi\ii)^{n+k}}\int_{\check{\Gamma}} e^{-h_{0,z^{(0)}}(x)} h_{1,z^{(1)}}(x)^{\gamma_1}\cdots h_{k,z^{(k)}}(x)^{\gamma_k}x^{-c}\frac{dx}{x}.
\end{equation}

\noindent
Using plane wave coordinate as before, we get
\begin{align}
f_{\check{\Gamma}}^\vee(z)
 =&\frac{sgn(A,\s)}{\det A_\s}\frac{z_\s^{A_\s^{-1}\delta}}{(2\pi\ii)^{n+k}}\displaystyle\int_{p_*\check{\Gamma}}\prod_{l=1}^k\left(\sum_{i\in\s^{(l)}}\tau_i+\sum_{j\in\bar{\s}^{(l)}}z_\s^{-A_\s^{-1}{\bf a}(j)}z_j\rho^{S_0({\bf a}(j))}(u_{\s^{(0)}},\tau_\s)^{A_\s^{-1}{\bf a}(j)}\right)^{\gamma_l}\times\nonumber\\
 &\exp\left\{ 
-\rho-\sum_{j\in\bs^{(0)}}z_\s^{-A_\s^{-1}{\bf a}(j)}z_j
\rho^{S_0({\bf a}(j))}
(u_{\s^{(0)}},\tau_\s)^{A_\s^{-1}{\bf a}(j)}
\right\}
\rho^{-S_0(\delta)}
(u_{\s^{(0)}},\tau_\s)^{-A_\s^{-1}\delta}
\frac{d\rho d u_{\s^{(0)}} \omega(\tau_\s) }{\rho u_{\s^{(0)}} \tau_\s},\label{DualPlaneWaveIntegral}
\end{align}

\noindent
The cycle $\check{\Gamma}$ in $(u_{\s^{(0)}},\tau_\s)$-direction is the product of Pochhammer cycles $\check{\Gamma}_0=\check{P}_{u_\s^{(0)}}\times\displaystyle\prod_{l=1}^k\check{P}_{\tau_{\s^{(l)}}}$. In $\rho$ direction, we take the dual Hankel contour $\check{C}_0$. $\check{C}_0$ is given by the formula
$\check{C}_0=-[\delta,\infty)+\check{l}_{(0+)}+[\delta,\infty)e^{2\pi\sqrt{-1}},$
where $e^{2\pi\sqrt{-1}}$ stands for the argument of the variable and $\check{l}_{(0+)}$ is a small loop which encircles the origin in the counter-clockwise direction starting from and ending at the point $\delta$ for some small positive $\delta$. Therefore, we take the cycle $\check{\Gamma}_{\s,0}$ so that $p_*\check{\Gamma}_{\s,0}=\check{C}_0\times \check{\Gamma}_0$. Note that the change of coordinate $\tilde{\rho}=e^{-\pi\ii}\rho$ transforms $\check{C}_0$ to $C_0$. We set
\begin{equation}\label{DualSeriesPhiVee}
\varphi^\vee_{\s,{\bf k}}(z;\delta)=z_\sigma^{A_\sigma^{-1}\delta}
\sum_{{\bf k+m}\in\Lambda_{\bf k}}\frac{(-1)^{\bf k_0+m_0}e^{\pi\ii S_0(A_{\bs}({\bf k+m}))}(z_\sigma^{-A_\sigma^{-1}A_{\bar{\sigma}}}z_{\bar{\sigma}})^{\bf k+m}}{\Gamma({\bf 1}_\sigma+A_\sigma^{-1}(\delta-A_{\bar{\sigma}}({\bf k+m}))){\bf (k+m)!}}.
\end{equation}

\noindent
Then, it is easy to see the formula
\begin{align}
f_{\s,0}^\vee(z;\delta)\overset{def}{=}&f^\vee_{\check{\Gamma}_{\s,0}}(z)\nonumber\\
=&e^{-\pi\ii S_0(\delta)}
C_\s(-\gamma)
\sum_{j=1}^{r_{\s}}\varepsilon_\s(-\delta,{\bf k}(j))\varphi_{\s,{\bf k}(j)}^\vee(z;\delta)
\end{align}
holds. As before, to any integer vector $\tilde{\bf k}\in\Z^{\s\times 1}$, we associate a deck transform $\check{\Gamma}_{\s,\tilde{\bf k}}$ of $\check{\Gamma}_{\s,0}$ along the loop $(\xi_{\s^{(0)}},[\tau_\s])\mapsto e^{2\pi\ii{}^t\tilde{\bf k}}(\xi_{\s^{(0)}},[\tau_\s])$. We have the dual statement of Theorem \ref{thm:fundamentalthm3}.

\begin{thm}\label{thm:dualfundamentalthm3}
Take a convergent regular triangulation $T$. Assume that the parameter vector $\delta$ is very generic and that $\gamma_l\notin\Z_{\geq 0}$ for any $l=1,\dots,k$. Then, if one puts
\begin{equation}
f_{\s,\tilde{\bf k}}^\vee(z;\delta)=\frac{1}{(2\pi\ii)^{n+k}}\int_{\check{\Gamma}_{\s,\tilde{\bf k}(j)}} e^{-h_{0,z^{(0)}}(x)}h_{1,z^{(1)}}(x)^{\gamma_1}\cdots h_{k,z^{(k)}}(x)^{\gamma_k}x^{-c}\frac{dx}{x},
\end{equation}
for each $\sigma\in T,$ one has a transformation formula 
\begin{equation}
\begin{pmatrix}
f^\vee_{\sigma,\tilde{\bf k}(1)}(z;\delta)\\
\vdots\\
f^\vee_{\sigma,\tilde{\bf k}(r_\s)}(z;\delta)
\end{pmatrix}
=
T_\sigma^\vee
\begin{pmatrix}
\varphi_{\sigma,{\bf k}(1)}^\vee(z;\delta)\\
\vdots\\
\varphi_{\sigma,{\bf k}(r_\s)}^\vee(z;\delta)
\end{pmatrix}.
\end{equation}
Here, $T_\sigma^\vee$ is an $r_\s\times r_\s$ matrix given by 
\begin{align}
T_\sigma^\vee=&e^{-\pi\ii S_0(\delta)}
C_\s(-\gamma)
\diag\Big( \exp\left\{
-2\pi\ii\transp{
\tilde{\bf k}(i)
}
A_{\s}^{-1}\delta
\right\}\Big)_{i=1}^{r_\s}
U_\s
\diag \Big( \varepsilon_\s\left(-\delta,{\bf k}(j)\right)\Big)_{j=1}^{r_\s}.
\end{align}
In particular, if $z$ is nonsingular, $\gamma_l\notin\Z$ for any $l=1,\dots,k$, $\displaystyle\bigcup_{\s\in T}\left\{ \check{\Gamma}_{\s,\tilde{\bf k}(j)}\right\}_{j=1}^{r_\s}$ is a basis of the rapid decay homology group $\check{\Homo}_{n,z}^{r.d.}\overset{def}{=}\Homo_n^{r.d.}\left( \pi^{-1}(z);\nabla_z\right).$
\end{thm}

\begin{exa}
We consider a $3\times 5$ matrix 
$
A=
\left(
\begin{array}{ccc|cc}
1&1&1&0&0\\
\hline
0&0&0&1&1\\
\hline
0&1&-1&0&1
\end{array}
\right)
$
and a $5\times 2$ matrix 
$
B
=
\begin{pmatrix}
-1&-1\\
1&0\\
0&1\\
1&-1\\
-1&1
\end{pmatrix}
$ so that $L_A=\Z B$ holds. For a parameter vector $\delta={}^t(\gamma_1,\gamma_2,c)$, the GKZ system $M_A(\delta)$ is related to the Horn's $G_1$ function (\cite[Vol.1, \S5.7.1]{ErdelyiEtAl}, \cite{DworkLoeser}). By considering an exact sequence $0\rightarrow \R^{1\times 3}\overset{\times A}{\rightarrow}\R^{1\times 5}\overset{\times B}{\rightarrow}\R^{1\times 2}\rightarrow 0$, we can draw a projected image of the secondary fan ${\rm Fan}(A)$ in $\R^{1\times 2}$ as in Figure \ref{SecondaryFanOfG1}.

The Euler integral representation we consider is $f_\Gamma(z)=\frac{1}{(2\pi\ii)^3}\int_\Gamma (z_1+z_2x+\frac{z_3}{x})^{-\gamma_1}(z_4+z_5x)^{-\gamma_2}x^c\frac{dx}{x}.$ Let us describe the basis of solutions associated to the regular triangulation $T_4$. We first consider the simplex $345\in T_4$. This choice of simplex corresponds to the degeneration $z_1,z_2\rightarrow 0$. This induces a degeneration of the configuration of branch points of the integrand. We write $\zeta_\pm$ for the zeros of the equation $z_1+z_2x+\frac{z_3}{x}=0$ in $x$. The induced degeneration is $\zeta_\pm\rightarrow\infty$. If we put $\zeta=-\frac{z_4}{z_5}$, the cycle $\Gamma_{345,0}$ is just a Pochhammer cycle connecting $\zeta$ and the origin as in Figure \ref{Cycle345G1}. Since $|\Z^{\{345\}\times 1}/\Z {}^tA_{345}|=1$, we are done for this simplex.

\begin{figure}[H]
\begin{center}
\begin{tikzpicture}[scale=0.85]
\draw (0,0) node[left]{O}; 
\draw[thick, ->] (0,0)--(1,0);
\draw[thick, ->] (0,0)--(0,1);
\draw[thick, ->] (0,0)--(-1,1);
\draw[thick, ->] (0,0)--(-1,-1);
\draw[thick, ->] (0,0)--(1,-1);
\draw[-] (0,0)--(4,0);
\draw[-] (0,0)--(0,4);
\draw[-] (0,0)--(-4,4);
\draw[-] (0,0)--(-3,-3);
\draw[-] (0,0)--(3,-3);
\node at (2,2){$T_1=\{ 125,134,145\}$};
\node at (-2,4){$T_2=\{ 124,134,245\}$};
\node at (-3,0){$T_3=\{ 234,245\}$};
\node at (0,-2){$T_4=\{ 235,345\}$};
\node at (3,-1){$T_5=\{ 125,135,345\}$};
\end{tikzpicture}
\caption{The secondary fan of Horn's $G_1$ in $\R^{1\times 2}$}\label{SecondaryFanOfG1}
\end{center}
\end{figure}
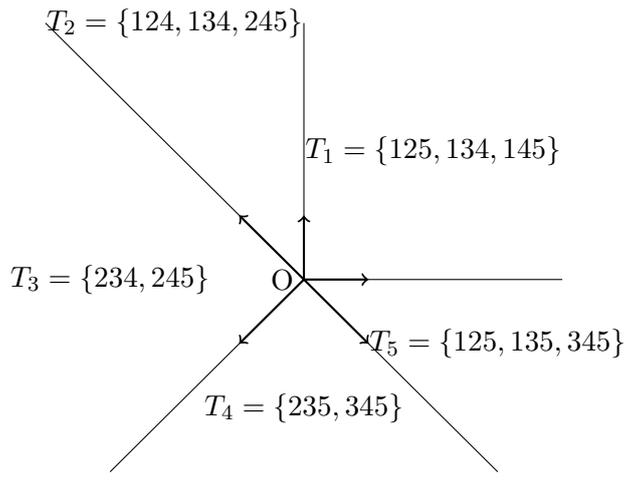

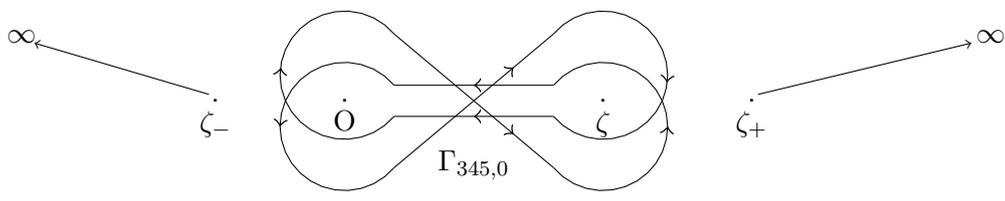
\begin{figure}[H]
\begin{center}
\begin{tikzpicture}[scale=0.85]
\node at (0,0){$\cdot$};
\node at (4,0){$\cdot$};
\draw (0,0) node[below]{O};
\draw (4,0) node[below]{$\zeta$};
\draw[-<-=.5,domain=-140:140] plot ({4+cos(\x)}, {0.4+sin(\x)});
\draw[-<-=.5,domain=40:320] plot ({cos(\x)}, {0.4+sin(\x)});
\draw[->-=.5,domain=-140:140] plot ({4+cos(\x)}, {-0.4+sin(\x)});
\draw[->-=.5,domain=40:320] plot ({cos(\x)}, {-0.4+sin(\x)});
    \coordinate (A1) at ({cos(40)}, {0.4+sin(40)});
    \coordinate (A2) at ({4+cos(-140)}, {-0.4+sin(-140)});
    \coordinate (B1) at ({4+cos(140)}, {-0.4+sin(140)});
    \coordinate (B2) at ({cos(40)}, {-0.4+sin(40)});
    \coordinate (C1) at ({cos(320)}, {-0.4+sin(320)});
    \coordinate (C2) at ({4+cos(140)}, {0.4+sin(140)});
    \coordinate (D1) at ({4+cos(-140)}, {0.4+sin(-140)});
    \coordinate (D2) at ({cos(320)}, {0.4+sin(320)});
\draw[->-=.75] (A1) -- (A2);
\draw[->-=.5] (B1) -- (B2);
\draw[->-=.75] (C1) -- (C2);
\draw[->-=.5] (D1) -- (D2);
\node at (6.3,0) {$\cdot$};
\draw (6.3,0) node[below]{$\zeta_+$};
\node at (-2,0) {$\cdot$};
\draw (-2,0) node[below]{$\zeta_-$};
\draw[->] (6.4,0.1) -- (9.7,0.9);
\draw[->] (-2.1,0.1) -- (-4.8,0.9);
\node at (10,1) {$\infty$};
\node at (-5,1) {$\infty$};
\node at (2,-1) {$\Gamma_{345,0}$};
\end{tikzpicture}
\caption{Degeneration of an arrangement associated to a simplex $345$}
\label{Cycle345G1}
\end{center}
\end{figure}

\begin{figure}[H]
\begin{minipage}{0.5\hsize}
\begin{center}
\begin{tikzpicture}
\node at (0,0){$\cdot$};
\node at (2,0){$\cdot$};
\draw (0,0) node[below]{O};
\draw (2,0) node[below]{$\zeta_+$};
\draw[-<-=.5,domain=-140:140] plot ({2+(0.5)*(cos(\x))}, {(0.5)*(sin(\x))});
\draw[->-=.5] ({2+(0.5)*(cos(-140))},{(0.5)*(sin(-140))}) -- ({(0.5)*(cos(-40))},{(0.5)*(sin(-40))});
\draw[->-=.5,domain=-40:-200] plot ({(0.5)*cos(\x)},{(0.5)*sin(\x)});
\draw[->-=.5] ({(0.5)*cos(-200)},{(0.5)*sin(-200)})--({-2+(0.5)*cos(20)},{(0.5)*sin(20)});
\draw[->-=.5,domain=20:340] plot ({-2+(0.5)*cos(\x)},{(0.5)*sin(\x)});
\draw[->-=.5] ({-2+(0.5)*cos(340)},{(0.5)*sin(340)})--({(0.7)*cos(-165)},{(0.7)*sin(-165)});
\draw[->-=.5,domain=-165:25] plot ({(0.7)*cos(\x)},{(0.7)*sin(\x)}); 
\draw[->-=.5] ({(0.7)*cos(25)},{(0.7)*sin(25)}) -- ({2+(0.5)*cos(140)},{(0.5)*sin(140)});
\node at ({(0.7)*cos(25)},{(0.7)*sin(25)}) {$\bullet$};
\draw ({(0.7)*cos(25)},{(0.7)*sin(25)}) node[above]{$\arg x=0$};
\node at (2,0) {$\cdot$};
\draw (2,0) node[below]{$\zeta_+$};
\node at (-2,0) {$\cdot$};
\draw (-2,0) node[below]{$\zeta_-$};
\node at (0.3,0) {$\cdot$};
\draw (0.3,0) node[right]{$\zeta$};
\draw[->] (0.2,0) -- (0.1,0);
\end{tikzpicture}
\caption{The cycle $\Gamma_{235,0}$}
\label{Cycle2350G1}
\end{center}
\end{minipage}
\begin{minipage}{0.5\hsize}
\begin{center}
\begin{tikzpicture}
\node at (0,0){$\cdot$};
\node at (2,0){$\cdot$};
\draw (0,0) node[below]{O};
\draw (2,0) node[below]{$\zeta_+$};
\draw[-<-=.5,domain=-140:140] plot ({-2-(0.5)*(cos(\x))},{(-0.5)*(sin(\x))});
\draw[->-=.5] ({-2-(0.5)*(cos(-140))},{(-0.5)*(sin(-140))}) -- ({(-0.5)*(cos(-40))},{(-0.5)*(sin(-40))});
\draw[->-=.5,domain=-40:-200] plot ({(-0.5)*cos(\x)},{(-0.5)*sin(\x)});
\draw[->-=.5] ({(-0.5)*cos(-200)},{(-0.5)*sin(-200)})--({2+(-0.5)*cos(20)},{(-0.5)*sin(20)});
\draw[->-=.5,domain=20:340] plot ({2+(-0.5)*cos(\x)},{(-0.5)*sin(\x)});
\draw[->-=.5] ({2-(0.5)*cos(340)},{(-0.5)*sin(340)})--({(-0.7)*cos(-165)},{(-0.7)*sin(-165)});
\draw[->-=.5,domain=-165:25] plot ({(-0.7)*cos(\x)},{(-0.7)*sin(\x)}); 
\draw[->-=.5] ({(-0.7)*cos(25)},{(-0.7)*sin(25)}) -- ({-2-(0.5)*cos(140)},{(-0.5)*sin(140)});
\node at ({(-0.7)*cos(25)},{(-0.7)*sin(25)}) {$\bullet$};
\draw ({(-0.7)*cos(25)},{(-0.7)*sin(25)}) node[below]{$\arg x=\pi$};
\node at (-2,0) {$\cdot$};
\draw (-2,0) node[below]{$\zeta_-$};
\node at (0.3,0) {$\cdot$};
\draw (0.3,0) node[below]{$\zeta$};
\draw[->] (0.2,0) -- (0.1,0);
\end{tikzpicture}
\caption{The cycle $\Gamma_{235,1}.$}
\label{Cycle2351G1}
\end{center}
\end{minipage}
\end{figure}

On the other hand, the simplex $235$ induces a different degeneration. This choice of simplex corresponds to the limit $z_1,z_4\rightarrow 0$. Therefore, the corresponding degeneration of branch points of the integrand is $\zeta\rightarrow 0$ and $\zeta_\pm\rightarrow \pm\sqrt{-\frac{z_3}{z_2}}$. Since $\Z^{\{235\}\times 1}/\Z{}^t A_{235}\simeq\Z/2\Z$, we have two independent cycles as in Figure \ref{Cycle2350G1} and \ref{Cycle2351G1}.

\end{exa}

\begin{exa}
We consider a $2\times 4$ matrix 
$
A=
\left(
\begin{array}{cc|cc}
0&0&1&1\\
\hline
1&-1&0&1
\end{array}
\right)
$
and a $4\times 2$ matrix 
$
B
=
\begin{pmatrix}
1&1\\
1&0\\
0&1\\
0&-1
\end{pmatrix}
$ so that $L_A=\Z B$ holds. For a parameter vector $\delta=
\begin{pmatrix}
\gamma\\
c
\end{pmatrix}
$, the GKZ system $M_A(\delta)$ is related to Horn's $\Gamma_2$ function (\cite{DworkLoeser}). The Euler-Laplace integral representation is of the form $f_\Gamma(z)=\frac{1}{(2\pi\ii)^2}\int_\Gamma e^{z_1x+z_2x^{-1}}(z_3+z_4x)^{-\gamma}x^c\frac{dx}{x}$. $T_2$ is the unique convergent regular triangulation. All the simplexes have normalized  volume 1. Let us consider $\s=14$. We set $\zeta=-\frac{z_3}{z_4}$. Then, the simplex $\s=14$ corresponds to the limit $z_2,z_3\rightarrow 0$ which induces a degeneration of the integrand $ e^{z_1x+z_2x^{-1}}(z_3+z_4x)^{-\gamma}x^c\rightarrow e^{z_1x}x^{c-\gamma}.$ Therefore, the resulting integration contour $\Gamma_{14,0}$ is as in the upper right one in Figure \ref{TheFirstFigure}. We can construct the contour $\Gamma_{23,0}$ in the same way as in the lower right picture of Figure \ref{TheFirstFigure}. Finally, the cycle $\Gamma_{34,0}$ is nothing but the Pochhammer cycle connecting $0$ and $\zeta$, hence bounded.

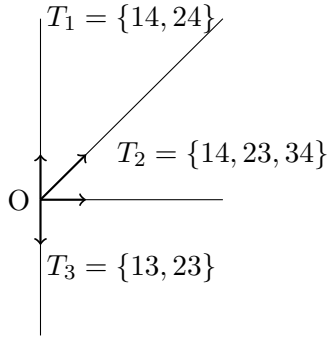
\begin{figure}[h]
\begin{center}
\begin{tikzpicture}[scale=0.6]
\draw (0,0) node[left]{O}; 
\draw[thick, ->] (0,0)--(1,0);
\draw[thick, ->] (0,0)--(0,1);
\draw[thick, ->] (0,0)--(1,1);
\draw[thick, ->] (0,0)--(0,-1);
\draw[-] (0,0)--(4,0);
\draw[-] (0,0)--(0,4);
\draw[-] (0,0)--(4,4);
\draw[-] (0,0)--(0,-3);
\node at (2,4){$T_1=\{ 14,24\}$};
\node at (4,1){$T_2=\{ 14,23,34\}$};
\node at (2,-1.5){$T_3=\{ 13,23\}$};
\end{tikzpicture}
\caption{The secondary fan of Horn's $\Gamma_2$ in $\R^{1\times 2}$}
\end{center}
\end{figure}

\end{exa}

\begin{exa}

We consider a $3\times 5$ matrix 
$
A=
\left(
\begin{array}{cc|ccc}
0&0&1&1&1\\
\hline
1&0&0&1&1\\
0&1&0&0&1
\end{array}
\right)
$
and a $5\times 2$ matrix 
$
B
=
\begin{pmatrix}
1&0\\
0&1\\
1&0\\
-1&1\\
0&-1
\end{pmatrix}
$ so that $L_A=\Z B$ holds. For a parameter vector $\delta=
{}^t(\gamma,c_1,c_2)$, the GKZ system $M_A(\delta)$ is related to Horn's ${\bf H}_4$ function (\cite[Vol.1, \S5.7.1]{ErdelyiEtAl}, \cite{DworkLoeser}). The Euler-Laplace integral representation is of the form $f_\Gamma(z)=\frac{1}{(2\pi\ii)^2}\int_\Gamma e^{z_1x+z_2y}(z_3+z_4x+z_5xy)^{-\gamma}x^{c_1}y^{c_2}\frac{dx\wedge dy}{xy}$. $T_1$ is the unique  convergent regular triangulation. All the simplexes have volume 1. Let us consider $\s=125$. The simplex $\s=125$ corresponds to the limit $z_3,z_4\rightarrow 0$ which induces a degeneration of the integrand $ e^{z_1x+z_2y}(z_3+z_4x+z_5xy)^{-\gamma}x^{c_1}y^{c_2}\rightarrow z_5^{-\gamma}e^{z_1x+z_2y}x^{c_1-\gamma}y^{c_2-\gamma}.$ Therefore, the resulting integration contour $\Gamma_{125,0}$ is as in Figure \ref{Cycle1250H4}. The construction is as follows: we consider a change of coordinate $(z_1x,z_2y)=(\rho u,\rho v)$ with $u+v=1$. Then the cycle $\Gamma_{125,0}$ is the product of a Hankel contour in $\rho$ direction and a Pochhammer cycle in $(u,v)$ direction. Note that the divisor $\{ z_3+z_4x+z_5xy=0\}\subset(\C^\times)^2$ is encircled by $\Gamma_{125,0}$. The constructions of $\Gamma_{145,0}$ and $\Gamma_{235,0}$ are similar. On the other hand, if we consider a simplex $\s=345$, the corresponding degeneration of the integrand is $ e^{z_1x+z_2y}(z_3+z_4x+z_5xy)^{-\gamma}x^{c_1}y^{c_2}\rightarrow (z_3+z_4x+z_5xy)^{-\gamma}x^{c_1}y^{c_2}.$ The change of coordinate $p(x,y)=(\xi,\eta)$ of the torus $(\C^\times)^2$ that we discussed in general fashion in this section, is explicitly given by $\xi=-\frac{z_4}{z_3}x, \eta=-\frac{z_5}{z_3}xy$. This change of coordinate can be seen as a part of blow-up coordinate of $Bl_{(0,0)}(\C^2)$.


\begin{figure}[h]
\begin{minipage}{0.5\hsize}
\begin{center}
\begin{tikzpicture}[scale=0.85]
\draw (0,0) node[left]{O}; 
\draw[thick, ->] (0,0)--(1,0);
\draw[thick, ->] (0,0)--(0,1);
\draw[thick, ->] (0,0)--(-1,1);
\draw[thick, ->] (0,0)--(0,-1);
\draw[-] (0,0)--(5,0);
\draw[-] (0,0)--(0,4);
\draw[-] (0,0)--(-4,4);
\draw[-] (0,0)--(0,-3);
\node at (2.5,2){$T_1=\{ 125,145,235,345\}$};
\node at (-2,4){$T_2=\{ 125,135,235\}$};
\node at (-2,-1){$T_3=\{ 123\}$};
\node at (2,-2){$T_4=\{124,234\}$};
\end{tikzpicture}
\caption{Projected image of the secondary fan of Horn's ${\bf H}_4$ in $\R^{1\times 2}$}
\end{center}
\end{minipage}
\begin{minipage}{0.5\hsize}
\begin{center}
\begin{tikzpicture}[scale=0.85,domain=0:3, samples=100] 
\draw (0,0) node[anchor= north west]{O}; 
\draw[->] (-4,0)--(4,0);
\draw[->] (0,-4)--(0,4);
\draw[-] (-1.8,-0.2)--(-0.2,-1.8);
\node at (4.4,0){$z_1x$};
\node at (0,4.2){$z_2y$};
\draw[-] (-0.2,-1.8) to [out=20,in=100] (0.2,-2.2);
\draw[dashed] (-0.2,-1.8) to [out=-90,in=-120] (0.2,-2.2);
\draw[-] (-1.8,-0.2) to [out=90,in=80] (-2.2,0.2);
\draw[dashed] (-1.8,-0.2) to [out=180,in=270] (-2.2,0.2);
\draw[->-=.5,dashed] (-3,-3.2) to (0.4,0.2);
\draw[-<-=.5] (-3,-2.8) to (0.2,0.4);
\draw[->-=.5,dashed] (0.4,0.2) to [out=45, in=45] (0.2,0.4);
\draw[-,domain=0.02:3] plot ( \x,0.06/\x );
\draw[-,domain=-3:-0.02] plot ( \x,0.06/\x );
\end{tikzpicture}
\caption{cycle $\Gamma_{125,0}$}
\label{Cycle1250H4}
\end{center}
\end{minipage}
\end{figure}

\end{exa}

\section{A formula for intersection numbers}\label{IntersectionNumbers}
In the following, we fix an $(n+k)$-simplex $\s$ such that the corresponding series $\varphi_{\s,{\bf k}}(z;\delta)$ is convergent. We assume that the parameter $\delta$ is very generic with respect to $\s$ and $\gamma_l\notin\Z$ for $l=1,\dots, k$. In the previous section, for any given convergent regular triangulation $T$, we constructed a basis of $\Homo_{n,z}^{r.d.}$ at each $z\in U_T$. In this section, we show that they behave well with respect to homology intersection pairing. Under the notation of \S\ref{subsec:3.2}, we set $\check{\Homo}_{n,z}^{mod}=\Homo_n^{mod}\left( \pi^{-1}(z);\nabla_z\right)$. Recall that there is a canonical morphism ${\rm can}:\check{\Homo}_{n,z}^{r.d.}\rightarrow\check{\Homo}_{n,z}^{mod}$ which appeared in (\ref{RDCommutativity}). We are interested in the intersection number $\langle \Gamma_{\s_1,\tilde{\bf k}_1},{\rm can} (\check{\Gamma}_{\s_2,\tilde{\bf k}_2})\rangle_h$.

Firstly, we observe that the open set $U_T$ is invariant by $z_j\mapsto e^{\pi\ii\theta_j}z_j$ for any $j$ and $\theta_j\in\R$. Let us consider a path $\gamma_j(\theta)$ $(0\leq\theta\leq 1)$ given by $\gamma_j(\theta)=(z_1,\dots,e^{2\pi\ii\theta}z_j,\dots,z_N)$ where $z=(z_1,\dots,z_N)$ is any point of $U_T$. From the explicit expression of $\Gamma$-series, we see that the analytic continuation $\gamma_{j*}\varphi_{\s,{\bf k}}(z;\delta)$ of $\varphi_{\s,{\bf k}}(z;\delta)$ along $\gamma_j$ satisfies $\gamma_{j*}\varphi_{\s,{\bf k}}(z;\delta)=e^{-2\pi\ii p_{\s j}(\delta+A_{\bs}{\bf k})}\varphi_{\s,{\bf k}}(z;\delta)$ if $j\in\s$ and $\gamma_{j*}\varphi_{\s,{\bf k}}(z;\delta)=\varphi_{\s,{\bf k}}(z;\delta)$ if $j\in\bs$. Since the morphism (\ref{Integration}) preserves monodromy, we see from \cref{thm:fundamentalthm3} that $\Gamma_{\s,\tilde{\bf k}}$ is a sum of eigenvectors with eigenvalues $e^{-2\pi\ii p_{\s j}(\delta+A_{\bs}{\bf k})}$ if $j\in\s$ or is itself an eigenvector with eigenvalue $1$. Therefore, we have the following proposition in view of the fact that homology intersection pairing (\ref{eqn:4.7}) is monodromy invariant.

\begin{prop}
If $\s_1\neq \s_2$, then $\langle \Gamma_{\s_1,\tilde{\bf k}_1},{\rm can}(\check{\Gamma}_{\s_2,\tilde{\bf k}_2})\rangle_h=0$.
\end{prop}

\begin{rem}
When there is no risk of confusion, the intersection number $\langle \Gamma_{\s_1,\tilde{\bf k}_1},{\rm can}(\check{\Gamma}_{\s_2,\tilde{\bf k}_2})\rangle_h$ is simply denoted by $\langle \Gamma_{\s_1,\tilde{\bf k}_1},\check{\Gamma}_{\s_2,\tilde{\bf k}_2}\rangle_h$.
\end{rem}

Thus, it remains to compute $\langle \Gamma_{\s,\tilde{\bf k}_1},\check{\Gamma}_{\s,\tilde{\bf k}_2}\rangle_h$. We compute this quantity when the regular triangulation $T$ is unimodular, i.e., when $\det A_\s=\pm 1$ for any simplex $\s\in T$. The  computation is based on the basic formula of the intersection numbers of Pochhammer cycles and that of Hankel contours. For complex numbers $\alpha_1,\dots,\alpha_{n+1}$, let us put $X=\C^n_x\setminus\{ x_1\cdots x_n(1-x_1-\dots-x_n)=0\}$, $\mathcal{L}=\C x_1^{\alpha_1}\cdots x_n^{\alpha_n}(1-x_1-\dots-x_n)^{\alpha_{n+1}}$, $x_i=e^{-\pi\ii}\frac{\tau_i}{\tau_0}$ $(i=1,\dots,n)$, and $\alpha_0=-\alpha_1-\dots-\alpha_{n+1}$. Under this notation, we have $X=\mathbb{P}^n_\tau\setminus\{ \tau_0\cdots\tau_n(\tau_0+\dots+\tau_n)=0\}$. The local system $\mathcal{L}$ is symbolically denoted by $\mathcal{L}=\underline{\C}\tau_0^{\alpha_0}\cdots\tau_n^{\alpha_n}(\tau_0+\dots+\tau_n)^{\alpha_{n+1}}$.

\begin{prop}\label{prop:PochhammerIntersection}
If $P_\tau\in\Homo_n(X,\mathcal{L})$ and $\check{P}_\tau\in\Homo_n(X,\mathcal{L}^\vee)$ denote the $n$-dimensional Pochhammer cycles with coefficients in $\mathcal{L}$ and $\mathcal{L}^\vee$ respectively, one has a formula 
\begin{equation}
\langle P_\tau,\check{P}_\tau\rangle_h=\prod_{i=0}^{n+1}(1-e^{-2\pi\ii\alpha_i})=(2\ii)^{n+2}\prod_{i=0}^{n+1}\sin\pi\alpha_i.
\end{equation}
\end{prop}
\noindent
The proof of this proposition will be given in the appendix. In $\rho$ direction, we also have a formula of the intersection number of the Hankel contour $C_0$ and the dual Hankel contour $C_0^\vee$. We set $\nabla_\alpha=d\rho+\alpha\frac{d\rho}{\rho}\wedge+d\rho\wedge$. The following proposition is an immediate consequence of \cite[THEOREM4.3]{MajimaMatsumotoTakayama}.

\begin{prop}\label{prop:HankelIntersection}
If $C_0\in\Homo_1^{r.d.}\left((\Gm)_\rho,\nabla_\alpha^\vee\right)$ and $C_0^\vee\in\Homo_1^{r.d.}\left((\Gm)_\rho,\nabla_\alpha\right)$ denote the Hankel contour and the dual Hankel contour respectively, one has a formula
\begin{equation}
\langle C_0,C_0^\vee\rangle_h=1-e^{-2\pi\ii\alpha}.
\end{equation}
\end{prop}

Now we apply \cref{prop:PochhammerIntersection} and \cref{prop:HankelIntersection} to integration cycles constructed in the previous section. In the following computations, we may assume that $z\approx z_\infty^\s$ since $\langle\bullet,\bullet\rangle_h$ is invariant under parallel transport. We factorize the integrand as

\begin{align}
&\prod_{l:l\geq  1, |\s^{(l)}|>1}\left\{\left(\sum_{i\in\s^{(l)}}\tau_i+\sum_{j\in\bar{\s}^{(l)}}z_\s^{-A_\s^{-1}{\bf a}(j)}z_j\rho^{S_0({\bf a}(j))}(u_{\s^{(0)}},\tau_\s)^{A_\s^{-1}{\bf a}(j)}\right)^{-\gamma_l}\prod_{i\in\s^{(l)}}\tau_i^{p_{\s i}(\delta)}\right\}\times\nonumber\\
&\prod_{l:l\geq  1, |\s^{(l)}|=1}\left(1+\tau_{i_l}^{-1}\sum_{j\in\bar{\s}^{(l)}}z_\s^{-A_\s^{-1}{\bf a}(j)}z_j\rho^{S_0({\bf a}(j))}(u_{\s^{(0)}},\tau_\s)^{A_\s^{-1}{\bf a}(j)}\right)^{-\gamma_l}\prod_{i\in\s^{(0)}}u_i^{p_{\s i}(\delta)}\times\nonumber\\
&\exp\left\{ 
\rho+\sum_{j\in\bs^{(0)}}z_\s^{-A_\s^{-1}{\bf a}(j)}z_j
\rho^{S_0({\bf a}(j))}
(u_{\s^{(0)}},\tau_\s)^{A_\s^{-1}{\bf a}(j)}
\right\}
\rho^{S_0(\delta)}.
\end{align}
Thus, on a neighborhood of the cycle $\Gamma_{\s,0}$, the factor 
\begin{equation}\label{Factor1}
\prod_{l:l\geq  1, |\s^{(l)}|=1}\left(1+\tau_{i_l}^{-1}\sum_{j\in\bar{\s}^{(l)}}z_\s^{-A_\s^{-1}{\bf a}(j)}z_j\rho^{S_0({\bf a}(j))}(u_{\s^{(0)}},\tau_\s)^{A_\s^{-1}{\bf a}(j)}\right)^{-\gamma_l}
\end{equation}
is holomorphic since $z_\s^{-A_\s^{-1}{\bf a}(j)}z_j$ are small complex numbers and $S_0({\bf a}(j))\leq 0$. By the formula $S_l(\delta)=\gamma_l$, the assumption of the \cref{prop:PochhammerIntersection} is satisfied.
On the other hand, if we take an open neighbourhood $\widetilde{V}\subset\widetilde{X_z}$ so that its slice in $\rho$-space is a small neighbourhood of both the Hankel contour and the dual Hankel contour, (cf. Figure \ref{TheSlice}) and its slice in $(u_{\s^{(0)}},\tau_\s)$-space is a small neighbourhood of $\Gamma_0=P_{u_\s^{(0)}}\times\displaystyle\prod_{l=1}^kP_{\tau_{\s^{(l)}}}$, we see that the factor 
\begin{equation}\label{Factor2}
\exp\left\{ 
\sum_{j\in\bs^{(0)},\ S_0({\bf a}(j))\leq 0}z_\s^{-A_\s^{-1}{\bf a}(j)}z_j
\rho^{S_0({\bf a}(j))}
(u_{\s^{(0)}},\tau_\s)^{A_\s^{-1}{\bf a}(j)}
\right\}
\end{equation}
is bounded on $\widetilde{V}$. The remaining exponential factor is 
\begin{equation}
\exp\left\{ \left(1+
\sum_{j\in\bs^{(0)},\ S_0({\bf a}(j))=1}z_\s^{-A_\s^{-1}{\bf a}(j)}z_j
(u_{\s^{(0)}},\tau_\s)^{A_\s^{-1}{\bf a}(j)}\right)\rho
\right\}.
\end{equation}
We introduce a new coordinate $\tilde{\rho}$ by setting $\tilde{\rho}=\left(1+
\displaystyle\sum_{j\in\bs^{(0)},\ S_0({\bf a}(j))=1}z_\s^{-A_\s^{-1}{\bf a}(j)}z_j
(u_{\s^{(0)}},\tau_\s)^{A_\s^{-1}{\bf a}(j)}\right)\rho
$. Since $\displaystyle\sum_{j\in\bs^{(0)},\ S_0({\bf a}(j))=1}z_\s^{-A_\s^{-1}{\bf a}(j)}z_j
(u_{\s^{(0)}},\tau_\s)^{A_\s^{-1}{\bf a}(j)}$ remains very small when $(u_{\s^{(0)}},\tau_{\s})$ runs over our contour, this change of coordinate still gives the Hankel contour in $\tilde{\rho}$ coordinate. 

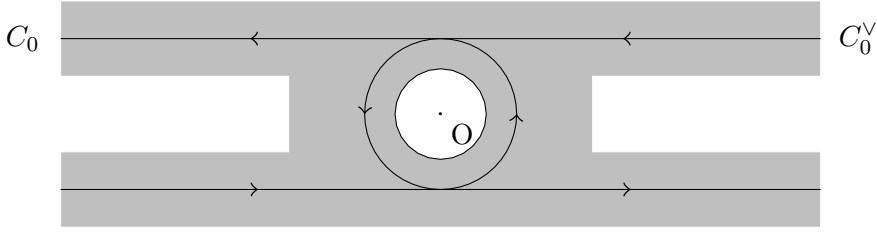
\begin{figure}[H]
\begin{center}
\begin{tikzpicture}
\filldraw[fill=lightgray,draw=black!0] (-5,1.5) -- (5,1.5) -- (5,-1.5) -- (-5,-1.5) -- cycle;
\filldraw[domain=0:360, fill=white] plot ({(0.6)*cos(\x)}, {(0.6)*sin(\x)});
\node at (0,0){$\cdot$};
\draw (0,0) node[below right]{O};
\draw[->-=.5,domain=-90:90] plot ({cos(\x)}, {sin(\x)});
\draw[-<-=.5] ({cos(-90)},{sin(-90)}) -- (-5, {sin(-90)});
\draw[->-=.5] ({cos(90)},{sin(90)}) -- (-5, {sin(90)});

\draw[->-=.5,domain=90:270] plot ({cos(\x)}, {sin(\x)});
\draw[-<-=.5] ({cos(90)},{sin(90)}) -- (5, {sin(90)});
\draw[->-=.5] ({cos(270)},{sin(270)}) -- (5, {sin(270)});

\filldraw[fill=white, draw=black!0] (-2,0.5) -- (-5,0.5) -- (-5, -0.5) -- (-2,-0.5) -- cycle;
\filldraw[fill=white, draw=black!0] (2,0.5) -- (5,0.5) -- (5, -0.5) -- (2,-0.5) -- cycle;
\node at (-5.5,1){$C_0$};
\node at (5.5,1){$C_0^\vee$};
\end{tikzpicture}
\caption{The slice of $\widetilde{V}$ in $\rho$-space (the gray zone)}
\label{TheSlice}
\end{center}
\end{figure}

Below, we apply \cref{prop:TheLocalizationFormula} and \cref{prop:TheProductFormula} to $\langle \Gamma_{\s,{\bf 0}},\check{\Gamma}_{\s,{\bf 0}}\rangle_h$. We put
$
\Psi_{0}=e^{\tilde{\rho}}\tilde{\rho}^{S_0(\delta)}
$, 
$
\Psi_{1}=\prod_{i\in\s^{(0)}}u_i^{p_{\s i}(\delta)}
$, 
$
\Psi_{2}=\prod_{l:l\geq  1, |\s^{(l)}|>1}\left\{\left(\sum_{i\in\s^{(l)}}\tau_i\right)^{-\gamma_l}\prod_{i\in\s^{(l)}}\tau_i^{p_{\s i}(\delta)}\right\}
$.
Note that these functions $\Psi_0,\Psi_1,\Psi_2$ are all multivalued functions on $\C^*_{\tilde{\rho}}$, $W_1=\left\{ u_{\s^{(0)}}\in(\C^*)^{|\s^{(0)}|}\mid \sum_{i\in\s^{(0)}}u_i=1\right\}$, and on \newline
$W_2=\prod_{l=1}^k\left(\PP^{|\s^{(l)}|-1}_{\tau_{\s^{(l)}}}\setminus\bigcup_{i\in\s^{(l)}}\{ \tau_i=0\}\cup\left\{ \sum_{i\in\s^{(l)}}\tau_i=0\right\}\right)$ respectively. Since $\widetilde{V}$ is a neighbourhood of both $\Gamma_{\s,{\bf 0}}$ and $\check{\Gamma}_{\s,{\bf 0}}$, we see that there exist cycles $\gamma\in\Homo^{\rm r.d.}_n\left( V,\nabla_z^\vee\right)$ and $\gamma^\vee\in\Homo^{\rm r.d.}_n\left( V,\nabla_z\right)$ such that $\Gamma_{\s,{\bf 0}}=\iota_{\widetilde{X}\widetilde{V}!}\gamma$ and $\check{\Gamma}_{\s,{\bf 0}}=\iota_{\widetilde{X}\widetilde{V}!}\gamma^\vee$. By \cref{prop:TheLocalizationFormula}, we obtain $\langle \Gamma_{\s,{\bf 0}},\check{\Gamma}_{\s,{\bf 0}}\rangle_h=\langle \gamma,\gamma^\vee\rangle_h$. Since (\ref{Factor1}) and (\ref{Factor2}) are both holomorphic and bounded on $\widetilde{V}$, we see that the connection $\nabla_z$ is equivalent to $\nabla_{\rm red}=\nabla_0+\nabla_1+\nabla_2$ on $\tilde{V}$ where $\nabla_0=\Psi^{-1}_0\circ d_{\tilde{\rho}}\circ\Psi_0$, $\nabla_1=\Psi^{-1}_1\circ d_{u_{\s^{(0)}}}\circ\Psi_1$, and $\nabla_2=\Psi^{-1}_2\circ d_{\tau_\s}\circ\Psi_2$. Now, we consider a Cartesian diagram 
\begin{equation}
\xymatrix{
 V \ar@{^{(}->}[r] \ar@{^{(}->}[d]&(\C^*)_{\tilde{\rho}}\times W_1\times W_2\ar@{^{(}->}[d]\\
 \widetilde{V^\prime} \ar@{^{(}->}[r]^{\iota_{\widetilde{Y}\widetilde{V^\prime}!}}                             &\widetilde{Y},
}
\end{equation}
where $\widetilde{Y}$ is a real oriented blow-up of a good compactification of $(\C^*)_{\tilde{\rho}}\times W_1\times W_2$ with respect to the connection $(\mathcal{O}_{(\C^*)_{\tilde{\rho}}\times W_1\times W_2},\nabla_{\rm red})$ and $\widetilde{V^\prime}$ is an open neighbourhood of the cycle $\gamma$ and $\gamma^\vee$ in $\widetilde{Y}$. In our setting, $Y$ is nothing but a product $\PP_{\tilde{\rho}}\times Y_{12}$ where $Y_{12}$ is a product of projective spaces in $u_{\s^{(0)}}$ and $\tau_{\s^{(l)}}$ coordinates. Note that our cycles $\Gamma_{\s,{\bf 0}}$ and $\check{\Gamma}_{\s,{\bf 0}}$ (hence $\gamma$ and $\gamma^\vee$) are defined by taking closures of semi-algebraic cycles (see the discussion right before \cref{thm:fundamentalthm3}) and therefore, do not depend on the choice of the compactification. Applying \cref{prop:TheLocalizationFormula} to the morphism $\iota_{\widetilde{Y}\widetilde{V^\prime}!}:\Homo_n^{\rm r.d.}\left( V,\nabla_{\rm red}\right)\rightarrow\Homo_n^{\rm r.d.}\left( (\C^*)_{\tilde{\rho}}\times W_1\times W_2,\nabla_{\rm red}\right)$ once again, we obtain $\langle\gamma,\gamma^\vee\rangle_h=\langle \iota_{\widetilde{Y}\widetilde{V^\prime}!}\gamma,\iota_{\widetilde{Y}\widetilde{V^\prime}!}\gamma^\vee\rangle_h$. By our construction of cycles $\Gamma_{\s,{\bf 0}}$ and $\check{\Gamma}_{\s,{\bf 0}}$ in \S\ref{SectionEuler}, we see that $\iota_{\widetilde{Y}\widetilde{V^\prime}!}\gamma$ and $\iota_{\widetilde{Y}\widetilde{V^\prime}!}\gamma^\vee$ are cross products of the forms $\iota_{\widetilde{Y}\widetilde{V^\prime}!}\gamma=C_0\times P_{u_{\s^{(0)}}}\times P_{\tau_\s}$ and $\iota_{\widetilde{Y}\widetilde{V^\prime}!}\gamma^\vee=C_0^\vee\times P_{u_{\s^{(0)}}}^\vee\times P_{\tau_\s}^\vee$. Applying \cref{prop:TheProductFormula}, we obtain $\langle \iota_{\widetilde{Y}\widetilde{V^\prime}!}\gamma,\iota_{\widetilde{Y}\widetilde{V^\prime}!}\gamma^\vee\rangle_h=\langle C_0,C_0^\vee\rangle_h \langle P_{u_{\s^{(0)}}},P_{u_{\s^{(0)}}}^\vee\rangle_h\langle P_{\tau_\s},P_{\tau_\s}^\vee\rangle_h$.

If $\s^{(0)}\neq\varnothing$, \cref{prop:HankelIntersection} implies that 
\begin{equation}
\langle C_0,C_0^\vee\rangle_h=\left( 1-e^{-2\pi\ii S_0(\delta)}\right).
\end{equation} 
If $|\s^{(0)}|\geq 2$, \cref{prop:PochhammerIntersection} implies that 
\begin{equation}
\langle P_{u_{\s^{(0)}}},P_{u_{\s^{(0)}}}^\vee\rangle_h=\left( 1-e^{2\pi\ii S_0(\delta)}\right)\prod_{i\in\s^{(0)}}\left( 1-e^{-2\pi\ii p_{\s i}(\delta)}\right).
\end{equation}
Finally, \cref{prop:PochhammerIntersection} also implies that
\begin{equation}
\langle P_{\tau_\s},P_{\tau_\s}^\vee\rangle_h=\prod_{l:|\s^{(l)}|>1}(1-e^{2\pi\ii\gamma_l})\prod_{i\in\s^{(l)}}\left(1-e^{-2\pi\ii p_{\s i}(\delta)}\right).
\end{equation}
Summing up all the arguments above, we obtain a

\begin{thm}\label{thm:IntersectionNumber}
We decompose $\s$ as $\s=\s^{(0)}\cup\dots\cup\s^{(k)}$ and set $\gamma_0=S_0(\delta)$. If $\det A_\s=\pm 1$, then,

\begin{equation}
\langle \Gamma_{\s,0},\check{\Gamma}_{\s,0}\rangle_h =
\begin{cases}
\displaystyle\prod_{l:|\s^{(l)}|>1}\left\{(1-e^{2\pi\ii\gamma_l})\prod_{i\in\s^{(l)}}\left(1-e^{-2\pi\ii p_{\s i}(\delta)}\right)\right\}&(\s^{(0)}=\varnothing)\\
\displaystyle\left( 1-e^{-2\pi\ii\gamma_0}\right)\prod_{l:|\s^{(l)}|>1}\left\{(1-e^{2\pi\ii\gamma_l})\prod_{i\in\s^{(l)}}\left(1-e^{-2\pi\ii p_{\s i}(\delta)}\right)\right\}&(\s^{(0)}\neq\varnothing).
\end{cases}
\end{equation}

\end{thm}

\section{Twisted period relations for $\Gamma$-series}\label{QuadraticRelationsForGammaSeries}
In this section, we derive a quadratic relation for $\Gamma$-series associated to a unimodular regular triangulation. For any complex numbers $\alpha,\beta$ such that $\alpha+\beta\notin\Z_{\leq 0}$, we put $(\alpha)_\beta=\frac{\Gamma(\alpha+\beta)}{\Gamma(\alpha)}$. In general, for any vectors ${\bf \alpha}=(\alpha_1,\dots,\alpha_s),{\bf \beta}=(\beta_1,\dots,\beta_s)\in\C^s$, we put $({\bf \alpha})_{\bf \beta}=\prod_{i=1}^s(\alpha_i)_{\beta_i}$. Combining the results of \S 4 and \S 5, we obtain the main result of this paper.

\begin{thm}\label{thm:QuadraticRelation}
Suppose that four vectors ${\bf a},{\bf a}^\prime\in\Z^{n\times 1},{\bf b},{\bf b}^\prime\in\Z^{k\times 1}$ and a convergent unimodular regular triangulation T are given. If the parameter $\delta$ is generic so that $\gamma_l\notin\Z$ for any $l=1,\dots,k$ and $\delta$, $
\begin{pmatrix}
\gamma-{\bf b}\\
c+{\bf a}
\end{pmatrix}
$  and 
$
\begin{pmatrix}
\gamma+{\bf b}^\prime\\
c-{\bf a}^\prime
\end{pmatrix}
$ are very generic with respect to $T$, then, for any $z\in U_T,$ one has an identity
\begin{align}
&(-1)^{|{\bf b}|+|{\bf b}^\prime|}\gamma_1\cdots\gamma_k(\gamma-{\bf b})_{\bf b}(-\gamma-{\bf b}^\prime)_{{\bf b}^\prime}
\sum_{\s\in T}\frac{\pi^{n+k}}{\sin\pi A_\s^{-1}\delta}\varphi_{\s,0}\left(z;
\begin{pmatrix}
\gamma-{\bf b}\\
c+{\bf a}
\end{pmatrix}
\right)\varphi_{\s,0}^\vee\left(z;
\begin{pmatrix}
\gamma+{\bf b}^\prime\\
c-{\bf a}^\prime
\end{pmatrix}
\right)\nonumber\\
=&\frac{\langle x^{\bf a}h^{\bf b}\frac{dx}{x},x^{{\bf a}^\prime}h^{{\bf b}^\prime}\frac{dx}{x}\rangle_{ch}}{(2\pi\ii)^n}.
\end{align}
\end{thm}

\begin{proof}
We put $\varphi=x^{{\bf a}^\prime}h^{{\bf b}^\prime}\frac{dx}{x}\in\Homo^n_{dR}\left( \pi^{-1}(z),\nabla_z\right)$, $\psi=x^{\bf a}h^{\bf b}\frac{dx}{x}\in\Homo^n_{dR}\left( \pi^{-1}(z),\nabla_z^\vee\right)$. First of all, let us confirm that $\langle\varphi,\psi\rangle_{ch}$ is well-defined. Observe that the  canonical morphism 
\begin{equation}\label{HomologyAndLFHomology}
{\rm can}:\Homo_n^{r.d.}\left( \pi^{-1}(z)^{an},\nabla_z\right)\rightarrow\Homo_n^{mod}\left( \pi^{-1}(z)^{an},\nabla_z\right)
\end{equation}
is an isomorphism. Indeed, by Poincare duality, \cref{thm:mainDresult}, and the fact that $z\notin{\rm Sing} M_A(\delta)$, both sides of (\ref{HomologyAndLFHomology}) have the same dimension. Since the canonical morphism (\ref{HomologyAndLFHomology}) is compatible with intersection pairing $\langle\bullet,\bullet\rangle_h$ and the intersection matrix $\left( \langle\Gamma_{\s,0},\check{\Gamma}_{\s,0}\rangle_h\right)_{\s\in T}$ is invertible by \cref{thm:IntersectionNumber}, we can verify that (\ref{HomologyAndLFHomology}) is an isomorphism. By taking the dual of (\ref{HomologyAndLFHomology}), the canonical morphism 
\begin{equation}
{\rm can}:\Homo^n_{r.d.}\left( \pi^{-1}(z)^{an},\nabla_z^\vee\right)\rightarrow\Homo^n_{dR}\left( \pi^{-1}(z),\nabla_z^\vee\right)
\end{equation}
is also an isomorphism. Thus, the cohomology intersection number $\langle\varphi,\psi\rangle_{ch}$ is well-defined as $\langle\varphi,{\rm can}^{-1}(\psi)\rangle_{ch}$. Then, by \cref{thm:fundamentalthm3} we have 
\begin{align}
\int_{\Gamma_{\s,0}}\Phi\varphi=&(2\pi\ii)^{n+k} f_{\s,0}\left(z;
\begin{pmatrix}
\gamma-{\bf b}\\
c+{\bf a}
\end{pmatrix}
\right)\nonumber\\
=&(2\pi\ii)^{n+k}
C_\s(\gamma-{\bf b})
\varepsilon_\s\left(\begin{pmatrix}
\gamma-{\bf b}\\
c+{\bf a}
\end{pmatrix},{\bf 0}\right)
\varphi_{\s,{\bf 0}}\left(z;
\begin{pmatrix}
\gamma-{\bf b}\\
c+{\bf a}
\end{pmatrix}
\right).
\end{align}

\noindent
and

\begin{align}
\int_{\check{\Gamma}_{\s,0}}\Phi^{-1}\psi=&(2\pi\ii)^{n+k}f_{\s,0}^\vee\left( z;
\begin{pmatrix}
\gamma+{\bf b}^\prime\\
c-{\bf a}^\prime
\end{pmatrix}
\right)\nonumber\\
=&(2\pi\ii)^{n+k}
\exp\left\{ -\pi\ii S_0
\begin{pmatrix}
\gamma+{\bf b}^\prime\\
c-{\bf a}^\prime
\end{pmatrix}
\right\}\times\nonumber\\
&C_\s(-\gamma-{\bf b}^\prime)\varepsilon_\s\left(
\begin{pmatrix}
-\gamma-{\bf b}^\prime\\
-c+{\bf a}^\prime
\end{pmatrix},
{\bf 0}
\right)
\varphi_{\s,{\bf 0}}^\vee\left(z;
\begin{pmatrix}
\gamma+{\bf b}^\prime\\
c-{\bf a}^\prime
\end{pmatrix}
\right).
\end{align}

\noindent
In view of these formulae, we can conclude that $\varphi$ and $\psi$ are non-zero as cohomology classes. We can take a basis $\{ [\varphi_j]\}_{j=1}^L$ (resp. $\{[\psi_j]\}_{j=1}^L$) of the cohomology group $\Homo^{n}_{dR}\left(\pi^{-1}(z),\nabla_z\right)$ (resp. $\Homo^{n}_{dR}\left(\pi^{-1}(z),\nabla_z^\vee\right)$) so that $\varphi_1=\varphi$ and $\psi_1=\psi$. We also take a basis $\{ [\Gamma_{\s,0}]\}_{\s\in T}$ (resp. $\{ [\check{\Gamma}_{\s,0}]\}_{\s\in T}$) of the homology group $\Homo_n^{r.d.}\left(\pi^{-1}(z)^{an},\nabla_z^\vee\right)$ (resp. $\Homo_n^{r.d.}\left(\pi^{-1}(z)^{an},\nabla_z\right)$). Then, $(1,1)$ entry of the quadratic relation (\ref{GeneralQuadraticRelation}) is
\begin{equation}\label{Formula73}
\sum_{\s\in T}\langle \Gamma_{\s,0},\check{\Gamma}_{\s,0}\rangle_h^{-1}\left( \int_{\Gamma_{\s,0}}\Phi\varphi\right)\left( \int_{\check{\Gamma}_{\s,0}}\Phi^{-1}\psi\right)=\langle\varphi,\psi\rangle_{ch}.
\end{equation}
Formula (\ref{Formula73}) combined with \cref{thm:IntersectionNumber} will lead to the desired formula. Note that we have $\varepsilon_\s\left(\begin{pmatrix}
\gamma-{\bf b}\\
c+{\bf a}
\end{pmatrix},{\bf 0}\right)$ $=\varepsilon_\s(\delta,{\bf 0})$ and $\varepsilon_\s\left(
\begin{pmatrix}
-\gamma-{\bf b}^\prime\\
-c+{\bf a}^\prime
\end{pmatrix},
{\bf 0}
\right)=\varepsilon_\s(-\delta,{\bf 0})$ by our assumption $\det A_\s=\pm 1$.
\end{proof}

\begin{rem}
It is expected that the cohomology intersection number $
\langle x^{\bf a}h^{\bf b}\frac{dx}{x},x^{{\bf a}^\prime}h^{{\bf b}^\prime}\frac{dx}{x}\rangle_{ch}
$ is a rational function in $\delta$ and $z$ with coefficients in $\Q$. This is proved only when the GKZ system is regular holonomic. See \cite[Theorem 2.9]{GotoMatsubara}.
\end{rem}

\begin{exa}{\bf (Appell's $F_1$-series)}

We consider a one dimensional integral $f_\Gamma(z)=\int_\Gamma(z_1+z_4x)^{-c_1}(z_2+z_5x)^{-c_2}(z_3+z_6x)^{-c_3}x^{c_4}\frac{dx}{x}$. In this case, the $A$ matrix is given by $
A=
\begin{pmatrix}
1&0&0&1&0&0\\
0&1&0&0&1&0\\
0&0&1&0&0&1\\
0&0&0&1&1&1
\end{pmatrix}
$ and the parameter vector is 
$
c={}^t(c_1,c_2,c_3,c_4)$. The associated GKZ system $M_A(c)$ is related to the differential equations satisfied by Appell's $F_1$ functions (\cite[Vol.1, \S5.7.1]{ErdelyiEtAl}, \cite{DworkLoeser}). As a regular triangulation, we can take $T=\{ 1234,2346,2456\}$. The local system in question is associated to the multivalued function $\Phi=(z_1+z_4x)^{-c_1}(z_2+z_5x)^{-c_2}(z_3+z_6x)^{-c_3}x^{c_4}$. By \cite{MatsumotoIntersection}, if we take $\varphi=\frac{dx}{x}\in\Homo^1_{dR}(\Gm\setminus\{ -\frac{z_1}{z_4},-\frac{z_2}{z_5},-\frac{z_3}{z_6}\};\nabla_z)$ and $\psi=\frac{dx}{x}\in\Homo^1_{dR}(\Gm\setminus\{ -\frac{z_1}{z_4},-\frac{z_2}{z_5},-\frac{z_3}{z_6}\};\nabla_z^\vee)$, we have a formula $\langle\varphi,\psi\rangle_{ch}=2\pi\ii\frac{c_1+c_2+c_3}{c_4(c_1+c_2+c_3-c_4)}$. Applying \cref{thm:QuadraticRelation} and taking a restriction to $z_2=z_3=z_4=z_6=1$, we obtain a new identity for Appell's $F_1$-series:
\begin{align}
 &\frac{c_1}{c_4(c_1-c_4)}F_1\left(\substack{c_4,c_2,c_3\\ 1+c_4-c_1};z_1z_5,z_1\right)F_1\left(\substack{-c_4,-c_2,-c_3 \\1-c_4+c_1};z_1z_5,z_1\right) \nonumber\\
 &+\frac{c_3}{(c_1+c_3-c_4)(c_4-c_1)}G_2(c_1,c_2,c_4-c_1,c_1+c_3-c_4;-z_1,-z_5)G_2(-c_1,-c_2,c_1-c_4,c_4-c_1-c_3;-z_1,-z_5) \nonumber\\
 &+\frac{c_2}{(c_1+c_2+c_3-c_4)(c_4-c_1-c_3)}F_1\left(\substack{c_1+c_2+c_3-c_4,c_1,c_3 \\1+c_1+c_3-c_4};z_1z_5,z_5\right)F_1\left(\substack{c_4-c_1-c_2-c_3,-c_1,-c_3 \\1+c_4-c_1-c_3};z_1z_5,z_5\right) \nonumber\\
=&\frac{c_1+c_2+c_3}{c_4(c_1+c_2+c_3-c_4)}
\end{align}

\noindent
Here, we have put
\begin{equation}
F_1\left(\substack{a,b,b^\prime\\ c};x,y\right)=\sum_{m,n\geq 0}\frac{(a)_{m+n}(b)_m(b^\prime)_n}{(c)_{m+n}m!n!}x^my^n
\end{equation}
and
\begin{equation}
G_2(a,a^\prime,b,b^\prime;x,y)=\sum_{m,n\geq 0}\frac{(a)_m(a^\prime)_n(b)_{n-m}(b^\prime)_{m-n}}{m!n!}x^my^n.
\end{equation}

\end{exa}

\begin{exa}{\bf (Horn's $\Phi_1$-series)}

We consider a one dimensional integral $f_\Gamma(z)=\int_\Gamma e^{z_1x}(z_2+z_3x)^{-\gamma_1}(z_4+z_5x)^{-\gamma_2}x^c\frac{dx}{x}$. In this case, the $A$ matrix is given by 
$
A=
\begin{pmatrix}
0&1&1&0&0\\
0&0&0&1&1\\
1&0&1&0&1
\end{pmatrix}
$ and the associated GKZ system $M_A(\delta)$ is related to the differential equations satisfied by Horn's $\Phi_1$-series (\cite[Vol.1, \S5.7.1]{ErdelyiEtAl}, \cite{DworkLoeser}). As a convergent regular triangulation, we take $T=\{ 135,234,345\}$. By \cite{MajimaMatsumotoTakayama}, if we take $\varphi=\frac{dx}{x}\in\Homo^1_{dR}(\Gm\setminus\{ -\frac{z_2}{z_3},-\frac{z_4}{z_5}\};\nabla_z)$ and $\psi=\frac{dx}{x}\in\Homo^1_{dR}(\Gm\setminus\{ -\frac{z_2}{z_3},-\frac{z_4}{z_5}\};\nabla_z^\vee)$, we have a formula $\langle\varphi,\psi\rangle_{ch}=\frac{2\pi\ii}{c}$. Applying \cref{thm:QuadraticRelation}, we obtain a new  identity for Horn's $\Phi_1$-series:

\begin{align}
&(c-\gamma_1-\gamma_2)(\gamma_1-c)\nonumber\\
=&c(\gamma_1-c)\Phi_2\left(\substack{\gamma_1,\gamma_2\\ 1+\gamma_1+\gamma_2-c};-zw,-w\right)\Phi_2\left(\substack{-\gamma_1,-\gamma_2\\ 1-\gamma_1-\gamma_2+c};zw,w\right)\nonumber\\
&+\gamma_1(c-\gamma_1-\gamma_2)\Phi_1\left(\substack{c,\gamma_2\\ 1+c-\gamma_1};z,-zw\right)\Phi_1\left(\substack{-c,-\gamma_2\\ 1-c+\gamma_1};z,zw\right)\nonumber\\
&+c\gamma_2\Gamma_1\left(\gamma_1,c-\gamma_1,\gamma_1+\gamma_2-c;-z,w\right)\Gamma_1\left(-\gamma_1,-c+\gamma_1,-\gamma_1-\gamma_2+c;-z,-w\right).
\end{align}

\noindent
Here, the series $\Phi_1\left(\substack{\alpha,\beta\\ \gamma};x,y\right),$ $\Phi_2\left(\substack{\beta_1,\beta_2\\ \gamma};x,y\right)$, and $\Gamma_1\left(\alpha,\beta_1,\beta_2;x,y\right)$ are given by

\begin{equation}
\Phi_1\left(\substack{\alpha,\beta\\ \gamma};x,y\right)
=\sum_{m,n=0}^\infty\frac{(\alpha)_{m+n}(\beta)_m}{(\gamma)_{m+n}m!n!}x^my^n,
\end{equation}

\begin{equation}
\Phi_2\left(\substack{\beta_1,\beta_2\\ \gamma};x,y\right)
=\sum_{m,n=0}^\infty\frac{(\beta_1)_{m}(\beta_2)_n}{(\gamma)_{m+n}m!n!}x^my^n,
\end{equation}
and
\begin{equation}
\Gamma_1\left(\alpha,\beta_1,\beta_2;x,y\right)
=\sum_{m,n=0}^\infty\frac{(\alpha)_{m}(\beta_1)_{n-m}(\beta_2)_{m-n}}{m!n!}x^my^n.
\end{equation}

\end{exa}

\section{Quadratic relations for Aomoto-Gelfand system}\label{QuadraticRelationsForGrassman}
In this section, we apply \cref{thm:QuadraticRelation} to the so-called Aomoto-Gelfand hypergeometric functions (\cite{AomotoKita}, \cite{GelfandGraevRetakh}). This class enjoys a special combinatorial structure. To begin with, we revise the general result on this class of hypergeometric functions based on \cite{GelfandGraevRetakh}. Let $k\leq n$ be two natural numbers. We consider the following integral
\begin{equation}
f_\Gamma(z)=\int_\Gamma\prod_{j=0}^nl_j(x;z)^{\alpha_j}\omega(x)=\int_\Gamma\prod_{j=0}^n(z_{0j}x_0+\dots+z_{kj}x_k)^{\alpha_j}\omega(x)
\end{equation}
where $\omega(x)=\displaystyle\sum_{i=0}^k(-1)^ix_idx_{\hat{i}}\in\Gamma(\mathbb{P}^k,\Omega^k_{\mathbb{P}^k}(k+1))$ and $z=(z_{ij})_{\substack{i=0,\dots,k\\ j=0,\dots,n}}\in Z_{k+1,n+1}.$ Here, $Z_{k+1,n+1}$ denotes the space of all $(k+1)\times (n+1)$ matrices with entries in $\C$. The Aomoto-Gelfand system $E(k+1,n+1)$ is defined, with the aid of parameters $\alpha_0,\dots,\alpha_n\in\C$ such that $\alpha_0+\dots+\alpha_n=-(k+1)$ by the formula

\begin{equation}
E(k+1,n+1):
\begin{cases}
\displaystyle\sum_{i=0}^kz_{ij}\frac{\partial f}{\partial z_{ij}}=\alpha_jf&(j=0,\dots,n)\\
\displaystyle\sum_{j=0}^nz_{ij}\frac{\partial f}{\partial z_{pj}}=-\delta_{ip}f&(i,p=0,1,\dots,k)\\
\frac{\partial^2 f}{\partial z_{ij}\partial z_{pq}}=\frac{\partial^2 f}{\partial z_{pj}\partial z_{iq}}& (i,p=0,1,\dots,k,\quad j,q=0,\dots, n).
\end{cases}
\end{equation}

\noindent
If we take a restriction to $z=
\begin{pmatrix}
1&         & &z_{0k+1}&\cdots&z_{0n}\\
 &\ddots& & \vdots &\ddots&\vdots\\
  &         &1&z_{kk+1}&\cdots&z_{kn}
\end{pmatrix}
$ and $x_0=1$, our integral $f_\Gamma(z)$ becomes 
\begin{equation}
f_\Gamma(z)=\int_\Gamma\prod_{j=k+1}^nl_j(x;z)^{\alpha_j}x_1^{\alpha_1}\dots x_k^{\alpha_k}dx.
\end{equation}
If we put $c={}^t(
\alpha_0+1,\dots,\alpha_k+1,-\alpha_{k+1},\dots,-\alpha_n)$, and put ${\bf a}(i,j)={\bf e}(i)+{\bf e}(j)$ $(i=0,1,\dots,k,j=k+1,\dots,n)$, where ${\bf e}(s)$ is the standard basis of $\Z^{(n+1)\times 1}$, $f_\Gamma(z)$ is a solution of $M_A(c)$ with $A=({\bf a}(i,j))_{\substack{i=0,\dots,k\\ j=k+1,\dots,n}}$. The system $M_A(c)$ is explicitly given by
\begin{equation}
M_A(c):
\begin{cases}
\displaystyle\sum_{i=0}^kz_{ij}\frac{\partial f}{\partial z_{ij}}=-c_jf&(j=k+1,\dots,n)\\
\displaystyle\sum_{j=k+1}^nz_{ij}\frac{\partial f}{\partial z_{ij}}=-c_if&(i=0,1,\dots,k)\\
\frac{\partial^2 f}{\partial z_{ij}\partial z_{pq}}=\frac{\partial^2 f}{\partial z_{pj}\partial z_{iq}}& (i,p=0,1,\dots,k,\quad j,q=k+1,\dots, n).
\end{cases}
\end{equation}

\noindent
We also put $\tilde{\bf a}(i,j)=-{\bf e}(i)+{\bf e}(j)$ $(i=0,1,\dots,k,j=k+1,\dots,n)$ and $\tilde{A}=(\tilde{\bf a}(i,j))_{\substack{i=0,\dots,k\\ j=k+1,\dots,n}}$. Note that this configuration is equivalent to ${\bf a}(i,j)$ via the isomorphism of the lattice $\Z^{(n+1)\times 1}$ given by ${}^t(m_0,\dots,m_n)\mapsto {}^t(-m_0,\dots,-m_k,m_{k+1},\dots,m_n)$. We should also be aware that $\tilde{A}$ does not generate the ambient lattice $\Z^{(n+1)\times 1}$ hence neither does $A$. However, since the quotient $\Z^{(n+1)\times 1}/\Z A$ is torsion free, we can apply the results of previous sections by, for example, considering a projection $p:\Z^{(n+1)\times 1}\rightarrow\Z^{n\times 1}$ which sends ${\bf e}(0)$ to $0$ and keeps other standard basis ${\bf e}(s)$ ($s=1,\dots,n$). Thus, if we set $A^\prime=pA$ and $c^\prime=p(c)$, it can readily be seen that the GKZ system $M_A(c)$ is equivalent to the reduced GKZ system $M_{A^\prime}(c^\prime)$.  

We consider the special regular triangulation called ``staircase triangulation'' (\cite[\S 6.2.]{DeLoeraRambauSantos},\cite[\S 8.2.]{GelfandGraevRetakh}). A subset $I\subset\{ 1,\dots,k\}\times\{ k+1,\dots,n\}$ is called a ladder if $|I|=n$ and if we write $I=\{ (i_1,j_1),\dots,(i_n,,j_n)\}$, we have $(i_1,j_1)=(k,k+1)$ and $(i_n,,j_n)=(0,n)$ and $(i_{p+1},j_{p+1})=(i_{p}+1,,j_p)$ or $(i_p,j_p+1)$ for any $p=1,\dots,n-1$. It can readily be seen that any ladder $I$ is a simplex. Moreover, the collection of all ladders $T=\{ I\mid I:\text{ladder}\}$ forms a regular triangulation. This regular triangulation $T$ is called the staircase triangulation. It is also known that staircase triangulation $T$ is unimodular. For any ladder $I\in T$, we consider the equation $Av^I=-c$ such that $v^I_{ij}=0\quad ((i,j)\notin I)$. Defining $\tilde{c}_l=
\begin{cases}
c_l&(l=0,\dots, k)\\
-c_l&(l=k+1,\dots,n),
\end{cases}
$ it is equivalent to the system $\tilde{A}v^I=\tilde{c}$. This equation can be solved in a unique way. We can even obtain an explicit formula for $v^I$ by means of graph theory. For each ladder $I$, we can associate a tree $G_I$ of a complete bipatite graph $K_{k+1,n-k}$. Recall that the complete bipartite graph $K_{k+1,n-k}$ consists of the set of vertices $V(K_{k+1,n-k})=\{ 0,\dots,n\}$ and the set of edges $E(K_{k+1,n-k})=\left\{ (i,j)\mid \substack{i=0,\dots,k\\ j=k+1,\dots,n}\right\}$. For a given ladder $I=\{ (i_1,j_1),\dots,(i_n,,j_n)\}$, we associate a tree $G_I$ so that edges are $E(G_I)=\{ (i_s,j_s)\}_{s=1}^n$ and vertices are $V(G_I)=\{ 0,\dots,n\}$. Let us introduce the dual basis $\phi(l)$ $(l=0,\dots,n)$ to ${\bf e}(l)$. For any edge $(i,j)\in G_I$, we can easily confirm that $G_I\setminus (i,j)$ has exactly two connected components. The connected component which contains $i$ (resp. $j$) is denoted by $C_i(i,j)$ (resp. $C_j(i,j)$). For each $(i,j)\in G_I$, we put
\begin{equation}
\varphi(ij)=\sum_{l\in V(C_j(i,j))}\phi(l).
\end{equation}

\begin{figure}[H]
\begin{minipage}{0.5\hsize}
\begin{center}
\begin{tikzpicture}
\draw[-] (0,0) -- (0,4);
\draw[-] (0,0) -- (5,0);
\draw[-] (5,0) -- (5,4);
\draw[-] (0,4) -- (5,4);
\draw[-] (1,0) -- (1,4);
\draw[-] (2,0) -- (2,4);
\draw[-] (3,0) -- (3,4);
\draw[-] (4,0) -- (4,4);
\draw[-] (0,1) -- (5,1);
\draw[-] (0,2) -- (5,2);
\draw[-] (0,3) -- (5,3);
\node at (0.5,3.5) {$\bullet$};
\node at (1.5,3.5) {$\bullet$};
\node at (1.5,2.5) {$\bullet$};
\node at (2.5,2.5) {$\bullet$};
\node at (2.5,1.5) {$\bullet$};
\node at (3.5,1.5) {$\bullet$};
\node at (4.5,1.5) {$\bullet$};
\node at (4.5,0.5) {$\bullet$};
\node at (-0.5,3.5) {$3$};
\node at (-0.5,2.5) {$2$};
\node at (-0.5,1.5) {$1$};
\node at (-0.5,0.5) {$0$};
\node at (0.5,-0.5) {$4$};
\node at (1.5,-0.5) {$5$};
\node at (2.5,-0.5) {$6$};
\node at (3.5,-0.5) {$7$};
\node at (4.5,-0.5) {$8$};
\end{tikzpicture}
\caption{ladder}
\end{center}
\end{minipage}
\begin{minipage}{0.5\hsize}
\begin{center}
\begin{tikzpicture}
\node at (-1.5,3.5) {$3$};
\node at (-1.5,2.5) {$2$};
\node at (-1.5,1.5) {$1$};
\node at (-1.5,0.5) {$0$};
\node at (1.5,4) {$4$};
\node at (1.5,3) {$5$};
\node at (1.5,2) {$6$};
\node at (1.5,1) {$7$};
\node at (1.5,0) {$8$};
\draw[-] (-1,3.5) -- (1,4);
\draw[-] (-1,3.5) -- (1,3);
\draw[-] (-1,2.5) -- (1,3);
\draw[-] (-1,2.5) -- (1,2);
\draw[-] (-1,1.5) -- (1,2);
\draw[-] (-1,1.5) -- (1,1);
\draw[-] (-1,1.5) -- (1,0);
\draw[-] (-1,0.5) -- (1,0);
\end{tikzpicture}
\caption{spanning tree corresponding to the ladder}
\end{center}
\end{minipage}
\end{figure}

\begin{figure}[H]
\begin{minipage}{0.5\hsize}
\begin{center}
\begin{tikzpicture}
\node at (-1.5,3.5) {$3$};
\node at (1.5,4) {$4$};
\node at (1.5,3) {$5$};
\draw[-] (-1,3.5) -- (1,4);
\draw[-] (-1,3.5) -- (1,3);
\end{tikzpicture}
\caption{connected component $C_5(2,5)$}
\end{center}
\end{minipage}
\begin{minipage}{0.5\hsize}
\begin{center}
\begin{tikzpicture}
\node at (-1.5,2.5) {$2$};
\node at (-1.5,1.5) {$1$};
\node at (-1.5,0.5) {$0$};
\node at (1.5,2) {$6$};
\node at (1.5,1) {$7$};
\node at (1.5,0) {$8$};
\draw[-] (-1,2.5) -- (1,2);
\draw[-] (-1,1.5) -- (1,2);
\draw[-] (-1,1.5) -- (1,1);
\draw[-] (-1,1.5) -- (1,0);
\draw[-] (-1,0.5) -- (1,0);
\end{tikzpicture}
\caption{connected component $C_2(2,5)$}
\end{center}
\end{minipage}
\end{figure}

\begin{prop}\label{prop:GraphicalInverse}
For $(i,j),(i^\prime,j^\prime)\in I$, we have 
\begin{equation}
\langle \varphi(ij),\tilde{\bf a}(i^\prime,j^\prime)\rangle=
\begin{cases}
1&((i,j)=(i^\prime,j^\prime))\\
0&(otherwise),
\end{cases}
\end{equation}
where $\langle\bullet,\bullet\rangle$ is the duality pairing.
\end{prop}

\begin{proof}
Suppose $(i^\prime,j^\prime)\in C_i(i,j)$. Then we have $\langle \varphi(ij),\tilde{\bf a}(i^\prime,j^\prime)\rangle=0.$ On the other hand, if $(i^\prime,j^\prime)\in C_j(i,j)$, we see $\langle \varphi(ij),\tilde{\bf a}(i^\prime,j^\prime)\rangle=\langle \phi(i^\prime)+\phi(j^\prime),\tilde{\bf a}(i^\prime,j^\prime)\rangle=0.$ Finally, since $i\notin V(C_j(i,j))$ and $j\in V(C_j(i,j))$, we have $\langle \varphi(ij),\tilde{\bf a}(i,j)\rangle=1$.
\end{proof}

\noindent
Therefore, we obtain a
\begin{cor}
Under the notation above, one has
\begin{equation}
v^I_{ij}=\sum_{l\in V(C_j(i,j))}\tilde{c}_l.
\end{equation}
\end{cor}

\noindent
Substitution of this formula to $\Gamma$-series yields the formula
\begin{equation}\label{eqn:9.8}
\varphi_{v^I}(z)=z_I^{v^I}\sum_{u_{\bar{I}}\in\Z^{\bar{I}}_{\geq 0}}\frac{\left(z_I^{-\langle \varphi(I),\tilde{A}_{\bar{I}}\rangle}z_{\bar{I}}\right)^{u_{\bar{I}}}}{\displaystyle\prod_{(i,j)\in I}\Gamma(1+v_{ij}^I-\langle\varphi(ij),\tilde{A}_{\bar{I}}u_{\bar{I}}\rangle)u_{\bar{I}}!}.
\end{equation}
Since the series (\ref{eqn:9.8}) is defined by means of a ladder $I$ and a parameter $\alpha$, it is also denoted by $f_I(z;\alpha)$. We write $f_I^\vee(z;\alpha)$ for the series $\varphi_{-v^I}(z)$.

Next, we consider the de Rham cohomology group $\Homo_{\rm dR}^k\left( \PP_x^k\setminus\bigcup_{j=0}^n\{ l_{j}(x;z)=0\},\nabla_z\right)$ with $\nabla_z=d_x+\sum_{j=0}^n\tilde{c}_jd_x\log l_{j}(x;z)\wedge$. Note that we identify the set of rational differential forms on $\PP^k_x$ of homogeneous degree $0$ with that on $\{ l_0(x;z)\neq 0\}\simeq\A^k$. As a convenient basis of the twisted cohomology group, we take the one of \cite{GotoMatsumotoContiguity}. We consider matrix variables $z=
\begin{pmatrix}
1&         & &z_{0k+1}&\cdots&z_{0n}\\
 &\ddots& & \vdots &\ddots&\vdots\\
  &         &1&z_{kk+1}&\cdots&z_{kn}
\end{pmatrix}.
$ For any subset $J=\{ j_0,\dots,j_k\}\subset\{0,\dots,n\}$ with cardinality $k+1$, we write $z_J$ for the submatrix of $z$ consisting of column vectors indexed by $J$. We always assume $j_0<\dots<j_k$. We put
\begin{equation}
\omega_J(z;x)=d_x\log\left(\frac{l_{j_1}(x;z)}{l_{j_0}(x;z)}\right)\wedge\dots\wedge d_x\log\left(\frac{l_{j_k}(x;z)}{l_{j_0}(x;z)}\right).
\end{equation}
By a simple computation, we see that $\omega_J(x;z)=\displaystyle\sum_{p=0}^k(-1)^pl_{j_p}(x;z)\frac{d_xl_{j_0}\wedge\dots\wedge \widehat{d_xl_{j_p}}\wedge\dots\wedge d_xl_{j_k}}{l_{j_0}(x;z)\cdots l_{j_k}(x;z)}.
$ 
As in \cite[Fact 2.5]{GotoMatsumotoContiguity}, we also see that
$
\displaystyle\sum_{p=0}^k(-1)^pl_{j_p}(x;z)d_xl_{j_0}\wedge\dots\wedge \widehat{d_xl_{j_p}}\wedge\dots\wedge d_xl_{j_k}=\det(z_J)\omega(x).
$ 
Therefore, we have $\omega_J(x;z)=\det(z_J)\frac{\omega(x)}{l_{j_0}(x;z)\cdots l_{j_k}(x;z)}.$ We set $\mathcal{J}\overset{def}{=}\llbracket 0,n\rrbracket\overset{def}{=}\{ 0,\dots,n\}$. Then, for any distinct elements $p,q\in\mathcal{J}$, we set ${}_q\mathcal{J}_p=\{ J\subset\mathcal{J}\mid |J|=k, q\notin J, p\in J\}$. \cite[Proposition 3.3]{GotoMatsumotoContiguity} tells us that the set $\{ \omega_J\}_{J\in{}_q\mathcal{J}_p}$ is a basis of $\Homo_{\rm dR}^k\left( \PP_x^k\setminus\bigcup_{j=0}^n\{ l_{j}(x;z)=0\},\nabla_z\right)$.

Now we are going to derive a quadratic relation for $f_I(z;\alpha)$. We take any pair of  subsets $J,J^\prime\subset\{0,\dots,n\}$ with cardinality $k+1$. Let us put $J_a=J\cap\{ 1,\dots,k\}$, $J_a^\prime=J^\prime\cap\{ 1,\dots,k\}$, $J_b=J\cap\{ k+1,\dots,n\}$, and $J_b^\prime=J^\prime\cap\{ k+1,\dots,n\}$. We write ${\bf 1}_{J_a}$ (resp. ${\bf 1}_{J_b}$) for the vector $\sum_{j\in J_a}{\bf e}(j)$ (resp. $\sum_{j\in J_b}{\bf e}(j)$). If we write $\alpha$ as $\displaystyle\sum_{j=0}^n\alpha_j{\bf e}(j)$, we also put $\alpha_a=\displaystyle\sum_{j=1}^k\alpha_j{\bf e}(j)$ and $\alpha_{b}=\displaystyle\sum_{j=k+1}^n\alpha_j{\bf e}(j)$. We can readily confirm the identities
\begin{equation}
\frac{\omega_J(x;z)}{\det(z_J)}=\frac{\omega(x)}{l_{j_0}(x;z)\cdots l_{j_k}(x;z)}=\frac{x_0\dots x_k}{l_{j_0}(x;z)\cdots l_{j_k}(x;z)}\frac{\omega(x)}{x_0\dots x_k}=x^{{\bf 1}-{\bf 1}_{J_a}}l^{-{\bf 1}_{J_b}}\frac{dx}{x}.
\end{equation}

\noindent
On the other hand, by \cite{MatsumotoIntersection}, we know
\begin{equation}\label{MatsumotoFormula}
\frac{\langle \omega_{J}(x;z),\omega_{J^\prime}(x;z)\rangle_{ch}}{(2\pi\ii)^k}=
\begin{cases}
\frac{\sum_{j\in J}\tilde{c}_j}{\prod_{j\in J}\tilde{c}_j}& (J=J^\prime)\\
\frac{\text{sgn} (J^\prime,J)}{\prod_{j\in J\cap J^\prime}\tilde{c}_j}& (\sharp(J\cap J^\prime)=k)\\
0&(otherwise)
\end{cases}.
\end{equation}
Here, sgn$(J,J^\prime)$ is defined to be $(-1)^{p+q}$ where $p$ and $q$ are chosen so that $J^\prime\setminus\{ j^\prime_p\}=J\setminus\{ j_q\}$. Setting ${\bf 1}_{\llbracket 0,k\rrbracket}=\sum_{j=0}^k{\bf e}(j)$, the quadratic relation (\ref{GeneralQuadraticRelation}) leads to the following

\begin{thm}\label{thm:QuadraticRelationsForAomotoGelfand}
Under the notation as above, for any $z\in U_T$, we have an identity
\begin{align}
&(-1)^{|J_b|+|J^\prime_b|+k}\alpha_{k+1}\dots\alpha_n(-\alpha_b+{\bf 1}_{J_b})_{-{\bf 1}_{J_b}}(\alpha_b+{\bf 1}_{J^\prime_b})_{-{\bf 1}_{J_b^\prime}}\times\nonumber\\
&\sum_{I:\text{ladder}}\frac{\pi^n}{\displaystyle\prod_{(i,j)\in I}\sin\pi v_{ij}^I}f_I(z;\alpha+{\bf 1}_{\llbracket 0,k\rrbracket}-{\bf 1}_J)f_I^\vee(z;\alpha-{\bf 1}_{\llbracket 0,k\rrbracket}+{\bf 1}_{J^\prime})\nonumber\\
=&
\det(z_J)^{-1}\det(z_{J^\prime})^{-1}\frac{\langle \omega_{J}(x;z),\omega_{J^\prime}(x;z)\rangle_{ch}}{(2\pi\ii)^k}.\label{EknQuadraticRelation}
\end{align}
Here, the right hand side is explicitly determined by (\ref{MatsumotoFormula}).
\end{thm}

\begin{rem}
Since the right-hand side (\ref{EknQuadraticRelation}) is a rational function in the parameters $\alpha_j$, (\ref{EknQuadraticRelation}) holds without any restriction on the parameters $\alpha_j$.
\end{rem}

\begin{exa}{\bf (Gau\ss' hypergeometric series)}

The simplest case is $E(2,4)$. This amounts to the classical Gau\ss' hypergeometric functions. By computing the cohomology intersection number $\langle\frac{dx}{x},\frac{dx}{x}\rangle_{ch}$, we obtain a quadratic relation (\ref{QRGauss}) in the introduction. Note in particular that this identity implies a series of combinatorial identities 
\begin{align}
&(1-\gamma+\alpha)(1-\gamma+\beta)\sum_{l+m=n}\frac{(\alpha)_l(\beta)_l}{(\gamma)_l(1)_l}\frac{(-\alpha)_m(-\beta)_m}{(2-\gamma)_m(1)_m}\nonumber\\
=&\alpha\beta\sum_{l+m=n}\frac{(\gamma-\alpha-1)_l(\gamma-\beta-1)_l}{(\gamma)_l(1)_l}\frac{(1-\gamma+\alpha)_m(1-\gamma+\beta)_m}{(2-\gamma)_m(1)_m}
\end{align}
where $n$ is a positive integer.
\end{exa}

\begin{exa}{\bf (Hypergeometric function of type $E(3,6)$)}\label{exa:AomotoGelfand}

This type of hypergeometric series was discussed by several authors (cf. \cite{MSY},\cite{MSTY}). The integral we consider is $f_\Gamma(z)=\int_\Gamma\prod_{j=3}^5(z_{0j}+z_{1j}x_1+z_{2j}x_2)^{-c_j}x_1^{c_1}x_2^{c_2}\frac{dx_1\wedge dx_2}{x_1x_2}$. The (reduced) $A$ matrix is given by 
$
A^\prime=
\bordermatrix{
    &z_{03}&z_{04}&z_{05}&z_{13}&z_{14}&z_{15}&z_{23}&z_{24}&z_{25}\cr
c_3&1&0&0&1&0&0&1&0&0\cr
c_4&0&1&0&0&1&0&0&1&0\cr
c_5&0&0&1&0&0&1&0&0&1\cr
c_1&0&0&0&1&1&1&0&0&0\cr
c_2&0&0&0&0&0&0&1&1&1
}.
$ The associated arrangement of hyperplanes is described as in Figure \ref{fig:AomotoGelfandConfiguration}. Let us put $H_j=\{ x\in\C^2\mid l_j(x;z)=0\}$ for ($j=1,\dots,5$). We write $H_0$ for the hyperplane at infinity $H_0=\mathbb{P}^2\setminus\C^2$. As was clarified in \S\ref{SectionEuler}, each ladder ($=$simplex) induces a degeneration of arrangements. The rule is simple: for each ladder $I$, we let variables $z_{\bar{I}}$ corresponding to the complement of $I$ go to $0$ while we keep variables $z_I$ corresponding to $I$ fixed. For example, if we take a ladder $\{ 23,24,25,15,05\}$, the induced degeneration is $z_{13},z_{14},z_{03},z_{04}\rightarrow 0$. By taking this limit, the hyperplanes $H_3$ and $H_4$ both tend to the hyperplane $H_2$ ($x_1$ axis) which is simply denoted by $\substack{ H_3\rightarrow H_2\\ H_4\rightarrow H_2}$. Therefore, there only remain 3 hyperplanes after this limit: $H_1,H_2$ and $H_5$. Restricted to the real domain they form a chamber when variables $z_{ij}$ are all real and generic. We consider the Pochhammer cycle associated to this bounded chamber. The important point of this construction is that, unlike the usual Pohhammer cycle, we have to go around several divisors at once. In this case, $H_3$ and $H_4$ should be regarded as a perturbation of $H_2$. Therefore, they are linked as in Figure \ref{fig:AomotoGelfandConfiguration}. We call such a cycle ``linked cycle'' (or ``Erd\'elyi cycle'' after the work of Erd\'elyi \cite{ErdelyiActa} where this type of cycle is called ``double circit'' in the cases of Appell's $F_1$ and its relatives). We summerize the correspondence between ladders and degenerations in the following table.

\begin{equation*}
\begin{array}{c|c|c|c|c|c|c}
\text{ladder}        
&\begin{matrix}
\bullet&\bullet&\bullet\\
 & &\bullet\\
 & &\bullet
\end{matrix}
&\begin{matrix}
\bullet&\bullet& \\
          &\bullet&\bullet\\
          &         &\bullet
\end{matrix}
&\begin{matrix}
\bullet&\bullet& \\
         &\bullet& \\
         &\bullet&\bullet
\end{matrix}
&\begin{matrix}
\bullet& & \\
\bullet&\bullet&\bullet\\
 & &\bullet
\end{matrix}
&\begin{matrix}
\bullet& & \\
\bullet&\bullet& \\
         &\bullet&\bullet
\end{matrix}
&\begin{matrix}
\bullet& & \\
\bullet& & \\
\bullet&\bullet&\bullet
\end{matrix}
\\ \hline
\text{degeneration} &\substack{H_3\rightarrow H_2\\ H_4\rightarrow H_2}& H_3\rightarrow H_2& \substack{H_3\rightarrow H_2\\ H_5\rightarrow H_0}&H_4\rightarrow H_1& H_5\rightarrow H_0&\substack{ H_4\rightarrow H_0\\ H_5\rightarrow H_0}
\end{array}
\end{equation*}

\begin{figure}[h]
\begin{center}
\begin{tikzpicture}
\draw (0,0) node[below right]{O}; 
\draw[->] (-6,0)--(6,0);
\draw (6,0) node[right]{$x_1$};
\draw[thick, ->] (0,-3)--(0,3);
\draw (0,3) node[above]{$x_2$};
\draw[-] (-5,2.5)--(6,-1);
\draw[-] (-4,3)--(2,-3);
\draw[-] (-3,-3.5)--(6,1);
\draw (6,-1) node[right]{$H_3=\{l_3=0\}$};
\draw (2,-3) node[right]{$H_4=\{l_4=0\}$};
\draw (6,1) node[above]{$H_5=\{l_5=0\}$};
\draw[-] (0.2,-1.4) to [out=70,in=100] (2,1);
\draw[dashed] (0.2,-1.4) to [out=20,in=-80] (2,1);
\draw[-] (0.1,-1.5) to [out=130,in=70] (-0.4,-1.5);
\draw[dashed] (0.1,-1.5) to [out=-110,in=-30] (-0.4,-1.5);
\draw[-] (0.25,-1.675) to [out=-50,in=20] (0.45,-2.075);
\draw[dashed] (0.25,-1.675) to [out=-100,in=-100] (0.45,-2.075);
\draw (0.1,-1.3)--(0.4,-1.6);
\draw (0.4,-1.6)--(0.1,-1.75);
\draw (0.1,-1.75)--(0.1,-1.3);
\end{tikzpicture}
\caption{Arrangement of hyperplanes and the cycle corresponding to the ladder $\{ 23,24,25,15,05\}$}
\label{fig:AomotoGelfandConfiguration}
\end{center}
\end{figure}
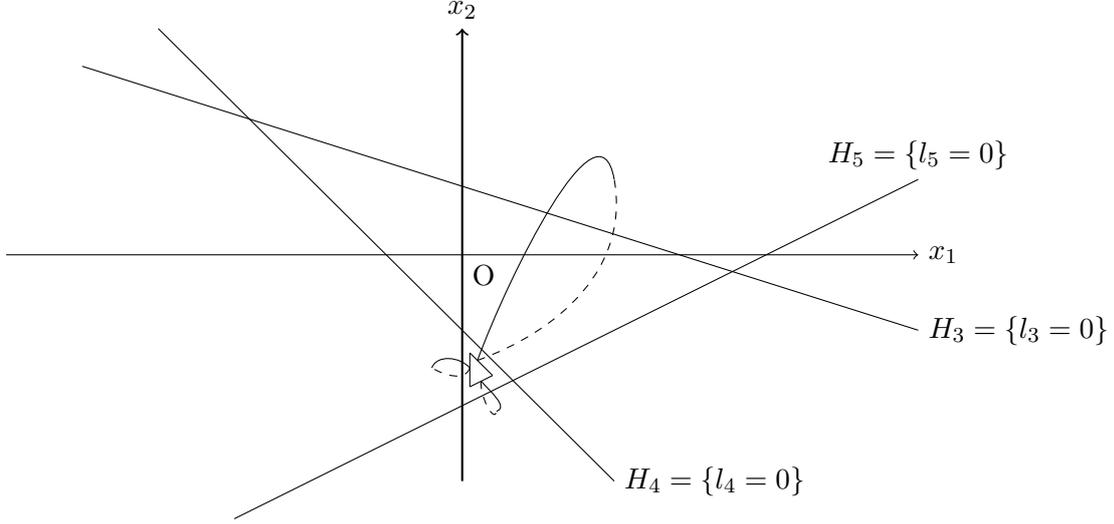

The quadratic relation with respect to the cohomology intersection number $\langle\frac{dx_1\wedge dx_2}{x_1x_2},\frac{dx_1\wedge dx_2}{x_1x_2}\rangle_{ch}$ is explicitly given by 
\begin{equation}
c_0c_1c_2c_3c_4c_5\sum_{i=1}^6\frac{\pi^5}{\sin\pi(-v_i)}\varphi_i(z;c)\varphi_i(z;-c)=c_3+c_4+c_5
\end{equation}
where parameters $c_0,\dots,c_5$ satisfy a linear relation
\begin{equation}
c_0+c_1+c_2-c_3-c_4-c_5=0
\end{equation}
and vectors $v_i$ are given by
\begin{align}
v_1&={}^t(-c_3,-c_4,c_0+c_1-c_5,-c_1,-c_0)\\
v_2&={}^t(-c_3,-c_2+c_3,-c_0-c_1+c_5,c_0-c_5,-c_0)\\
v_3&={}^t(-c_3,-c_2+c_3,-c_1,c_5-c_0,-c_5)\\
v_4&={}^t(-c_2,c_2-c_3,-c_4,c_0-c_5,-c_0)\\
v_5&={}^t(-c_2,c_2-c_3,c_0-c_4-c_5,c_5-c_0,-c_5)\\
v_6&={}^t(-c_2,-c_1,-c_0+c_4+c_5,-c_4,-c_5).
\end{align}

Below, we list the explicit formulae of $\Gamma$-series $\varphi_i(z;c)$:
\begin{align}
\varphi_1(z;c)&=z_{23}^{-c_3}z_{24}^{-c_4}z_{25}^{c_0+c_1-c_5}z_{15}^{-c_1}z_{05}^{-c_0}\nonumber\\
 &\sum_{u_{13},u_{14},u_{03},u_{04}\geq 0}\frac{1}{\Gamma(1-c_3-u_{13}-u_{03})\Gamma(1-c_4-u_{14}-u_{04})\Gamma(1+c_0+c_1-c_5+u_{13}+u_{14}+u_{03}+u_{04})}\nonumber\\
 &\frac{(z_{23}^{-1}z_{25}z_{15}^{-1}z_{13})^{u_{13}}(z_{24}^{-1}z_{25}z_{15}^{-1}z_{14})^{u_{14}}(z_{23}^{-1}z_{25}z_{05}^{-1}z_{03})^{u_{03}}(z_{24}^{-1}z_{25}z_{05}^{-1}z_{04})^{u_{04}}}{\Gamma(1-c_1-u_{13}-u_{14})\Gamma(1-c_0-u_{03}-u_{04})u_{13}!u_{14}!u_{03}!u_{04}!}
\end{align}

\begin{align}
&\varphi_2(z;c)\nonumber\\
&=z_{23}^{-c_3}z_{24}^{-c_2+c_3}z_{14}^{-c_0-c_1+c_5}z_{15}^{c_0-c_5}z_{05}^{-c_0}\nonumber\\
&\sum_{\substack{u_{25},u_{13},\\ u_{03},u_{04}\geq 0}}\frac{1}{\Gamma(1-c_3-u_{13}-u_{03})\Gamma(1-c_2+c_3-u_{25}+u_{13}+u_{03})\Gamma(1-c_0-c_1+c_5+u_{25}-u_{13}-u_{03}-u_{04})}\nonumber\\
&\frac{(z_{24}^{-1}z_{14}z_{15}^{-1}z_{25})^{u_{25}}(z_{23}^{-1}z_{24}z_{14}^{-1}z_{13})^{u_{13}}(z_{23}^{-1}z_{24}z_{14}^{-1}z_{15}z_{05}^{-1}z_{03})^{u_{03}}(z_{14}^{-1}z_{15}z_{05}^{-1}z_{04})^{u_{04}}}{\Gamma(1+c_0-c_5-u_{25}+u_{03}+u_{04})\Gamma(1-c_0-u_{03}-u_{04})u_{25}!u_{13}!u_{03}!u_{04}!}
\end{align}

\begin{align}
\varphi_3(z;c)&=z_{23}^{-c_3}z_{24}^{-c_2+c_3}z_{14}^{-c_1}z_{04}^{c_5-c_0}z_{05}^{-c_5}\nonumber\\
&\sum_{u_{25},u_{15},u_{13},u_{03}\geq 0}\frac{1}{\Gamma(1-c_3-u_{13}-u_{03})\Gamma(1-c_2+c_3-u_{25}+u_{13}+u_{03})\Gamma(1-c_1-u_{15}-u_{13})}\nonumber\\
&\frac{(z_{24}^{-1}z_{04}z_{05}^{-1}z_{25})^{u_{25}}(z_{14}^{-1}z_{04}z_{05}^{-1}z_{15})^{u_{15}}(z_{23}^{-1}z_{24}z_{14}^{-1}z_{13})^{u_{13}}(z_{23}^{-1}z_{24}z_{04}^{-1}z_{03})^{u_{03}}}{\Gamma(1+c_5-c_0+u_{25}+u_{15}-u_{03})\Gamma(1-c_5-u_{25}-u_{15})u_{25}!u_{15}!u_{13}!u_{03}!}
\end{align}

\begin{align}
\varphi_4(z;c)&=z_{23}^{-c_2}z_{13}^{c_2-c_3}z_{14}^{-c_4}z_{15}^{c_0-c_5}z_{05}^{-c_0}\nonumber\\
&\sum_{u_{24},u_{25},u_{03},u_{04}\geq 0}\frac{1}{\Gamma(1-c_2-u_{24}-u_{25})\Gamma(1+c_2-c_3+u_{24}+u_{25}-u_{03})\Gamma(1-c_4-u_{24}-u_{04})}\nonumber\\
&\frac{(z_{23}^{-1}z_{13}z_{14}^{-1}z_{24})^{u_{24}}(z_{23}^{-1}z_{13}z_{15}^{-1}z_{25})^{u_{25}}(z_{13}^{-1}z_{15}z_{05}^{-1}z_{03})^{u_{03}}(z_{14}^{-1}z_{15}z_{05}^{-1}z_{04})^{u_{04}}}{\Gamma(1+c_0-c_5-u_{25}+u_{03}+u_{04})\Gamma(1-c_0-u_{03}-u_{04})u_{24}!u_{25}!u_{03}!u_{04}!}
\end{align}

\begin{align}
\varphi_5(z;c)&=z_{23}^{-c_2}z_{13}^{c_2-c_3}z_{14}^{c_0-c_4-c_5}z_{04}^{c_5-c_0}z_{05}^{-c_5}\nonumber\\
&\sum_{u_{24},u_{25},u_{15},u_{03}\geq 0}\frac{1}{\Gamma(1-c_2-u_{24}-u_{25})\Gamma(1+c_2-c_3+u_{24}+u_{25}-u_{03})}\nonumber\\
&\frac{1}{\Gamma(1+c_0-c_4-c_5-u_{24}-u_{25}-u_{15}+u_{03})\Gamma(1+c_5-c_0+u_{25}+u_{15}-u_{03})\Gamma(1-c_5-u_{25}-u_{15})}\nonumber\\
&\frac{(z_{23}^{-1}z_{13}z_{14}^{-1}z_{24})^{u_{24}}(z_{23}^{-1}z_{13}z_{14}^{-1}z_{04}z_{05}^{-1}z_{25})^{u_{25}}(z_{14}^{-1}z_{04}z_{05}^{-1}z_{15})^{u_{15}}(z_{13}^{-1}z_{14}z_{04}^{-1}z_{03})^{u_{03}}}{u_{24}!u_{25}!u_{15}!u_{03}!}
\end{align}

\begin{align}
\varphi_6(z;c)&=z_{23}^{-c_2}z_{13}^{-c_1}z_{03}^{-c_0+c_4+c_5}z_{04}^{-c_4}z_{05}^{-c_5}\nonumber\\
&\sum_{u_{24},u_{25},u_{14},u_{15}\geq 0}\frac{1}{\Gamma(1-c_2-u_{24}-u_{25})\Gamma(1-c_1-u_{14}-u_{15})\Gamma(1-c_0+c_4+c_5+u_{24}+u_{25}+u_{14}+u_{15})}\nonumber\\
&\frac{(z_{23}^{-1}z_{03}z_{04}^{-1}z_{24})^{u_{24}}(z_{23}^{-1}z_{03}z_{05}^{-1}z_{25})^{u_{25}}(z_{13}^{-1}z_{03}z_{04}^{-1}z_{14})^{u_{14}}(z_{13}^{-1}z_{03}z_{05}^{-1}z_{15})^{u_{15}}}{\Gamma(1-c_4-u_{24}-u_{14})\Gamma(1-c_5-u_{25}-u_{15})u_{24}!u_{25}!u_{14}!u_{15}!}.
\end{align}

\noindent
Note that if we substitute 
\begin{equation}
\begin{pmatrix}
z_{03}&z_{04}&z_{05}\\
z_{13}&z_{14}&z_{15}\\
z_{23}&z_{24}&z_{25}
\end{pmatrix}
=
\begin{pmatrix}
1&1&1\\
1&\zeta_1&\zeta_1\zeta_2\\
1&\zeta_1\zeta_3&\zeta_1\zeta_2\zeta_3\zeta_{4}
\end{pmatrix},
\end{equation}

\noindent
all the Laurent series $\varphi_{i}(z;c)$ above become power series, i. e., they do not contain any negative power in $\zeta_1,\dots,\zeta_4$.

\end{exa}

\section{Quadratic relations for a confluence of Aomoto-Gelfand system}\label{E(21...1)}

In this section, we consider the generalized confluent hypergeometric system $M_{(2,1^{n-1})}$ in the sense of \cite{KimuraHaraokaTakano}. The system $M_{(2,1^{n-1})}$ can be obtained by taking a ``confluence'' of the system $E(k+1,n+1)$. We use the same notation as \S\ref{QuadraticRelationsForGrassman}. The general solution of the system $M_{(2,1^{n-1})}$ has an integral representation of the form

\begin{equation}
f_\Gamma(z)=\int_{\Gamma}\prod_{j=0}^{n-1}l_j(x;z)^{\alpha_j}e^{\frac{l_n(x;z)}{l_0(x;z)}}\omega(x)
\end{equation}
with $z=(z_{ij})_{\substack{i=0,\dots,k\\ j=0,\dots,n}}\in Z_{k+1,n+1}.$ The  parameters $\alpha_0,\dots,\alpha_{n-1}\in\C$ are subject to the constraint  $\alpha_0+\dots+\alpha_{n-1}=-(k+1)$. The system $M_{(2,1^{n-1})}$ in our setting is given by

\begin{equation}
M_{(2,1^{n-1})}:
\begin{cases}
\displaystyle\sum_{i=0}^k\left( z_{i0}\frac{\partial}{\partial z_{i0}}+z_{in}\frac{\partial}{\partial z_{in}}\right)f=\alpha_0 f\\
\displaystyle\sum_{i=0}^kz_{ij}\frac{\partial f}{\partial z_{ij}}=\alpha_jf&(j=1,\dots,n-1)\\
\displaystyle\sum_{i=0}^kz_{i0}\frac{\partial f}{\partial z_{in}}=f\\
\displaystyle\sum_{j=0}^nz_{ij}\frac{\partial f}{\partial z_{pj}}=-\delta_{ip}f&(i,p=0,1,\dots,k)\\
\frac{\partial^2 f}{\partial z_{ij}\partial z_{pq}}=\frac{\partial^2 f}{\partial z_{pj}\partial z_{iq}}& (i,p=0,1,\dots,k,\quad j,q=0,\dots, n).
\end{cases}
\end{equation}

\noindent
If we take a restriction to $z=
\begin{pmatrix}
1&         & &z_{0k+1}&\cdots&z_{0n}\\
 &\ddots& & \vdots &\ddots&\vdots\\
  &         &1&z_{kk+1}&\cdots&z_{kn}
\end{pmatrix}
$ and $x_0=1$, our integral $f_\Gamma(z)$ becomes 
\begin{equation}
f_\Gamma(z)=\int_\Gamma e^{z_{0n}+z_{1n}x_1+\cdots+z_{kn}x_k}\prod_{j=k+1}^{n-1}l_j(x;z)^{\alpha_j}x_1^{\alpha_1}\dots x_k^{\alpha_k}dx.
\end{equation}
If we put $c={}^t(
\alpha_0+1,\dots,\alpha_k+1,-\alpha_{k+1},\dots,-\alpha_{n-1})$, and put ${\bf a}(i,j)={\bf e}(i)+{\bf e}(j)$ $(i=0,1,\dots,k,j=k+1,\dots,n-1)$ and ${\bf a}(i,n)=-{\bf e}(0)+{\bf e}(i)$ $(i=1,\dots,k)$, $g_{\Gamma}(z)=e^{-z_{0n}}f_\Gamma(z)$ is a solution of $M_A(c)$ with $A=({\bf a}(i,j))_{(i,j)\in\llbracket 0,k\rrbracket\times\llbracket k+1,n\rrbracket\setminus\{ (0,n)\}}$. As in \S\ref{QuadraticRelationsForGrassman}, we also put $\tilde{\bf a}(i,j)=-{\bf e}(i)+{\bf e}(j)$ $(i=0,1,\dots,k,j=k+1,\dots,n-1)$, $\tilde{\bf a}(i,j)={\bf e}(0)-{\bf e}(i)$ and $\tilde{A}=(\tilde{\bf a}(i,j))_{(i,j)\in\llbracket 0,k\rrbracket\times\llbracket k+1,n\rrbracket\setminus\{ (0,n)\}}$. Since the variable $z_{0n}$ does not appear in $g_{\Gamma}(z)$, we write $z=
\begin{pmatrix}
1&&         & &z_{0k+1}&\cdots&z_{0n-1}&*\\
 &1&        & &z_{1k+1}&\cdots&z_{1n-1}&z_{1n}\\
 &&\ddots& & \vdots &\ddots&\vdots&\vdots\\
  &&         &1&z_{kk+1}&\cdots&z_{kn-1}&z_{kn}
\end{pmatrix}
$ by abuse of notation.

We consider a ``confluence'' of the staircase triangulation. Namely, for a ladder $I_0\subset\llbracket 1,\dots,k\rrbracket\times\llbracket k+1,\dots,n\rrbracket$, we associate a simplex $I=I_0\setminus\{ (0,n)\}$ of $A$. The collection of all such simplices $T=\{ I=I_0\setminus\{ (0,n)\}\mid I_0:\text{ladder}\}$ forms a convergent regular triangulation. Since $\vol_{\Z A}(\Delta_A)=\begin{pmatrix} n-1\\ k\end{pmatrix}$ and $|T|=\begin{pmatrix} n-1\\ k\end{pmatrix}$, $T$ is unimodular. By abuse of terminology, we also call $I=I_0\setminus\{ (0,n)\}$ a ladder. For any simplex $I\in T$, we consider the equation $Av^I=-c$ such that $v^I_{ij}=0\quad ((i,j)\notin I)$. Defining $\tilde{c}_l=
\begin{cases}
c_l&(l=0,\dots, k)\\
-c_l&(l=k+1,\dots,n-1),
\end{cases}
$ it is equivalent to the system $\tilde{A}v^I=\tilde{c}$. Let us introduce the dual basis $\phi(l)$ $(l=0,\dots,n-1)$ to ${\bf e}(l)$. For each $(i,j)\in G_{I_0}$, we put
\begin{equation}
\varphi(ij)=\sum_{l\in V(C_j(i,j))\setminus\{ n\}}\phi(l).
\end{equation}

\begin{prop}
For $(i,j),(i^\prime,j^\prime)\in I$, we have 
\begin{equation}
\langle \varphi(ij),\tilde{\bf a}(i^\prime,j^\prime)\rangle=
\begin{cases}
1&((i,j)=(i^\prime,j^\prime))\\
0&(otherwise).
\end{cases}
\end{equation}
\end{prop}

\begin{proof}
We set $S=\llbracket 0,k\rrbracket\times\llbracket k+1,n\rrbracket$, $S^\prime=S\setminus\{ (0,n)\}$. Let $A_S:\Z^{S\times 1}\rightarrow \Z^{\llbracket 0,n\rrbracket\times 1}$ be the $\Z$-linear map defined by $X=(X_{ij})\mapsto \sum_{(i,j)\in S}X_{ij}(-{\bf e}(i)+{\bf e}(j))$ and let $A_{S^\prime}:\Z^{S^\prime\times 1}\rightarrow \Z^{\llbracket 0,n-1\rrbracket\times 1}$ be the $\Z$-linear map defined by $X\mapsto \sum_{(i,j)\in S^\prime}X_{ij}\tilde{\bf a}(i,j)$. Let $\pi_1:\Z^{S\times 1}\rightarrow\Z^{S^\prime \times 1}$ be the canonical projection and $\pi_2:\Z^{\llbracket 0,n\rrbracket}\rightarrow \Z^{\llbracket 0,n-1\rrbracket}$ be the $\Z$-linear map given by $(x_0,\dots,x_n)\mapsto (x_0+x_n,x_1,\dots,x_{n-1})$. It is easy to check the identity $A_{S^\prime}\circ \pi_1=\pi_2\circ A_S.$ Let $A_{I_0}$ (resp. $A_I$) denote the restriction of the map $A_S$ (resp. $A_{S^\prime}$) to the submodule $M_{I_0}=\{ X\in\Z^{S\times 1}\mid X_{ij}=0\text{ if }(i,j)\notin I_0\}$ (resp. $M_I=\{ X\in\Z^{S^\prime\times 1}\mid X_{ij}=0\text{ if }(i,j)\notin I\}$). Then, we have the commutative diagram 
\begin{equation}
\xymatrix{
 \{ X\in M_{I_0}\mid X_{0n}=0\} \ar[r]^-{A_{I_0}} \ar[d]^{\pi_1}&\{ x\in\Z^{\llbracket 0,n\rrbracket\times 1}\mid x_0=x_0+\dots+x_n=0\}\ar[d]^{\pi_2}\\
 M_I \ar[r]^-{A_I}                             &\{ x\in\Z^{\llbracket 0,n-1\rrbracket\times 1}\mid x_0+\dots+x_{n-1}=0\}.
}
\end{equation}
Since $A_{I_0}$ and the two vertical morphisms are isomorphisms, we see that $A_I^{-1}$ is given by $\pi_1\circ A_{I_0}^{-1}\circ\pi_2^{-1}$. In view of \cref{prop:GraphicalInverse}, we obtain the proposition.
\end{proof}

\noindent
Therefore, we obtain a
\begin{cor}
Under the notation above, one has
\begin{equation}
v^I_{ij}=\sum_{l\in V(C_j(i,j))}\tilde{c}_l.
\end{equation}
\end{cor}

\noindent
Substitution of this formula to $\Gamma$-series yields the formula
\begin{equation}\label{eqn:10.7}
\varphi_{v^I}(z)=z_I^{v^I}\sum_{u_{\bar{I}}\in\Z^{\bar{I}}_{\geq 0}}\frac{\left(z_I^{-\langle \varphi(I),\tilde{A}_{\bar{I}}\rangle}z_{\bar{I}}\right)^{u_{\bar{I}}}}{\displaystyle\prod_{(i,j)\in I}\Gamma(1+v_{ij}^I-\langle\varphi(ij),\tilde{A}_{\bar{I}}u_{\bar{I}}\rangle)u_{\bar{I}}!}.
\end{equation}
Since the series (\ref{eqn:10.7}) is defined by means of a ladder $I$ and a parameter $\alpha$, it is denoted by $g_I(z;\alpha)$. Similarly, setting $\bar{I}_{irr}=\bar{I}\cap\{ (1,n),\dots,(k,n)\}$, we obtain the formula for the dual series

\begin{equation}
\varphi_{v^I}^\vee(z)=z_I^{v^I}\sum_{u_{\bar{I}}\in\Z^{\bar{I}}_{\geq 0}}\frac{(-1)^{|u_{\bar{I}_{irr}}|+\sum_{(i,n)\in I}\langle \varphi(in),\tilde{A}u_{\bar{I}}\rangle}\left(z_I^{-\langle \varphi(I),\tilde{A}_{\bar{I}}\rangle}z_{\bar{I}}\right)^{u_{\bar{I}}}}{\displaystyle\prod_{(i,j)\in I}\Gamma(1+v_{ij}^I-\langle\varphi(ij),\tilde{A}_{\bar{I}}u_{\bar{I}}\rangle)u_{\bar{I}}!}.
\end{equation}
This series is denoted by $g_I^\vee(z;\alpha)$.

We consider the de Rham cohomology group $\Homo^k_{dR}\left( \PP^k_x\setminus\bigcup_{j=0}^{n-1}\{ l_j(x;z)\},\nabla_z\right)$ with $\nabla_z=d_x+\sum_{j=0}^{n-1}d_x\log l_j(x;z)\wedge+d_x\left( \frac{l_n(x;z)}{l_0(x;z)}\right)\wedge$. For any subset $J=\{ j_1,\dots,j_k\}\subset\{1,\dots,n-1\}$ with cardinality $k$, we write $z_J$ for the submatrix of $z$ consisting of column vectors indexed by $\{ 0\}\cup J$. We always assume $j_1<\dots<j_k$. We put
\begin{equation}
\eta_J(z;x)=d_x\log\left(\frac{l_{j_1}(x;z)}{l_{0}(x;z)}\right)\wedge\dots\wedge d_x\log\left(\frac{l_{j_k}(x;z)}{l_{0}(x;z)}\right).
\end{equation}
The set $\{ \eta_J(x;z)\}_{J}$ forms a basis of the algebraic de Rham cohomology group $\Homo_{\rm dR}^k\left( \mathbb{P}^k_x\setminus\bigcup_{j=0}^{n-1}\{ l_j(x;z)=0\};\nabla_z\right)$ (\cite{AomotoKitaOrlikTerao}, \cite{KimuraBasisOfCohomology}). Now we are going to derive a quadratic relation for $g_I(z;\alpha)$. We take any pair of  subsets $J,J^\prime\subset\{1,\dots,n-1\}$ with cardinality $k$. Let us put $J_a=J\cap\{ 1,\dots,k\}$, $J_a^\prime=J^\prime\cap\{ 1,\dots,k\}$, $J_b=J\cap\{ k+1,\dots,n-1\}$, and $J_b^\prime=J^\prime\cap\{ k+1,\dots,n-1\}$. We write ${\bf 1}_{J_a}$ (resp. ${\bf 1}_{J_b}$) for the vector $\sum_{j\in J_a}{\bf e}(j)$ (resp. $\sum_{j\in J_b}{\bf e}(j)$). If we write $\alpha$ as $\displaystyle\sum_{j=0}^{n-1}\alpha_j{\bf e}(j)$, we also put $\alpha_a=\displaystyle\sum_{j=1}^k\alpha_j{\bf e}(j)$ and $\alpha_{b}=\displaystyle\sum_{j=k+1}^{n-1}\alpha_j{\bf e}(j)$. By the same argument as \cite{MatsumotoIntersection}, we obtain a formula of the cohomology intersection numbers $\langle \eta_{J}(x;z),\eta_{J^\prime}(x;z)\rangle_{ch}$.
\begin{prop}\label{prop:E(21...1)c.i.n.}
Under the assumption $\tilde{c_j}\notin\Z$, one has 
\begin{equation}\label{MatsumotoFormula2}
\frac{\langle \eta_{J}(x;z),\eta_{J^\prime}(x;z)\rangle_{ch}}{(2\pi\ii)^k}=
\begin{cases}
\frac{1}{\prod_{j\in J}\tilde{c}_j}& (J=J^\prime)\\
0&(J\neq J^\prime)
\end{cases}.
\end{equation}
\end{prop}

\begin{proof}
The result immediately follows from the computation of \cite{MatsumotoIntersection} after a minor modification in view of \cite{MajimaMatsumotoTakayama}. For readers' convenience, we give necessary modifications. Let $X$ be the projective variety obtained by blowing-up $\mathbb{P}^k$ along the 2-codimensional linear subvariety $\{ l_0(x;z)=l_n(x;z)=0\}$. It is readily seen that there is a well defined regular map $f:X\rightarrow\mathbb{P}^1$ whose restriction to the subspace $\{ l_0(x;z)\neq 0\}$ is identical with $\frac{l_n(x;z)}{l_0(x;z)}$. We write $L_j$ for the proper transformation of the divisor $\{ l_j(x;z)=0\}\subset\mathbb{P}^k$ ($j=0,\dots,n-1$) and write $L_n$ for the exceptional divisor of the blowing-up. For any pair of $k$-tuples $P=(p_1,\dots,p_k)$ and $Q=(q_1,\dots,q_k)$ $(0\leq p_1<\dots< p_k\leq n-1, \ \ 0\leq q_1<\dots<q_k\leq n-1)$, we set 
\begin{equation}
\delta(P,Q)=
\begin{cases}
1&(\text{if }P=Q)\\
0&(\text{if }P\neq Q)
\end{cases}.
\end{equation}
For each $j=0,\dots,n$, we write $V_j$ for a tubular neighbourhood of $L_j$ and write $h_j$ for a smooth function on $X$ which is equal to $1$ near $L_j$ and $0$ outside $V_j$. For each multi-index $P=(p_1,\dots,p_k)$, we denote by $w_P=(w_1,\dots,w_k)$ the local coordinate around $\displaystyle\bigcap_{i=1}^kL_{p_i}$ so that $L_{p_i}=\{ w_i=0\}$ and we set $D_P=\displaystyle\bigcap_{i=1}^kV_{p_i}$. By solving the equation $\nabla^\vee_z \xi=\eta_{J^\prime}(x;z)$ locally as in \cite[Lemma 5.1, 5.2 and Lemma 6.1]{MatsumotoIntersection}, we can find a smooth $k$-form $\psi_{J^\prime}(x)$ on $X$ such that the following properties 1-5 hold.
\begin{enumerate}
\item $\psi_{J^\prime}$ has a compact support in $X\setminus \bigcup_{j=0}^{n}L_j$ and cohomologous to $\eta_{J^\prime}$ in $\Homo_{\rm dR}^k\left( \mathbb{P}^k_x\setminus\bigcup_{j=0}^{n-1}\{ l_j(x;z)=0\};\nabla_z^\vee\right)$.
\item The $k$-form $\psi_{J^\prime}-\eta_{J^\prime}$ vanishes outside a small tubular neighbourhood of the divisor $\bigcup_{j=0}^{n}L_j$.
\item On each $V_j\setminus D_P$, $\psi_{J^\prime}$ is of the form $\psi_{J^\prime}=\sum_{i}\xi_i\wedge\eta_i$ where $\xi_i$ are smooth forms and $\eta_i$ are rational differential forms of degree greater or equal to one. 
\item For each multi-index $P=(p_1,\dots,p_k)$ such that $p_1=0$, we have $\psi_{J^\prime}=f_P(w_P)dh_{p_1}\wedge\dots\wedge dh_{p_k}$ where $f_P(w_P)$ is a holomorphic function such that $f_P(0,w_2,\dots,w_k)\equiv 0$.
\item For each multi-index $P=(p_1,\dots,p_k)$ such that $0\notin P$, we have $\psi_{J^\prime}=f_P(w_P)dh_{p_1}\wedge\dots\wedge dh_{p_k}$ where $f_P(w_P)$ is a holomorphic function with $f_P({\bf O})=(-1)^{k}\frac{\delta(P,J^\prime)}{\prod_{j=1}^k\tilde{c}_{p_j}}$.
\end{enumerate}
The property 4 is due to the fact that the term $d\left( \frac{l_n(x;z)}{l_0(x;z)}\right)\wedge$ in $\nabla^\vee_z=d-\sum_{j=1}^{n-1}\tilde{c}_jd\log l_j(x;z)\wedge-d\left( \frac{l_n(x;z)}{l_0(x;z)}\right)\wedge$ gives rise to a non-logarithmic term $-\frac{dw_1}{w_1^2}\wedge$. If we write $\iota_z$ for the inverse of the canonical isomorphism 
\begin{equation}
{\rm can}:\Homo_{r.d.}^k\left( \mathbb{P}^k_x\setminus\displaystyle\bigcup_{j=0}^{n-1}\{ l_j(x;z)=0\};\nabla_z^\vee\right)\tilde{\rightarrow}\Homo_{\rm dR}^k\left( \mathbb{P}^k_x\setminus\displaystyle\bigcup_{j=0}^{n-1}\{ l_j(x;z)=0\};\nabla_z^\vee\right),
\end{equation}
it is clear that $\psi_{J^\prime}$ is a representative of $\iota_z\eta_{J^\prime}$ in $\Homo_{r.d.}^k\left( \mathbb{P}^k_x\setminus\bigcup_{j=0}^{n-1}\{ l_j(x;z)=0\};\nabla_z^\vee\right)$. Therefore, by properties 3,4 and 5 in view of the proof of \cite[Theorem 4.1]{MajimaMatsumotoTakayama}, we obtain
\begin{align}
\langle \eta_{J}(x;z),\eta_{J^\prime}(x;z)\rangle_{ch}&=\int \eta_{J}(x;z)\wedge\psi_{J^\prime}\\
&=\sum_{P}\int_{D_P}\eta_{J}(x;z)\wedge\psi_{J^\prime}\\
&=(2\pi\ii)^k\sum_{P}\underset{w_k=0}{\rm Res}\left( \underset{w_{k-1}=0}{\rm Res}\left( \cdots \underset{w_1=0}{\rm Res}\left((-1)^kf_{P}\eta_{J}(x;z)\right)\cdots\right)\right)\\
&=(2\pi\ii)^k\frac{\delta(J,J^\prime)}{\prod_{j\in J}\tilde{c}_{j}}.
\end{align}

\end{proof}

In sum, we obtain the general quadratic relation of a confluence of Aomoto-Gelfand hypergeometric functions:

\begin{thm}\label{thm:QuadraticRelationsForAomotoGelfand2}
Under the notation as above, for any $z\in U_T$, we have an identity
\begin{align}
&(-1)^{|J_b|+|J^\prime_b|+k}\alpha_{k+1}\dots\alpha_n(-\alpha_b+{\bf 1}_{J_b})_{-{\bf 1}_{J_b}}(\alpha_b+{\bf 1}_{J^\prime_b})_{-{\bf 1}_{J_b^\prime}}\times\nonumber\\
&\sum_{I:\text{ladder}}\frac{\pi^{n-1}}{\displaystyle\prod_{(i,j)\in I}\sin\pi v_{ij}^I}g_I(z;\alpha+{\bf 1}_{\llbracket 0,k\rrbracket}-{\bf 1}_J)g_I^\vee(z;\alpha-{\bf 1}_{\llbracket 0,k\rrbracket}+{\bf 1}_{J^\prime})\nonumber\\
=&
\det(z_J)^{-1}\det(z_{J^\prime})^{-1}\frac{\langle \eta_{J}(x;z),\eta_{J^\prime}(x;z)\rangle_{ch}}{(2\pi\ii)^k}.
\end{align}
Here, the right-hand side is explicitly determined by (\ref{MatsumotoFormula2}).
\end{thm}

\begin{exa}{\bf (Kummer's hypergeometric series)}

The simplest case is $k=1$ and $n=3$. This case is known as Kummer's hypergeometric equation. By computing the cohomology intersection number $\langle \frac{dx}{x},\frac{dx}{x}\rangle_{ch}$, we obtain a quadratic relation
We have a relation 
\begin{equation}
(\gamma-\alpha-1){}_1F_1\left(\substack{\alpha\\ \gamma};z\right){}_1F_1\left(\substack{-\alpha\\ 2-\gamma};-z\right)+\alpha{}_1F_1\left(\substack{1+\alpha-\gamma\\ 2-\gamma};z\right){}_1F_1\left(\substack{\gamma-\alpha-1\\ \gamma};-z\right)=\gamma-1,
\end{equation}
where the series ${}_1F_1\left(\substack{\alpha\\ \gamma};z\right)$ is Kummer's hypergeometric series
\begin{equation}
{}_1F_1\left(\substack{\alpha\\ \gamma};z\right)=\sum_{n=0}^\infty\frac{(\alpha)_n}{(\gamma)_nn!}z^n.
\end{equation}
This identity implies a series of combinatorial identities
\begin{equation}
(\gamma-\alpha-1)\sum_{l+m=n}(-1)^m\frac{(\alpha)_l(-\alpha)_{m}}{(\gamma)_l(1)_l(2-\gamma)_m(1)_m}
+\alpha\sum_{l+m=n}(-1)^m\frac{(1+\alpha-\gamma)_l(\gamma-\alpha-1)_m}{(2-\gamma)_l(1)_l(\gamma)_m(1)_m}=0,
\end{equation}
where $n$ is a positive integer.
\end{exa}

\begin{exa}{\bf (A confluence of $E(3,6)$)}

This is a confluence of \cref{exa:AomotoGelfand}. The integral in question takes the form $f_\Gamma(z)=\int_\Gamma\prod_{j=3}^4(z_{0j}+z_{1j}x_1+z_{2j}x_2)^{-c_j}e^{z_{1j}x_1+z_{2j}x_2}x_1^{c_1}x_2^{c_2}\frac{dx_1\wedge dx_2}{x_1x_2}$. The quadratic relation with respect to the cohomology intersection number $\langle\frac{dx_1\wedge dx_2}{x_1x_2},\frac{dx_1\wedge dx_2}{x_1x_2}\rangle_{ch}$ is given by
\begin{equation}
c_1c_2c_3c_4\sum_{i=1}^6\frac{\pi^4}{\sin\pi(v_i)}\varphi_i(z;c)\varphi_i^\vee(z;c)=1,
\end{equation}
where $v_i$ are given by
\begin{align}
v_1&={}^t(-c_3,-c_4,c_0+c_1,-c_1)\\
v_2&={}^t(-c_3,-c_2+c_3,-c_0-c_1,c_0)\\
v_3&={}^t(-c_3,-c_2+c_3,-c_1,-c_0)\\
v_4&={}^t(-c_2,c_2-c_3,-c_4,c_0)\\
v_5&={}^t(-c_2,c_2-c_3,c_0-c_4,-c_0)\\
v_6&={}^t(-c_2,-c_1,-c_0+c_4,-c_4).
\end{align}
Note that we have a relation $c_0+c_1+c_2-c_3-c_4=0$. The series $\varphi_i(z;c)$ are given by the following series.

\begin{align}
\varphi_1(z;c)&=z_{23}^{-c_3}z_{24}^{-c_4}z_{25}^{c_0+c_1}z_{15}^{-c_1}\nonumber\\
 &\sum_{u_{13},u_{14},u_{03},u_{04}\geq 0}\frac{1}{\Gamma(1-c_3-u_{13}-u_{03})\Gamma(1-c_4-u_{14}-u_{04})\Gamma(1+c_0+c_1+u_{13}+u_{14}+u_{03}+u_{04})}\nonumber\\
 &\frac{(z_{23}^{-1}z_{25}z_{15}^{-1}z_{13})^{u_{13}}(z_{24}^{-1}z_{25}z_{15}^{-1}z_{14})^{u_{14}}(z_{23}^{-1}z_{25}z_{03})^{u_{03}}(z_{24}^{-1}z_{25}z_{04})^{u_{04}}}{\Gamma(1-c_1-u_{13}-u_{14})u_{13}!u_{14}!u_{03}!u_{04}!}
\end{align}

\begin{align}
\varphi_2(z;c)&=z_{23}^{-c_3}z_{24}^{-c_2+c_3}z_{14}^{-c_0-c_1}z_{15}^{c_0}\nonumber\\
&\sum_{u_{25},u_{13},u_{03},u_{04}\geq 0}\frac{1}{\Gamma(1-c_3-u_{13}-u_{03})\Gamma(1-c_2+c_3-u_{25}+u_{13}+u_{03})}\nonumber\\
&\frac{1}{\Gamma(1-c_0-c_1+u_{25}-u_{13}-u_{03}-u_{04})\Gamma(1+c_0-u_{25}+u_{03}+u_{04})}\nonumber\\
&\frac{(z_{24}^{-1}z_{14}z_{15}^{-1}z_{25})^{u_{25}}(z_{23}^{-1}z_{24}z_{14}^{-1}z_{13})^{u_{13}}(z_{23}^{-1}z_{24}z_{14}^{-1}z_{15}z_{03})^{u_{03}}(z_{14}^{-1}z_{15}z_{04})^{u_{04}}}{u_{25}!u_{13}!u_{03}!u_{04}!}
\end{align}

\begin{align}
\varphi_3(z;c)&=z_{23}^{-c_3}z_{24}^{-c_2+c_3}z_{14}^{-c_1}z_{04}^{-c_0}\nonumber\\
&\sum_{u_{25},u_{15},u_{13},u_{03}\geq 0}\frac{1}{\Gamma(1-c_3-u_{13}-u_{03})\Gamma(1-c_2+c_3-u_{25}+u_{13}+u_{03})\Gamma(1-c_1-u_{15}-u_{13})}\nonumber\\
&\frac{(z_{24}^{-1}z_{04}z_{25})^{u_{25}}(z_{14}^{-1}z_{04}z_{15})^{u_{15}}(z_{23}^{-1}z_{24}z_{14}^{-1}z_{13})^{u_{13}}(z_{23}^{-1}z_{24}z_{04}^{-1}z_{03})^{u_{03}}}{\Gamma(1-c_0+u_{25}+u_{15}-u_{03})u_{25}!u_{15}!u_{13}!u_{03}!}
\end{align}

\begin{align}
\varphi_4(z;c)&=z_{23}^{-c_2}z_{13}^{c_2-c_3}z_{14}^{-c_4}z_{15}^{c_0}\nonumber\\
&\sum_{u_{24},u_{25},u_{03},u_{04}\geq 0}\frac{1}{\Gamma(1-c_2-u_{24}-u_{25})\Gamma(1+c_2-c_3+u_{24}+u_{25}-u_{03})\Gamma(1-c_4-u_{24}-u_{04})}\nonumber\\
&\frac{(z_{23}^{-1}z_{13}z_{14}^{-1}z_{24})^{u_{24}}(z_{23}^{-1}z_{13}z_{15}^{-1}z_{25})^{u_{25}}(z_{13}^{-1}z_{15}z_{03})^{u_{03}}(z_{14}^{-1}z_{15}z_{04})^{u_{04}}}{\Gamma(1+c_0-u_{25}+u_{03}+u_{04})u_{24}!u_{25}!u_{03}!u_{04}!}
\end{align}

\begin{align}
\varphi_5(z;c)&=z_{23}^{-c_2}z_{13}^{c_2-c_3}z_{14}^{c_0-c_4}z_{04}^{-c_0}\nonumber\\
&\sum_{u_{24},u_{23},u_{15},u_{03}\geq 0}\frac{1}{\Gamma(1-c_2-u_{24}-u_{25})\Gamma(1+c_2-c_3+u_{24}+u_{25}-u_{03})}\nonumber\\
&\frac{1}{\Gamma(1+c_0-c_4-u_{24}-u_{25}-u_{15}+u_{03})\Gamma(1-c_0+u_{25}+u_{15}-u_{03})}\nonumber\\
&\frac{(z_{23}^{-1}z_{13}z_{14}^{-1}z_{24})^{u_{24}}(z_{23}^{-1}z_{13}z_{14}^{-1}z_{04}z_{25})^{u_{25}}(z_{14}^{-1}z_{04}z_{15})^{u_{15}}(z_{13}^{-1}z_{14}z_{04}^{-1}z_{03})^{u_{03}}}{u_{24}!u_{25}!u_{15}!u_{03}!}
\end{align}

\begin{align}
\varphi_6(z;c)&=z_{23}^{-c_2}z_{13}^{-c_1}z_{03}^{-c_0+c_4}z_{04}^{-c_4}\nonumber\\
&\sum_{u_{24},u_{25},u_{14},u_{15}\geq 0}\frac{1}{\Gamma(1-c_2-u_{24}-u_{25})\Gamma(1-c_1-u_{14}-u_{15})\Gamma(1-c_0+c_4+u_{24}+u_{25}+u_{14}+u_{15})}\nonumber\\
&\frac{(z_{23}^{-1}z_{03}z_{04}^{-1}z_{24})^{u_{24}}(z_{23}^{-1}z_{03}z_{25})^{u_{25}}(z_{13}^{-1}z_{03}z_{04}^{-1}z_{14})^{u_{14}}(z_{13}^{-1}z_{03}z_{15})^{u_{15}}}{\Gamma(1-c_4-u_{24}-u_{14})u_{24}!u_{25}!u_{14}!u_{15}!}
\end{align}

The series $\varphi_i^\vee(z;c)$ is obtained from $\varphi_i(z;c)$ by replacing $c_i$ by $-c_i$ and $z_{15},z_{25}$ by $-z_{15},-z_{25}$ in the summand. Note that if we substitute 
\begin{equation}
\begin{pmatrix}
z_{03}&z_{04}&*\\
z_{13}&z_{14}&z_{15}\\
z_{23}&z_{24}&z_{25}
\end{pmatrix}
=
\begin{pmatrix}
1&1&*\\
1&\zeta_1&\zeta_1\zeta_2\\
1&\zeta_1\zeta_3&\zeta_1\zeta_2\zeta_3\zeta_{4}
\end{pmatrix},
\end{equation}

\noindent
all the Laurent series $\varphi_{i}(z;c)$ and $\varphi_{i}^\vee(z;c)$ above become power series, i. e., they do not contain any negative power.

\end{exa}

\section*{Appendix 1: A lemma on holonomic dual}
In this appendix, we prove \cref{lem:Duality}. Let $\Delta_X:X\hookrightarrow X\times X$ be the diagonal embedding. We also denote its image by $\Delta_X$. Since ${\rm Ch}(M\boxtimes N)={\rm Ch}(M)\times {\rm Ch}(N)$ and ${\rm Ch}(M)\cap {\rm Ch}(N)\subset T^*_XX$ by the assumption of \cref{lem:Duality}, we obtain the inclusion 
\begin{equation}
T_{\Delta_X}(X\times X)\cap {\rm Ch}(M\boxtimes N)=\{ (x,\xi;x,\xi)\in T^*X\times T^*X\mid (x,\xi)\in {\rm Ch}(M)\cap {\rm Ch}(N)\}\subset T^*_{X\times X}X\times X.
\end{equation}
Therefore, $M\boxtimes N$ is non-characteristic with respect to the morphism $\Delta_X$. By \cite[Theorem 2.7.1.]{HTT}, we have a commutativity $\D_X (\mathbb{L}\Delta^*_X(M\boxtimes N))\simeq\mathbb{L}\Delta^*_X\D_{X\times X}(M\boxtimes N)$. Therefore, we have quasi-isomorphisms
\begin{align}
\D_X (M\overset{\D}{\otimes}N)&=\D_X (\mathbb{L}\Delta^*_X(M\boxtimes N))\\
                                              &\simeq\mathbb{L}\Delta^*_X(\D_XM\boxtimes \D_XN)\\
                                              &\simeq \D_XM\overset{\D}{\otimes} \D_XN.
\end{align}

\section*{Appendix 2: Proof of \cref{prop:PochhammerIntersection}}

We apply the twisted period relation (\ref{GeneralQuadraticRelation}) to $\Homo_n\left( X, \mathcal{L}\right)$, where $X=\C^n_x\setminus\{ x_1\cdots x_n(1-x_1-\dots-x_n)=0\}$ $\mathcal{L}=\C x_1^{\alpha_1}\cdots x_n^{\alpha_n}(1-x_1-\dots-x_n)^{\alpha_{n+1}}$. We take a basis $\frac{dx}{x}=\frac{dx_1\wedge \dots\wedge dx_n}{x_1\dots x_n}$ of twisted cohomology group $\Homo^n\left( X, \mathcal{L}\right)$ and of $\Homo^n\left( X, \mathcal{L}^\vee\right).$ By \cite{MatsumotoIntersection}, we have $\langle\frac{dx}{x},\frac{dx}{x}\rangle_{ch}=(2\pi\ii)^n\frac{\alpha_0+\dots+\alpha_n}{\alpha_0\dots\alpha_n}.$ On the other hand, we have
\begin{equation}
\int_{P_\tau}x_1^{\alpha_1}\cdots x_n^{\alpha_n}(1-x_1-\dots-x_n)^{\alpha_{n+1}}\frac{dx}{x}=\prod_{i=1}^{n+1}(1-e^{-2\pi\ii\alpha_i})\frac{\Gamma(\alpha_1)\dots\Gamma(\alpha_n)\Gamma(\alpha_{n+1}+1)}{\Gamma(1-\alpha_0)}
\end{equation}
and
\begin{equation}
\int_{\check{P}_\tau}x_1^{-\alpha_1}\cdots x_n^{-\alpha_n}(1-x_1-\dots-x_n)^{-\alpha_{n+1}}\frac{dx}{x}=\prod_{i=1}^{n+1}(1-e^{2\pi\ii\alpha_i})\frac{\Gamma(-\alpha_1)\dots\Gamma(-\alpha_n)\Gamma(1-\alpha_{n+1})}{\Gamma(1+\alpha_0)}.
\end{equation}

\noindent
Therefore, we obtain \cref{prop:PochhammerIntersection} in view of the twisted period relation (\ref{GeneralQuadraticRelation}).
\section*{Appendix 3: Construction of a lift of a Pochhammer cycle}

In this appendix, we summarize the construction of Pochhammer cycles following \cite[\S 6]{Beukers} and construct its lift by a covering map. 

We consider a hyperplane $H$ in $\C^{n+1}$ defined by $H=\{ t_0+\cdots+t_n=1\}$. Let $\varepsilon$ be a small real positive number. We consider a polytope $F$ in $\R^{n+1}$ defined by 
\begin{equation}
|x_{i_1}|+\cdots+|x_{i_k}|\leq 1-(n+1-k)\varepsilon
\end{equation}
for all $k=1,\dots, n+1$ and all $0\leq i_1 <i_2<\dots <i_k\leq n$. The faces of this polytope can be labeled by vectors $\mu\in\{0,\pm 1\}^n\setminus\{ 0\}^n$. We set $|\mu|=\displaystyle\sum_{i=0}^n|\mu_i|$. The face $F_\mu$ corresponding to $\mu$ is defined by
\begin{equation}
\mu_0 x_0+\mu_1 x_1+\dots+\mu_nx_n=1-(n+1-|\mu|)\varepsilon,\; \mu_jx_j\geq\varepsilon\text{ whenever }\mu_j\neq 0,\; |x_j|\leq\varepsilon\text{ whenever }\mu_j=0.
\end{equation}

\noindent
The number of faces of $F$ is $3^n-1$ and each $F_\mu$ is isomorphic to $\Delta_{|\mu|-1}\times {\rm I}^{n+1-|\mu|}$ where ${\rm I}$ is a closed interval. The vertices of $F$ are points with one coordinate $\pm(1-n\varepsilon)$ and all other coordinates $\pm \varepsilon$. Therefore, the number of vertices is $(n+1)2^{n+1}$. Define a continuous piecewise smooth map $P:\cup_\mu F_\mu\rightarrow H$ by
\begin{equation}
P(x_0,\dots ,x_n)=\frac{1}{\tilde{y}_0+\dots+\tilde{y}_n}(y_0,\dots,y_n)
\end{equation}
where

\begin{equation}\label{argument}
y_j=
\begin{cases}
x_j & (x_j\geq\varepsilon)\\
e^{-2\pi\ii}|x_j| & (x_j\leq-\varepsilon)\\
\varepsilon e^{-\pi\ii(1-\frac{x_j}{\varepsilon})} & (|x_j|\leq\varepsilon).
\end{cases} 
\end{equation}

\begin{equation}\label{argument2}
\tilde{y}_j=
\begin{cases}
|x_j| & (|x_j|\geq\varepsilon)\\
\varepsilon e^{-\pi\ii(1-\frac{x_j}{\varepsilon})} & (|x_j|\leq\varepsilon).
\end{cases} 
\end{equation}

\noindent
Let us denote by $\pi:H\rightarrow \C^n$ be the projection $\pi(t_0,\dots,t_n)=(t_1,\dots,t_n).$ By definition, the image of the map $\pi\circ P$ is contained in the complement of a divisor $\{ 1=t_1+\dots+t_n\}$ in the torus $(\C^\times)^n\subset\C^n$.  On each face $F_\mu$, the branch of a multivalued function $t_1^{\beta_1-1}\dots t_n^{\beta_n-1}(1-t_1-\dots -t_n)^{\beta_{0}-1}$ on $\pi\circ P (F_\mu)$ is defined by

\begin{equation}\label{TheBranch}
t_1^{\beta_1-1}\dots t_n^{\beta_n-1}(1-t_1-\cdots -t_n)^{\beta_{0}-1}=\prod_{\mu_j\neq 0}|x_j|^{\beta_j-1}e^{\pi\ii(\mu_j-1)\beta_j}\prod_{\mu_k=0}\varepsilon^{\beta_k-1}e^{\pi\ii(\frac{x_k}{\varepsilon}-1)(\beta_k-1)}.
\end{equation}
Thus, we can define a multi-dimensional Pochhammer cycle $P_n$ as a cycle with local system coefficients.

Now we consider a covering map $p:(\C^\times)^n_\tau\rightarrow (\C^\times)^n_t$ defined by $p(\tau)=\tau^A$ where $A=({\bf a}(1)|\dots|{\bf a}(n))$ is an invertible $n$ by $n$ matrix with integer entries. We put $\beta^\prime={}^t(\beta_1,\dots,\beta_n).$

\begin{prop}
There exists an element $[P_n^\prime]$ of $\Homo_n\left((\C^\times)^n_\tau\setminus\left\{ 1=\sum_{i=1}^n\tau^{{\bf a}(i)}\right\};\underline{\C}\left(1-\sum_{i=1}^n\tau^{{\bf a}(i)}\right)^{\beta_0}\tau^{A\beta^\prime}\right)$ such that the identity $p_*(P^\prime_n)=P_n$ holds. 
\end{prop}

\begin{proof}
Let us put $\pi\circ P(x)=(q_1(x),\dots, q_n(x)).$ Define a map $P^\prime:\cup_\mu F_\mu\rightarrow (\C^\times)^n_\tau\setminus\left\{ 1=\displaystyle\sum_{i=1}^n\tau^{{\bf a}(i)}\right\}$ by 

\begin{equation}
P^\prime(x)=(q_1(x),\dots, q_n(x))^{A^{-1}}.
\end{equation}
Note that this is a well-defined continuous map in view of (\ref{argument}) and (\ref{argument2}). The branch of a multivalued function $\left(1-\displaystyle\sum_{i=1}^n\tau^{{\bf a}(i)}\right)^{\beta_0}\tau^{A\beta^\prime}$ on the face $F_\mu$ is therefore defined by the formula 
\begin{equation}
\left(1-\sum_{i=1}^n\tau^{{\bf a}(i)}\right)^{\beta_0}\tau^{A\beta^\prime}=\prod_{\mu_j\neq 0}|x_j|^{\beta_j-1}e^{\pi\ii(\mu_j-1)\beta_j}\prod_{\mu_k=0}\varepsilon^{\beta_k-1}e^{\pi\ii(\frac{x_k}{\varepsilon}-1)(\beta_k-1)}.
\end{equation}
Thus, we can define a twisted cycle $P^\prime_n$. It is obvious from the construction that the identity $p_*(P^\prime_n)=P_n$ holds.

\end{proof}

Write $A=(A_1|\cdots|A_k)$, $A_l=({\bf a}^{(l)}(1)|\cdots|{\bf a}^{(l)}(n_l))$
One can easily generalize the result above to the following

\begin{prop}
Set $t=(t^{(1)},\dots,t^{(k)})$ and $\beta_i^{(l)}\in\C$ $(l=1,\dots,k, \; i=1,\dots,n_l)$. We put $\mathcal{L}=\displaystyle\prod_{l=1}^k\underline{\C}(1-\sum_{i=1}^{n_l}t^{(l)}_i)^{\beta^{(l)}_0}(t^{(l)}_1)^{\beta_{1}^{(l)}}\cdots (t^{(l)}_{n_l})^{\beta_{n_l}^{(l)}}$. Then, there exists an element $[P_n^\prime]$ of 
\newline
$\Homo_n\left(\prod_{l=1}^k\left((\C^\times)^{n_l}_{\tau^{(l)}}\setminus\{ 1=\sum_{i=1}^{n_l}\tau^{{\bf a}^{(l)}(i)}\}\right); p^{-1}\mathcal{L}\right)$ such that the identity $p_*(P^\prime_n)=\prod_{l=1}^kP^{(l)}_{n_l}$ holds.
\end{prop}

\section*{Acknowledgement}
The author would like to thank Yoshiaki Goto, Katsuhisa Mimachi, Kanami Park, Genki Shibukawa, Nobuki Takayama, and Yumiko Takei for valuable comments. The use of triangulations of a semi-analytic set was suggested by Takuro Mochizuki. The suthor would like to thank him. The author is grateful to Francisco-Jesus Castro-Jim\'enez, Maria-Cruz Fern\'andez-Fern\'andez, Michael Granger, and Susumu Tanab\'e for their interest. Finally, the author would like to thank Toshio Oshima and Hidetaka Sakai for their constant encouragement during the preparation of this paper.

This work is supported by JST CREST Grant Number JP19209317 and JSPS KAKENHI Grant Number 19K14554.

\bibliographystyle{alpha}
\bibliography{myreferences}

\end{document}